\documentclass[12pt,centertags]{amsart}
\usepackage{amsfonts}
\usepackage{amsmath,amstext,amsthm,amssymb,amscd}
\usepackage[mathscr]{eucal}
\usepackage{mathrsfs}
\usepackage{epsf}
\usepackage{tikz}
\usepackage{shorttoc}
\usepackage{yhmath}
\DeclareSymbolFont{largesymbol}{OMX}{yhex}{m}{n}
\DeclareMathAccent{\Widehat}{\mathord}{largesymbol}{"62}

\usetikzlibrary{matrix,arrows,decorations.pathmorphing}
\usepackage[a4paper,top=3cm,bottom=3cm,left=2cm,right=2cm]
{geometry}

\numberwithin{equation}{section}

\usepackage{color}
\usepackage{soul}

\newcommand{\field}[1]{\mathbb{#1}}
\newcommand{\Z}{\field{Z}}
\newcommand{\R}{\field{R}}
\newcommand{\C}{\field{C}}
\newcommand{\N}{\field{N}}

 \def\cC{\mathscr{C}}

\def\mC{\mathcal{C}}
\def\mA{\mathcal{A}}
\def\mB{\mathcal{B}}
\def\mD{\mathcal{D}}
\def\mE{\mathcal{E}}

\def\mH{\mathcal{H}}

\def\mL{\mathcal{L}}
\def\mM{\mathcal{M}}
\def\mN{\mathcal{N}}
\def\mO{\mathcal{O}}

\def\mQ{\mathcal{Q}}

\def\mU{\mathcal{U}}
\def\mV{\mathcal{V}}

\newcommand\mS{\mathcal{S}}

\def\kg{\mathfrak{g}}

\def\Re{{\rm Re}}
\def\Im{{\rm Im}}

\def\la{\langle}
\def\ra{\rangle}

\DeclareMathOperator{\End}{End}

\DeclareMathOperator{\Ker}{Ker}

\DeclareMathOperator{\Id}{Id}

\DeclareMathOperator{\tr}{Tr}
\DeclareMathOperator{\Ind}{Ind}

\DeclareMathOperator{\ch}{ch}

\newtheorem{thm}{Theorem}[section]
\newtheorem{lemma}[thm]{Lemma}
\newtheorem{prop}[thm]{Proposition}

\theoremstyle{definition}
\newtheorem{rem}[thm]{Remark}
\theoremstyle{definition}
\newtheorem{defn}[thm]{Definition}
\newtheorem{assump}[thm]{Assumption}

\newcommand{\be}{\begin{eqnarray}}
\newcommand{\ee}{\end{eqnarray}}

\newcommand{\wi}{\widetilde}
\newcommand{\var}{\varepsilon}
\numberwithin{equation}{section}
\numberwithin{thm}{section}

\newcommand{\comment}[1]{}

\begin{document}

\title{Comparison of two equivariant $\eta$-forms}

\author{Bo LIU}
\address{School of Mathematical Sciences,
	Shanghai Key Laboratory of PMMP,
	East China Normal University, 
	500 Dongchuan Road, 
	Shanghai, 200241 
	P.R. China}
\email{bliu@math.ecnu.edu.cn}

\author{Xiaonan MA}

\address{Universit\'e de Paris and Sorbonne Université, CNRS, 
IMJ-PRG, F-75013 Paris, France}

\email{xiaonan.ma@imj-prg.fr}

\date{\today}

\begin{abstract}
In this paper, we 
first define
the equivariant infinitesimal 
$\eta$-form, then we compare it with the equivariant $\eta$-form, 
modulo exact forms, by a locally computable form. As a 
consequence, we obtain the singular behavior of the equivariant 
$\eta$-form, modulo exact forms, as a function on the acting 
Lie group. This result extends a result of Goette
and it plays an important role 
in our recent work on the localization of $\eta$-invariants
and on the differential $K$-theory.
\end{abstract}
\maketitle



\tableofcontents

\setcounter{section}{-1}

\section{Introduction} \label{s00}
In order to find a well-defined index for a first order elliptic 
differential operator over an even-dimensional compact manifold 
with nonempty boundary, Atiyah-Patodi-Singer \cite{APS75} 
introduced a global boundary condition which is particularly 
significant for applications. In this final index formula, 
the contribution from the boundary is given by the 
Atiyah-Patodi-Singer (APS) $\eta$-invariant associated with 
the restriction of the operator on the boundary. Formally, the 
$\eta$-invariant is equal to the number of positive eigenvalues 
of the self-adjoint operator minus the number of its negative 
eigenvalues. If the manifold admits a compact Lie group action, 
in \cite{D78}, extending the APS index theorem \cite{APS75}, 
Donnelly proved a Lefschetz type formula 
for manifolds with boundary. The contribution of the boundary is 
expressed as the equivariant $\eta$-invariant $\eta_g$.

Note that the $\eta$-invariant and the equivariant $\eta$-invariant 
are well-defined for any compact manifold. In 
\cite[Theorem 0.5]{Go00}, Goette studied the singularity of 
$\eta_g$ at $g=e$ the identity element,
when the group action is locally free. He 
defined the equivariant infinitesimal $\eta$-invariant as a 
formal power series and express the singularity of $\eta_g$ 
at $g=e$ as a locally computable term through the comparison of 
the equivariant infinitesimal $\eta$-invariant and the 
equivariant $\eta$-invariant.

In \cite{BG00,BG04}, Bismut and Goette established the general 
comparison formulas for holomorphic analytic torsions and 
de Rham torsions. They used
the analytic localization techniques 
developed by Bismut and Lebeau in \cite{BL91} and developed 
new techniques to overcome the difficulty that the operators 
do not have lower bounds. In the holomorphic case 
\cite[Theorem 0.1]{BG00}, besides the predictable Bott-Chern 
current, in the final formula, there is an exotic additive 
characteristic class of the normal bundle, which is closely 
related to the Gillet-Soul\'e R-genus \cite{GilS} and Bismut's 
equivariant extension \cite{Bi94}. In the real case 
\cite[Theorem 0.1]{BG04}, in the final formula, besides the 
predictable Chern-Simons current, they discovered an exotic 
locally computable diffeomorphism invariant of the fixed 
point set, the so-called $V$-invariant. The mysterious 
$V$-invariant should be understood as a finite dimensional 
analogue of the real analytic (de Rham) torsion.

On the other hand, extending the works of Bismut-Freed  \cite{BF86II} 
and Cheeger \cite{Ch87} on 
the Witten's holonomy conjecture, Bismut and Cheeger \cite{BC89} 
studied the adiabatic limit for a fibration of compact spin 
manifolds and found that under the invertible assumption of 
the fiberwise Dirac operator, the adiabatic limit of the 
$\eta$-invariant of the associated Dirac operators on the 
total space is expressible in terms of a canonically constructed 
differential form, $\tilde{\eta}$, so-called Bismut-Cheeger 
$\eta$-form, on the base space. Later, Dai \cite{Dai91} 
extended this result to the case when the kernels of the 
fiberwise Dirac operators form a vector bundle over the base 
manifold. The Bismut-Cheeger $\eta$-form, $\tilde{\eta}$, is 
the families version of the $\eta$-invariant and its $0$-degree 
part is just the APS $\eta$-invariant. 
It appears naturally as the boundary contribution 
of the family index theorem 
for manifolds with boundary (cf. \cite{BC90I,BC90II,MP97,Mu95}).
We cite also \cite{Zh94} for a nice topological application of eta forms.
As the holomorphic analytic torsion and its family version,
Bismut-K\"ohler holomorphic torsion form \cite{BK92}
are the analytic counterpart
to the direct image in Arakelov geometry \cite{Soule92}, whose
foundation was developed by Gillet-Soul\'e and Bismut in the 1980s, 
the Bismut-Cheeger $\eta$-form is also the analytic counterpart
to the direct image in differential $K$-theory
introduced  by Freed-Hopkins \cite{FH00} 
and developed further by
 \cite{BS09},  \cite{FreedLott10}, \cite{HopkinsSinger05},
  \cite{SimonsS08}, etc.

When the fibration admits a fiberwise compact Lie group action,
the Bismut-Cheeger $\eta$-form could be naturally extended to 
the equivariant $\eta$-form $\wi{\eta}_g$. Recently, the 
functoriality of equivariant $\eta$-forms with respect to the 
composition of two submersions was established in \cite{Liu17a}, 
which extends the previous work of Bunke-Ma \cite{BM04} for 
usual $\eta$-forms for flat vector bundles with duality, cf. 
\cite{Be09,BB94,BKR11,DZ15,Ma99,Ma00,Ma00a,Puchol16} 
for related works on $\eta$-forms and holomorphic torsions.

In the same way as fixed-point formula has two equivariant versions,
the Lefschetz fixed-point formula and Kirillov-like formula 
of Berline-Vergne \cite{BV83}, the same is true for equivariant
$\eta$-forms. 
In this paper, we use the analytic techniques of Bismut-Goette 
in \cite{BG00} to define the equivariant infinitesimal 
Bismut-Cheeger $\eta$-form and prove a general comparison 
formula  
between the equivariant infinitesimal Bismut-Cheeger 
$\eta$-form and the equivariant Bismut-Cheeger 
$\eta$-form which extend the work of Goette \cite{Go00}. 
In particular, we express the singularity of 
$\wi{\eta}_g$ modulo exact forms, at any $g\in G$ as a locally 
computable differential form.

Let $G$ be a compact Lie group with Lie algebra $\mathfrak{g}$. 
We assume that $G$ acts isometrically on an odd-dimensional 
compact oriented Riemannian manifold $X$ and the $G$-action 
lifts on a Clifford module $\mE$ over $X$. In general, the 
equivariant APS $\eta$-invariant $\eta_g$ is not a continuous 
function on $g\in G$. In \cite{Go00}, Goette studied the 
singularity of the equivariant $\eta$-invariant $\eta_g$ at 
$g=e$. He 
defined
a formal power series $\eta_K\in 
\C[[\mathfrak{g}^*]]$ for $K\in \mathfrak{g}$, called the 
equivariant infinitesimal $\eta$-invariant and 
showed that 
if the Killing vector field $K^X$ induced by $K$ has no 
zeroes on $X$, for any $N\in \N$, as $0\neq t\rightarrow 0$,
\begin{align}\label{eq:0.01}
[\eta_{tK}]_N-\eta_{e^{tK}}=\mM_{tK}+\mathcal{O}(t^N),
\end{align}
where $[\eta_{tK}]_N$ is the part of the formal power series 
$\eta_{tK}$ with degree $\leq N$ and $\mM_{tK}$ could be 
expressed precisely as a locally computable term. Moreover, 
there exist $c_j(K)\in \C$ such that when $t\rightarrow 0$,
\begin{align}\label{eq:0.02}
\mM_{tK}=\sum_{j=1}^{(\dim X+1)/2}c_j(K)t^{-j}+\mathcal{O}(t^0).
\end{align}
It means that if the Killing vector field $K^X$ is nowhere 
vanishing, the singular behavior of $\eta_{e^{tK}}$ when 
$t\rightarrow 0$ could 
be computed as the integral of the local terms explicitly.

In this paper, we show first that $\eta_{tK}$ is an analytic 
function on $t$ for $t$ small enough and for any 
$0\neq K\in \mathfrak{g}$,
\begin{align}\label{eq:0.03}
\eta_{tK}-\eta_{e^{tK}}=\mM_{tK},\quad
\text{for $t\neq 0$ small enough.}
\end{align}
In Theorem \ref{thm:0.2}, we establish 
a general version of (\ref{eq:0.03}), in particular, its family version.

Let's explain in detail our result here. Let 
$\pi:W\rightarrow B$ be a smooth submersion of smooth compact 
manifolds with fiber $X$. Note that $n=\dim X$ can be even 
or odd. Let $TX=TM/B$ be the relative tangent bundle to the 
fiber $X$. We assume that $TX$ is oriented and that the compact 
Lie group $G$ acts fiberwise on $W$ and as identity on $B$ 
and preserves the orientation of $TX$. 

Let $g^{TX}$ be a $G$-invariant metric on $TX$. Let 
$(\mE,h^{\mE})$ be a Clifford module of $TX$ to the fiber 
$X$ and we assume that the $G$-action lifts on 
$(\mE,h^{\mE})$ and is compatible with the Clifford action. 
Let $\nabla^{\mE}$ be a $G$-invariant 
Clifford connection on $(\mE, h^{\mE})$, i.e., $\nabla^{\mE}$ 
is a $G$-invariant Hermitian connection on $(\mE, h^{\mE})$ 
and compatible 
with the Clifford action (see (\ref{eq:1.17})). Let $D$ 
be the fiberwise Dirac operator associated with 
($g^{TX}, \nabla^{\mE}$) (see (\ref{eq:1.18})). 

\textbf{We assume that the
kernels  $\Ker (D)$ form a vector bundle over $B$.}
Then for any $g\in G$, the equivariant $\eta$-form $\tilde{\eta}_g$ 
is well-defined (see Definition \ref{defn:1.03})\footnote{For 
even dimensional fiber, 
any family of Dirac operators could be deformed
 to another one which satisfies this assumption
and has the same family index in $K^0(B)$ (see e.g., \cite[\S 9.5]{BGV}).
But for odd dimensional fiber, some topological obstruction appears:
if a family of Dirac operators $D$ satisfies this assumption, the family
index of $D$ vanishes in $K^1(B)$ (this fact is implicitly contained 
in \cite{AS69}, a proof of which is presented in \cite[Theorem 4.1]{E13}). 
Recently, for odd dimensional fiber
case, Wittmann \cite{Wi15} 
defined an $\eta$-form under the 
assumption that the family
of Dirac operators has one eigenvalue of multiplicity one crossing 
zero transversally. It is expected that many properties of Bismut-Cheeger
$\eta$-form could be extended to this case.}.

In the whole paper, if $n=\dim X$ is even, $\mE$ is naturally 
$\Z_2$-graded by the chirality operator $\Gamma$ defined in 
(\ref{eq:1.15}) and the supertrace for $A\in \End(\mE)$ is 
defined by $\tr_s[A]:=\tr[\Gamma A]$; if $\dim X$ is odd, 
$\mE$ is ungraded. For $\sigma=\alpha\otimes A$ with 
$\alpha\in \Lambda(T^*B)$, $A\in \End(\mE)$, we define 
$\tr[\sigma]:=\alpha\cdot\tr[A]$.
We denote by $\tr^{\mathrm{odd}}[\sigma]$  the odd degree part 
of $\tr[\sigma]$. Set
\begin{align}\label{eq:0.04}
\widetilde{\tr}[\sigma]=
\left\{
\begin{array}{ll}
\tr_s[\sigma]\hspace{10mm} & \hbox{if $n=\dim X$ is even;} \\
\tr^{\mathrm{odd}}[\sigma]\hspace{10mm} & 
\hbox{if $n=\dim X$ is odd.}
\end{array}
\right.
\end{align}

For $\alpha\in \Omega^j(\R\times B)$, the space of $j$-th 
differential forms on $\R\times B$,  set
\begin{align}\label{eq:0.05}
\psi_{\R\times B}(\alpha)=\left\{
\begin{array}{ll}
\left(2i\pi  \right)^{-\frac{j}{2}}\cdot \alpha
\hspace{10mm} & \hbox{if $j$ is even;} \\
\pi^{-\frac{1}{2}}\left(2i\pi 
 \right)^{-\frac{j-1}{2}}\cdot \alpha 
\hspace{10mm} & \hbox{if $j$ is odd.}
\end{array}
\right.
\end{align}
Let $t$ be the coordinate of $\R$ in $\R\times B$. If
$\alpha=\alpha_0+dt\wedge \alpha_1$, with $\alpha_0, \alpha_1\in 
\Lambda (T^*B)$, set
\begin{align}\label{eq:0.06}
[\alpha]^{dt}:=\alpha_1.
\end{align}

Let $\mL_{K}$ be the infinitesimal action on $\cC^{\infty}(W,\mE)$ 
induced by $K\in \mathfrak{g}$ (see (\ref{eq:2.02b})). 

For $g\in G$, we denote by 
$Z(g)\subset G$ the centralizer subgroup of $g$ with Lie algebra 
$\mathfrak{z}(g)$. Let $W^g=\{x\in W: gx=x \}$ be the fixed point 
set of $g$. Then the restriction of $\pi$ on $W^g$, 
$\pi|_{W^g}:W^g\rightarrow B$ is a 
fibration with compact fiber $X^g$. 


Let $\mathbb{B}_t$ be the rescaled Bismut superconnection defined 
in (\ref{eq:1.22}). Let $d$ be the exterior differential operator.

Let $\widehat{\mathrm{A}}_{g,K}(\cdot)$ and $\ch_{g,K}(\cdot)$ 
be equivariant infinitesimal versions of the $\widehat{
	\mathrm{A}}$-form and the Chern character form (cf. 
(\ref{eq:2.14}) and (\ref{eq:2.15})).
The following result extends the equivariant infinitesimal 
$\eta$-invariant to the family case at any $g\in G$ 
(see Definition \ref{defn:2.03}, (\ref{eq:2.29}), 
(\ref{eq:2.30}), (\ref{eq:2.33}) and 
(\ref{eq:2.34})).

\begin{thm}\label{thm:0.1} 
For any $g\in G$, there exists $\beta>0$ such that if
$K\in \mathfrak{z}(g)$ with $|K|<\beta$, the integral
\begin{align}\label{eq:0.07}
\wi{\eta}_{g,K}=-\int_{0}^{+\infty}\left\{\psi_{\R\times B}
\wi{\tr}\left[g\exp\left(-\left(\mathbb{B}_{t}
+\frac{c(K^X)}{4\sqrt{t}}+dt\wedge \frac{\partial}{\partial
t}\right)^2-\mathcal{L}_{K}\right) \right]\right\}^{dt}dt
\end{align}
is a well-defined differential form on $B$,
and
\begin{align}\label{eq:0.08}
d\tilde{\eta}_{g,K}=\left\{
\begin{aligned}
&\int_{X^g}\widehat{\mathrm{A}}_{g,K}(TX,\nabla^{TX})\, 
\ch_{g,K}(\mE/\mS, \nabla^{\mE})\\
&\hspace{50mm}-\ch_{ge^K}(\Ker (D), \nabla^{\Ker (D)}) 
& \hbox{if $n$ is even;} \\
&\int_{X^g}\widehat{\mathrm{A}}_{g,K}(TX,\nabla^{TX})\, 
\ch_{g,K}(\mE/\mS, \nabla^{\mE}) & \hbox{if $n$ is odd.}
\end{aligned}
\right.
\end{align}
 Moreover, for fixed $K\in \mathfrak{z}(g)$, 
$\wi{\eta}_{g,zK}$ is an  analytic function of
 $z\in \C$ for $|zK|<\beta$. 
\end{thm}

In the sequel, $\wi{\eta}_{g,K}$ is called 
the equivariant infinitesimal (Bismut-Cheeger)
$\eta$-form.

Let $\vartheta_K\in T^*X$ be the $1$-form which is dual to 
$K^X$ by the metric $g^{TX}$. 
Now we state the main result of this paper.

\begin{thm}\label{thm:0.2} 	
For $g\in G$ and $K_{0}\in \mathfrak{z}(g)$, 
there exists  $\beta>0$ such that for any $K=zK_{0}$,  $K\neq 0$ 
and $-\beta<z<\beta$,
 modulo exact forms on $B$, we have
\begin{align}\label{eq:0.09}
\tilde{\eta}_{g,K}=\tilde{\eta}_{ge^K}+\mM_{g,K},
\end{align}
where $\mM_{g,K}$ is a well-defined integral defined by
\begin{align}\label{eq:0.10}
\mM_{g,K}=-\int_0^{+\infty}\int_{X^g}\frac{\vartheta_K}{2 i\pi v}
\exp\left(\frac{d\vartheta_K-2i\pi |K^X|^2}{2 i\pi v}
\right)\widehat{\mathrm{A}}_{g,K}(TX,\nabla^{TX})
\ch_{g,K}(\mE/\mS,\nabla^{\mE})\frac{dv}{v},
\end{align}
and 
$t^{\lfloor(\dim W^g+1)/2\rfloor}
\mM_{g,tK}$ is real analytic on $t\in \R$, $|t|<1$. 
Moreover, we have
\begin{multline}\label{eq:0.12}
d\mM_{g,K}=\int_{X^{g}}
\widehat{\mathrm{A}}_{g,K}(TX,\nabla^{TX})
\ch_{g,K}(\mE/\mS,\nabla^{\mE})
\\
-
\int_{X^{ge^K}}
\widehat{\mathrm{A}}_{ge^K}(TX,\nabla^{TX})
\ch_{ge^K}(\mE/\mS,\nabla^{\mE}).
\end{multline}	
\end{thm}

By Theorem \ref{thm:0.1}, $\tilde{\eta}_{g,tK}$ is 
an analytic function of $t$ near 
$t=0$. Thus when $t\rightarrow 0$, modulo exact forms, the 
singularity of $\tilde{\eta}_{ge^{tK}}$ is the same as 
that of $-\mM_{g,tK}$.

Note that the general comparison formula for the two versions of 
equivariant holomorphic analytic torsions is established in 
\cite[Theorem 5.1]{BG00}, which is the model of our paper. The 
analytical tools in this paper are inspired by those of \cite{BG00} with 
necessary modifications. For this problem on de Rham torsion 
forms, a comparison formula is stated in  
\cite[Theorem 5.13]{BG04}.

\begin{rem}\label{rem:0.3}
	Let $G$ act on an odd dimensional
	compact Riemannian manifold $(X,g^{TX})$ and 
	on a Clifford module $(\mE, h^{\mE}, \nabla^{\mE})$ compatible 
	with the Clifford action. Then for $g=e$ the identity element
	of $G$, (\ref{eq:0.07}) defines a 
	complex number $\eta_{K}$ for any $K\in \mathfrak{g}$, $|K|<\beta$.
	As formal power series on $K$, this $\eta_K$ is just the 
	equivariant infinitesimal $\eta$-invariant 
	$\eta_{K}$ in \cite[Definition 0.4]{Go00}.
	
	Let $P\rightarrow B$ be a $G$-principal bundle with connection
	and associated curvature $\Omega$.
	Then we get naturally a fibration $P\times_GX\rightarrow B$ with 
	fiber $X$.
	Let $\widetilde{\eta}$ be the associated Bismut-Cheeger $\eta$-form.
	For this fibration, by Bismut  \cite[\S1d), \S3b)]{Bi86}, 
	under the notation of (\ref{eq:1.22}),
	the term $c(T^H)$ in the Bismut
	superconnection is $c(\Omega)$,
	and $(\nabla^{\mathbb{E},u})^2=\mathcal{L}_\Omega$,
	thus we get \cite[Lemma 1.14]{Go00},
	\begin{align}\label{eq:19b}
	\widetilde{\eta}
	=\eta_{\frac{i}{2\pi}\Omega}.
	\end{align}
	Thus we can understand the formal power series of
	$\eta_K$ as a universal $\eta$-form.
\end{rem}

\begin{rem}\label{rem:0.4}
Assume 
temporarily
that $B=\mathrm{pt}$, $\dim X=n$
is odd, and $X$ is the boundary of a $G$-equivariant
Riemannian manifold
$Z$, which has product structure near $X$. We also assume that
$\mE_Z=\mE_Z^+\oplus \mE_Z^-$ is a $G$-equivariant Clifford
module on $Z$ such that $\left.\mE_Z^+\right|_X=\mE$ and 
$\mE_Z^{\pm}$ near $X$
is the pull-back of $\mE$ as Hermitian vector
bundles with connections.

Let $D_Z$ be the associated Dirac operator on $\mE_Z$ over $Z$.
Then the index of $D_Z^+:=D_Z|_{\cC^{\infty}(Z,\mE_Z^+)}$
with respect to the Atiyah-Patodi-Singer (APS) boundary condition
is a virtual representation of $G$. For $g\in G$, its equivariant
APS index $\Ind_{\mathrm{APS},g}(D_Z^+)$ can be computed by 
Donnelly's theorem \cite{D78},
\begin{align}\label{eq:0.12b}
\Ind_{\mathrm{APS},g}(D_Z^+)=\int_{Z^g}
\widehat{\mathrm{A}}_{g}(TZ,\nabla^{TZ})
\ch_{g}(\mE_Z/\mS_Z,\nabla^{\mE_Z})
-\frac{1}{2}\left(\eta_g(D)+\tr|_{\Ker(D)}[g] \right).
\end{align}
By combining (\ref{eq:0.09}), (\ref{eq:0.12}) 
(more precisely the Stokes formula \cite[p. 775]{BrM06},
(\ref{eq:3.25b}) and (\ref{eq:3.26a})), 
and (\ref{eq:0.12b}), for any
$K\in \kg$, there exists $\beta>0$ such that, for any 
$-\beta<t<\beta$, we have 
\begin{align}\label{eq:0.12c}
\Ind_{\mathrm{APS},e^{tK}}(D_Z^+)=\int_{Z}
\widehat{\mathrm{A}}_{tK}(TZ,\nabla^{TZ})
\ch_{tK}(\mE_Z/\mS_Z,\nabla^{\mE_Z})
-\frac{1}{2}\left(\eta_{tK}(D)+\tr|_{\Ker(D)}[e^{tK}] \right).
\end{align}
Here $\widehat{\mathrm{A}}_{tK}(\cdot):=
\widehat{\mathrm{A}}_{e,tK}(\cdot)$
and $\ch_{tK}(\cdot):=\ch_{e,tK}(\cdot)$.
\end{rem}

The main result of this paper is announced in \cite{LM18} and 
plays an important role in our recent work \cite{LM18a}.

This paper is organized as follows.   
In Section \ref{s01}, we recall the definition of the equivariant 
Bismut-Cheeger $\eta$-form. In Section \ref{s02}, we state the 
family Kirillov formula and define the equivariant infinitesimal 
$\eta$-form, in particular, we establish Theorem \ref{thm:0.1}
modulo some technical  details. 
In Section \ref{s03}, we prove that $\mM_{g,tK}$ 
in (\ref{eq:0.10}) is 
well-defined and state our main result, Theorem \ref{thm:0.2}. In 
Section \ref{s04}, we state some intermediate results and 
prove Theorem \ref{thm:0.2}. In Section \ref{s05}, we give an 
analytic proof 
of the family Kirillov formula and the technical  details
to establish Theorem \ref{thm:0.1} 
following the 
lines of \cite[\S 7]{BG00}. For the convenience to compare the 
arguments here with those in \cite[\S 7]{BG00}, especially how 
the extra terms for the families version appear, the structure of 
this section is formulated almost the same as in \cite[\S 7]{BG00}. 
In Section 
\ref{s06}, we prove the intermediate results in Section \ref{s04} 
using the analytical 
techniques in \cite[\S 8, \S9]{BG00}.

From Remark \ref{rem:1.03}, to simplify the presentation, in Sections 
\ref{s05}, \ref{s06}, we will {\bf assume} that $TX^g$ is oriented.

\textbf{Notation}: we use the Einstein summation convention 
in this paper: when an index variable appears twice in a single 
term and is not otherwise defined, it implies summation of that 
term over all the values of the index.

We denote by $\lfloor x\rfloor$ the maximal integer not larger than $x$.

We denote by $d$ the exterior differential operator and 
$d^B$ when we like to insist the base manifold $B$.
Let $\Omega^{\mathrm{even/odd}}(B,\C)$ be the space of 
even/odd degree complex valued differential forms on $B$.
For a real vector bundle $E$, we denote by $\dim E$ the 
real rank of $E$.

If $\mA$ is a $\Z_2$-graded algebra, and if $a,b\in \mA$, 
then we will note $[a,b]:=ab-(-1)^{\deg a\cdot\deg b}ba$ 
as the supercommutator of $a$, $b$.
In the whole paper, if $\mA$, $\mA'$ are $\Z_2$-graded algebras
we will note $\mA\widehat{\otimes}\mA'$ as the $\Z_2$-graded 
tensor product as in \cite[\S 1.3]{BGV}.
If one of $\mA, \mA'$ is ungraded, we understand 
it as $\Z_2$-graded by taking its odd part as zero.

For the fiber bundle $\pi: W\rightarrow B$, we will often use 
the integration of the differential forms along the oriented 
fibers $X$ in this paper. Since the fibers may be odd 
dimensional, we must make precisely our sign conventions: 
for $\alpha\in \Omega^{\bullet}(B)$ and 
$\beta\in \Omega^{\bullet}(W)$, then
\begin{align}\label{eq:0.11}
\int_X(\pi^*\alpha)\wedge\beta=\alpha\wedge \int_X\beta.
\end{align}


\noindent{\bf Acknowledgments}. 
B.\ L.\ is partially supported by Science and Technology Commission 
of Shanghai Municipality (STCSM), grant No.18dz2271000, 
Natural Science Foundation of Shanghai, grant No.20ZR1416700 
and NSFC No.11931007.
X.\ M.\ is partially supported by
NSFC No.11528103, No.11829102, ANR-14-CE25-0012-01,  and
funded through the Institutional Strategy of
the University of Cologne within the German Excellence 
Initiative. Part of this work was done while the authors 
were visiting University of Science and Technology in 
China and Wuhan University. 

\section{Equivariant $\eta$-forms}\label{s01} 
In this section, we recall the definition of the equivariant 
$\eta$-form 
in the language of Clifford modules. In Section \ref{s0101}, 
we recall the definition of the Clifford algebra. In Section
\ref{s0102}, we explain the Bismut superconnection. In 
Section \ref{s0103}, we define the equivariant $\eta$-form 
for Clifford module.

\subsection{Clifford algebras}\label{s0101}

Let $(V, \la,\ra)$ be a Euclidean space, such that 
$\dim V=n$, 
with orthonormal basis $\{e_i\}_{i=1}^n$. 
Let $c(V)$ be the 
Clifford algebra of $V$ defined by the relations
\begin{align}\label{eq:1.01}
e_ie_j+e_je_i=-2\delta_{ij}.
\end{align}
To avoid ambiguity, we denote by $c(e_i)$ the element of 
$c(V)$ corresponding to $e_i$.

If $e\in V$, let $e^*\in V^*$ correspond to $e$ 
by the scalar product 
$\la , \ra$ of $V$. The exterior algebra $\Lambda V^*$ is 
a module of $c(V)$ defined by 
\begin{align}\label{eq:1.02} 
c(e)\alpha=e^*\wedge\alpha -i_{e}\alpha
\end{align} 
for any $\alpha\in\Lambda V^*$, where $\wedge$ is the 
exterior product and $i$ is the contraction 
operator. The map $a\mapsto c(a)\cdot 1$, $a\in c(V)$,
induces an isomorphism of vector spaces
\begin{align}\label{eq:1.03}
\sigma:c(V)\rightarrow \Lambda V^*.	
\end{align}

\subsection{Bismut superconnection}\label{s0102}
Let $\pi:W\rightarrow B$ be a smooth submersion of 
smooth compact manifolds with $n$-dimensional fibers $X$. 
Let $TX=TW/B$ be the relative tangent bundle to the fibers $X$. 

Let $G$ be a compact Lie group acting on $W$ along the fibers 
$X$, that is, if $g\in G$, $\pi\circ g=\pi$. Then $G$ acts 
on $TW$ and on $TX$. Let $T^HW\subset TW$ be a 
$G$-invariant horizontal subbundle, so that
\begin{align}\label{eq:1.04}
TW=T^HW\oplus TX.
\end{align}
Since $G$ is compact, such $T^HW$ always exists. Let $P^{TX}:TW
\rightarrow TX$ be the projection associated with the splitting 
(\ref{eq:1.04}). Note that
\begin{align}\label{eq:1.05}
T^HW\cong \pi^*TB.
\end{align}

Let $g^{TX}$ be a $G$-invariant metric on $TX$. Let $g^{TB}$ 
be a Riemannian metric on $TB$. We equip $TW$ with the 
$G$-invariant metric via (\ref{eq:1.04}) and (\ref{eq:1.05}),
\begin{align}\label{eq:1.06}
g^{TW}=\pi^*g^{TB}\oplus g^{TX}.
\end{align}
Let $\nabla^{TW,L}$ (resp. $\nabla^{TB}$) be the Levi-Civita 
connection on $(TW,g^{TW})$ (resp. $(TB, g^{TB})$). 
Let $\nabla^{TX}$ be the connection on $TX$ defined by
\begin{align}\label{eq:1.07}
\nabla^{TX}=P^{TX}\nabla^{TW,L}P^{TX}.
\end{align}
It is $G$-invariant.
Let $\nabla^{TW}$ be the $G$-invariant connection on $TW$, 
via (\ref{eq:1.04}) and (\ref{eq:1.05}),
\begin{align}\label{eq:1.08}
\nabla^{TW}=\pi^*\nabla^{TB}\oplus\nabla^{TX}.
\end{align}
Put
\begin{align}\label{eq:1.09}
S=\nabla^{TW,L}-\nabla^{TW}.
\end{align}
Then $S$ is a 1-form on $W$ with values in antisymmetric 
elements of $\End(TW)$. Let $T$ be the torsion of 
$\nabla^{TW}$. By \cite[(1.28)]{Bi86}, if $U,V,Z\in TW$,
\begin{align}\label{eq:1.10}
\begin{split}
&S(U)V-S(V)U+T(U,V)=0,
\\
&2\la S(U)V,Z\ra+\la T(U,V),Z\ra+\la T(Z,U),V\ra-\la 
T(V,Z),U\ra=0.
\end{split}
\end{align}

If $U$ is a vector field on $B$, let $U^H$ be its lift in 
$T^HW$ and let $\mL_{U^H}$ be the Lie derivative operator 
associated with the vector field $U^H$. 
Then $\mL_{U^H}$ acts on the tensor algebra of $TX$. In 
particular, if $U\in TB$,$\left(g^{TX}\right)^{-1}
\mL_{U^H}g^{TX}$ defines a self-adjoint endomorphism of $TX$.
If $U,V$ are vector fields on $B$, from 
\cite[Theorem 1.1]{Bi97},
\begin{align}\label{eq:1.11}
T(U^H,V^H)=-P^{TX}[U^H,V^H],
\end{align}
and if $U\in TB$, $Z, Z'\in TX$,
\begin{align}\label{eq:1.12}
T(U^H,Z)=\frac{1}{2}\left(g^{TX}\right)^{-1}\mL_{U^H}g^{TX}Z,
\quad T(Z,Z')=0.
\end{align}
From (\ref{eq:1.10}) and (\ref{eq:1.12}), 
if $U\in TB$, $Z,Z'\in TX$, we have
\begin{align}\label{eq:1.13}
\la S(Z)Z',U^H\ra=-\la T(U^H,Z),Z'\ra=-\la T(U^H,Z'),Z\ra.
\end{align}

We recall some properties in 
\cite[\S 1.1]{Bi97}.
\begin{prop}\label{prop:1.01}

1) The connection $\nabla^{TX}$ does not 
depend on $g^{TB}$ and on each fiber $X$, 
it restricts to the Levi-Civita connection of $(TX,g^{TX})$.

2) If $U\in TB$, then
\begin{align}\label{eq:1.14}
\nabla^{TX}_{U^H}=\mL_{U^H}+\frac{1}{2}
\left(g^{TX}\right)^{-1}\mL_{U^H}g^{TX}.
\end{align}

3) The tensors $T$ and $\la 
	S(\cdot)\cdot,\cdot\ra$ do not depend on $g^{TB}$.	
\end{prop}

Let $c(TX)$ be the Clifford algebra bundle of $(TX, g^{TX})$, 
whose fiber at $x\in W$ is the Clifford algebra $c(T_xX)$ of 
the Euclidean space $(T_xX, g^{T_xX})$. Let $\mE$ be a Clifford 
module of $c(TX)$. It means that $\mE$ is a complex vector bundle and 
restricted on a fiber, $\mE_x$ is a representation 
of $c(T_xX)$. We assume that the $G$-action lifts on $\mE$ and 
commutes with the Clifford action. 

From now on, we assume that $TX$ is $G$-equivariant oriented.

\emph{In the whole paper},
if $n$ is even, as in \cite[Lemma 3.17]{BGV}, 
for a locally oriented orthonormal frame
$e_1,\cdots,e_n$ of $TX$, we define the chirality operator by
\begin{align}\label{eq:1.15}
\Gamma=i^{n/2}c(e_1)\cdots c(e_n). 
\end{align}
Then $\Gamma$ does not depend on the choice of the frame, 
commutes with the $G$-action and $\Gamma^2=\Id$. Thus $\mE$ 
is naturally $\Z_2$-graded by the chirality operator $\Gamma$. 
The supertrace for $A\in \End(\mE)$ is defined by 
\begin{align}\label{eq:1.16}
\tr_s[A]:=\tr[\Gamma A].
\end{align}
If $n$ is odd, $\mE$ is ungraded.  

Let $h^{\mE}$ be a $G$-invariant Hermitian metric on $\mE$. 
For $b\in B$, let $\mathbb{E}_{b}$ be the set of smooth 
sections over $X_b=\pi^{-1}(b)$ of 
$\mE|_{X_b}$. As in \cite{Bi86},
we will regard $\mathbb{E}$ as
an infinite dimensional vector bundle over $B$.
Let $dv_X(x)$ be the Riemannian volume element of $X_b$. 
The bundle $\mathbb{E}_b$ is naturally endowed with the 
Hermitian product
\begin{align}\label{eq:1.19} 
\la s,s'\ra_0=\int_{X_b}\la s,s'\ra(x)dv_X(x), \quad
\text{for}\ s,s'\in \mathbb{E}.
\end{align}
Then $G$ acts on $\mathbb{E}_{b}=\cC^{\infty}
(X_b,\mE|_{X_b})$ as 
\begin{align}\label{eq:1.19b}
(g.s)(x)=g(s(g^{-1}x)) \quad \text{for any}\ g\in G.
\end{align}
Let $\nabla^{\mE}$ be a $G$-invariant Clifford connection 
on $\mE$ (cf. \cite[\S 10.2]{BGV}), 
that is, $\nabla^{\mE}$ is $G$-invariant, preserves 
$h^{\mE}$ and for any $U\in TW$, $Z\in \cC^{\infty}(W,TX)$,
\begin{align}\label{eq:1.17}
\left[\nabla_U^{\mE}, c(Z)\right]=c\left(\nabla^{TX}_U
Z\right).
\end{align}
The fiberwise Dirac operator is defined by
\begin{align}\label{eq:1.18}
D=\sum_{i=1}^nc(e_{i})\nabla_{e_i}^{\mE},
\end{align}
which is independent of the choice of  the orthonormal frame
$\{e_i \}_{i=1}^n$.

Let $k\in (T^HW)^*$ such that for any $U\in TB$, 
$\mL_{U^H}dv_X(x)/dv_X(x)=2k(U^H)(x)$. The connection 
$\nabla^{\mathbb{E},u}$ on $\mathbb{E}$ defined 
by (cf. \cite[Definition 1.3]{BF86I})
\begin{align}\label{eq:1.20}
\nabla_U^{\mathbb{E},u}s:=\nabla_{U^H}^{\mE}s+k(U^H) s
\quad \text{ for } s\in \cC^{\infty}(B, \mathbb{E})
=\cC^{\infty}(W, \mE) ,
\end{align}
is $G$-invariant and preserves the $G$-invariant 
$L^2$-product (\ref{eq:1.19})
(see e.g., \cite[Proposition 1.4]{BF86I}).
 
Let $\{f_p\}$ be a local frame of $TB$ and $\{f^p\}$ be 
its dual. Set
\begin{align}\label{eq:1.21}
\nabla^{\mathbb{E},u}=f^p\wedge \nabla^{\mathbb{E},
u}_{f_p},\quad c(T^H)=\frac{1}{2}\,c\left(T(f_p^H, 
f_q^H)\right)f^p\wedge f^q\wedge.
\end{align}
Then $c(T^H)$ is a section of $\pi^*\Lambda^2(T^*B)
\widehat{\otimes}\End(\mE)$.

By \cite[(3.18)]{Bi86}, the rescaled Bismut superconnection
$\mathbb{B}_u$, $u>0$, is defined by
\begin{align}\label{eq:1.22}
\mathbb{B}_u=\sqrt{u}D+\nabla^{\mathbb{E},u}-\frac{1}{4
\sqrt{u}}c(T^H):\cC^{\infty}(B,\Lambda(T^*B)\widehat{
\otimes}\mathbb{E})\rightarrow\cC^{\infty}(B,\Lambda(T^*B)
\widehat{\otimes}\mathbb{E}).
\end{align}
Obviously, the Bismut superconnection $\mathbb{B}_u$ commutes 
with the $G$-action. Furthermore, $\mathbb{B}_u^2$ is a 
$2$nd-order elliptic differential operator along the fiber 
$X$ (cf. \cite[(3.4)]{Bi86}) acting on $\Lambda(T^*B)
\widehat{\otimes}\mathbb{E}$. Let $\exp(-\mathbb{B}_u^2)$ 
be the heat operators associated with the 
fiberwise elliptic operator $\mathbb{B}_u^2$.

\subsection{Equivariant $\eta$-forms}\label{s0103}

Take $g\in G$ fixed and set $W^g=\{x\in W: gx=x\}$, the 
fixed point set of $g$. Then $W^g$ is a submanifold of 
$W$ and $\pi|_{W^g}:W^g\rightarrow B$ is a fibration with 
compact fiber $X^g$. Let $N_{W^g/W}$ denote the normal 
bundle of $W^g$ in $W$, then 
\begin{align}\label{eq:1.23}
N_{W^g/W}:=\frac{TW}{TW^g}=\frac{TX}{TX^g}=:N_{X^g/X}.
\end{align}

Let $\{X^g_{\alpha} \}_{\alpha\in \mathfrak{B}}$ be the 
connected components of $X^g$ with 
\begin{align}\label{eq:1.24}
\dim X^g_{\alpha}=\ell_{\alpha}.
\end{align}
By an abuse of notation, we will often  simply denote 
by all $\ell_{\alpha}$ the same $\ell$.

\begin{assump}\label{ass:1.02}
We assume that the kernels $\Ker (D)$ form a vector bundle over $B$.
\end{assump}

For $\sigma=\alpha\widehat{\otimes} A$ with 
$\alpha\in \Lambda(T^*B)$, 
$A\in \End(\mE)$, we define
\begin{align}\label{eq:1.25}
\tr[\sigma]=\alpha\cdot\tr[A],\quad 
\tr^{\mathrm{odd}}[\sigma]=\{\alpha\}^{\mathrm{odd}}\cdot
\tr[A],\quad \tr^{\mathrm{even}}[\sigma]=\{\alpha\}^{
	\mathrm{even}}\cdot\tr[A],
\end{align}
where
$\{\alpha\}^{\mathrm{odd/even}}$ is the odd or even degree 
part of $\alpha$.
Set
\begin{align}\label{eq:1.26}
\widetilde{\tr}[\sigma]=\left\{
\begin{array}{ll}
\tr_s[\sigma]:=\alpha\cdot\tr[\Gamma A]\hspace{10mm} 
& \hbox{if $n=\dim X$ is even;} \\
\tr^{\mathrm{odd}}[\sigma]\hspace{10mm} & \hbox{if $n
	=\dim X$ is odd.}
\end{array}
\right.
\end{align}

Let $\End_{c(TX)}(\mE)$ be the set of endomorphisms of $\mE$ 
supercommuting with the Clifford action. It is a vector bundle 
over $W$. As in \cite[Definition 3.28]{BGV}, we 
define the relative trace 
$\tr^{\mE/\mS}:\End_{c(TX)}(\mE)\rightarrow \C$ by: 
for any 
$A\in \End_{c(TX)}(\mE)$, 
\begin{align}\label{eq:1.27}
\tr^{\mE/\mS}[A]=\left\{
\begin{array}{ll}
2^{-n/2}\tr_s[\Gamma A]\hspace{10mm} & 
\hbox{if $n=\dim X$ is even;} \\
2^{-(n-1)/2}\tr[A]\hspace{10mm} & \hbox{if $n=\dim X$ is odd.}
\end{array}
\right.
\end{align}

Let $R^{TX}=\left(\nabla^{TX}\right)^2$,
 $R^{\mE}=\left(\nabla^{\mE}\right)^2$ be the curvatures 
of $\nabla^{TX}$, $\nabla^{\mE}$. Then
\begin{align}\label{eq:1.28}
R^{\mE/\mS}:=R^{\mE}-\frac{1}{4}\la R^{TX}e_i, e_j\ra c(e_i)c(e_j)
\in \cC^{\infty}(W, \Lambda^2(T^*W)\otimes \End_{c(TX)}(\mE))
\end{align}
is the twisting curvature of the Clifford module $\mE$ as in 
\cite[Proposition 3.43]{BGV}.

Note that if $TX$ has a $G$-equivariant spin structure, then there 
exists a $G$-equivariant Hermitian vector bundle $E$ such that 
$\mE=\mS_X\otimes E$, with $\mS_X$ the spinor bundle of $TX$, 
$\nabla^{\mE}$ is induced by $\nabla^{TX}$ and a $G$-invariant 
Hermitian connection $\nabla^E$ on $E$ and 
\begin{align}\label{eq:1.29}
R^{\mE/\mS}=R^E=(\nabla^E)^2.
\end{align}

We denote the differential of $g$ by $dg$ which gives a bundle 
isometry $dg: N_{X^g/X}\rightarrow N_{X^g/X}$. As $G$ is compact, 
we know that there is an orthonormal decomposition 
of real vector bundles over $W^g$,
\begin{align}\label{eq:1.30}
TX|_{W^g}=TX^g\oplus N_{X^g/X}
=TX^g\oplus \bigoplus_{0<\theta\leq 	\pi}N(\theta),
\end{align}
where $dg|_{N(\pi)}=-\mathrm{Id}$ and for each $\theta$, 
$0<\theta<\pi$, $N(\theta)$ is the underlying real vector bundle 
of a complex vector bundle $N_{\theta}$ over $W^g$ on which 
$dg$ acts by multiplication by $e^{i\theta}$. Since $g$ preserves 
the metric and the orientation of $TX$, thus 
$\det(dg|_{N(\pi)})=1$, this means 
$\dim N(\pi)$ is even. So the normal bundle $N_{X^g/X}$ is even 
dimensional. 

Since $\nabla^{TX}$ commutes with the group action, its 
restriction on $W^g$, $\nabla^{TX}|_{W^g}$, preserves the 
decomposition (\ref{eq:1.30}). Let $\nabla^{TX^g}$ and 
$\nabla^{N(\theta)}$ be the corresponding induced 
connections on $TX^g$ and $N(\theta)$, with curvatures 
$R^{TX^g}$ and $R^{N(\theta)}$.

Set
\begin{multline}\label{eq:1.31}
\widehat{\mathrm{A}}_g(TX,\nabla^{TX})=\mathrm{det}^{\frac{1}{2}}
\left(\frac{\frac{ i}{4\pi}R^{TX^g}}{\sinh
	\left(\frac{ i}{4\pi}R^{TX^g}\right)}\right)
\\
\cdot
\prod_{0<\theta\leq \pi}\left( i^{\frac{1}{2}\dim 
	N(\theta)}\mathrm{det}^{\frac{1}{2}}\left(1-g
\exp\left(\frac{ i}{2\pi}R^{N(\theta)}\right)\right)
\right)^{-1}\in \Omega^{2\bullet}(W^g,\C).
\end{multline}
The sign convention in (\ref{eq:1.31}) is that the degree $0$
part in $\prod_{0<\theta\leq \pi}$ is given by
$\left(\frac{e^{i\theta/2}}{e^{i\theta}-1}\right)^{\frac{1}{2}
	\dim N(\theta)}$.

By \cite[Lemma 6.10]{BGV}, along $W^g$, the action of $g\in G$ 
on $\mE$ may be identified with a section $g^{\mE}$ of 
$c(N_{X^g/X})\otimes \End_{c(TX)}(\mE)$. Under the isomorphism 
(\ref{eq:1.03}), $\sigma(g^{\mE})\in \cC^{\infty}(W^g, 
\Lambda (N_{X^g/X}^*)\otimes \End_{c(TX)}(\mE))$. 
Let $\sigma_{n-\ell}(g^{\mE})\in 
\cC^{\infty}(W^g, \Lambda^{n-\ell}(N_{X^g/X}^*)\otimes
\End_{c(TX)}(\mE))$ 
be the highest degree part of $\sigma(g^{\mE})$ in  
$\Lambda(N_{X^g/X}^*)$. 
Then we define the localized relative Chern character 
$\ch_g(\mE/\mS,\nabla^{\mE})$ 
 as in \cite[Definition 6.13]{BGV}:
\begin{multline}\label{eq:1.32}
\ch_g(\mE/\mS, \nabla^{\mE}):=\frac{2^{(n-\ell)/2}}
{\mathrm{det}^{1/2}(1-g|_{N_{X^g/X}})}
\tr^{\mE/\mS}\left[\sigma_{n-\ell}(g^{\mE})\exp\left(-\frac{R^{
		\mE/\mS}|_{W^g}}{2i\pi}\right)\right]
\\	\in \Omega^{\bullet}\big(W^g,\det N_{X^g/X}\big).
\end{multline}

\begin{rem}\label{rem:1.03} 
In general, $TX^g$ is not necessary oriented. 
The orientation of $TX$ allows us to identify
$\det N_{X^g/X}$ as the orientation line of $X^{g}$,  thus 
the integral $\int_{X^g}$
of a form in $\Omega^{\bullet}\big(W^g,\det N_{X^g/X}\big)$
makes sense as in \cite[Theorem 6.16]{BGV}. Assume that
 $TX^g$ is oriented,
then the orientations of $TX^g$ and $TX$ induce canonically 
an orientation on $N_{X^g/X}$. By pairing with the volume form of 
$N_{X^g/X}$, we obtain
\begin{align}
\ch_g(\mE/\mS, \nabla^{\mE})\in \Omega^{\bullet}(W^g,\C).
\end{align}
If $TX$ has a $G$-equivariant spin$^c$ 
structure, then $TX^g$ is canonically oriented (cf. 
\cite[Proposition 6.14]{BGV}, \cite[Lemma 4.1]{LMZ00}). 
If $TX$ has a $G$-equivariant spin structure, 
$\ch_g(\mE/\mS,\nabla^{\mE})$  under the above convention
is just the usual equivariant Chern character (cf. (\ref{eq:1.29}))
\begin{align}\label{eq:1.33}
\ch_g(E,\nabla^E)=\tr^{E}\left[g
\exp\left(-\frac{R^{E}|_{W^g}}{2i\pi}\right)\right].
\end{align}
\end{rem}

As in (\ref{eq:0.05}),
for $\alpha\in \Omega^j(B)$, set
\begin{align}\label{eq:1.34}
\psi_B(\alpha)=\left\{
\begin{array}{ll}
\left(2i\pi\right)^{-\frac{j}{2}}\cdot \alpha\hspace{10mm} 
& \hbox{if $j$ is even;} \\
\pi^{-\frac{1}{2}}\left(2i\pi\right)^{-\frac{j-1}{2}}
\cdot \alpha\hspace{10mm} & \hbox{if $j$ is odd.}
\end{array}
\right.
\end{align}
Then from the equivariant family local index theorem (see e.g., 
\cite[Theorem 4.17]{Bi86}, \cite[Theorem 2.10]{BF86II},
\cite[Theorem 2.2]{Liu17b}, \cite[Theorem 1.3]{LM00}), 
for any $u>0$, the 
differential form $\psi_B\widetilde{\tr}[g\exp(-\mathbb{B}_u^2)]
\in \Omega^{\bullet}(B, \C)$ is closed, its cohomology class is 
independent of $u>0$, and
\begin{align}\label{eq:1.35}
\lim_{u\rightarrow 0}\psi_B\widetilde{\tr}[g\exp(-\mathbb{B}_u^2)]
=\int_{X^g}\widehat{\mathrm{A}}_g(TX,\nabla^{TX})\,
\ch_g(\mE/\mS,\nabla^{\mE}).
\end{align}

Let $P^{\Ker (D)}:\mathbb{E}\rightarrow \Ker (D)$ be the orthogonal 
projection with respect to (\ref{eq:1.19}). Let
\begin{align}\label{eq:1.36}
\nabla^{\Ker (D)}=P^{\Ker (D)}\nabla^{\mathbb{E},u}P^{\Ker (D)}
\end{align}
and $R^{\Ker(D)}$
be the curvature of the connection $\nabla^{\Ker(D)}$ on $\Ker(D)$.

$\bullet$ If $n=\dim X$ is even, from the natural 
equivariant extension of 
\cite[Theorem 9.19]{BGV}, we have
\begin{align}\label{eq:1.37}
\lim_{u\rightarrow +\infty}\psi_B\tr_s[g\exp(-\mathbb{B}_u^2)]
=\tr_s\left[g\exp\left(-\frac{R^{\Ker(D)}}{2i\pi}
\right)\right]=\ch_g(\Ker(D), \nabla^{\Ker(D)}).
\end{align}
Since $\mathbb{B}_{u}$ is $G$-invariant,
the equivariant version of \cite[Theorem 9.17]{BGV} shows that
\begin{align}\label{eq:1.38}
\frac{\partial}{\partial u}\tr_s\left[g
\exp(-\mathbb{B}_{u}^{2})\right]
=-d^B\tr_s\left[g\frac{\partial \mathbb{B}_{u}
}{\partial u}
\exp(-\mathbb{B}_{u}^{2})\right].
\end{align}
Thus for $0<\var<T<+\infty$,
\begin{align}\label{eq:1.39}
\tr_s\left[g
\exp(-\mathbb{B}_{\var}^{2})\right]-\tr_s\left[g
\exp(-\mathbb{B}_{T}^{2})\right]
=d^B\int_{\var}^T\tr_s\left[g\frac{\partial \mathbb{B}_{u}
}{\partial u}
\exp(-\mathbb{B}_{u}^{2})\right]du.
\end{align}
The natural equivariant extension of 
\cite[Theorems 9.23 and 10.32(1)]{BGV}
(cf. e.g., \cite[(2.72) and (2.77)]{Liu17a}) shows that
\begin{align}\label{eq:1.40} 
\begin{split}
&\tr_s\left[g\frac{\partial \mathbb{B}_{u}
}{\partial u}
\exp(-\mathbb{B}_{u}^{2})\right]=\mathcal{O}(u^{-1/2})\quad
\hbox{as $u\rightarrow 0$,}
\\
&\tr_s\left[g\frac{\partial \mathbb{B}_{u}
}{\partial u}
\exp(-\mathbb{B}_{u}^{2})\right]=\mathcal{O}(u^{-3/2})\quad 
\hbox{as $u\rightarrow +\infty$.}
\end{split}
\end{align} 
In this case, by (\ref{eq:1.34}) and (\ref{eq:1.40}), 
the equivariant $\eta$-form 
is defined by
\begin{align}\label{eq:1.41}
\tilde{\eta}_g=\int_0^{+\infty} \left.\frac{1}{2  i
	\sqrt{\pi}}\psi_{B}
\tr_s\right.\left[g\left.\frac{\partial\mathbb{B}_{u}}{
\partial u}\right.\exp(-\mathbb{B}_{u}^{2})\right] du\in 
\Omega^{\mathrm{odd}}(B,\C).
\end{align}
By (\ref{eq:1.35}), (\ref{eq:1.37}),
(\ref{eq:1.39}) and (\ref{eq:1.41}), we have
\begin{align}\label{eq:1.42}
d^B\tilde{\eta}_g=\int_{X^g}
\widehat{\mathrm{A}}_{g}(TX,\nabla^{TX})
\ch_{g}(\mE/\mS, \nabla^{\mE})
-\ch_{g}(\Ker (D), \nabla^{\Ker (D)}).
\end{align}

$\bullet$ If $n$ is odd, since the equivariant extension of 
\cite[Theorem 9.19]{BGV} also holds, we have
\begin{align}\label{eq:1.43}
\lim_{u\rightarrow +\infty}\tr^{\mathrm{odd}}[g
\exp(-\mathbb{B}_u^2)]
=\tr^{\mathrm{odd}}\left[g\exp\left(-R^{\Ker(D)}
\right)\right]=0.
\end{align} 
As an analogue of (\ref{eq:1.39}),
for $0<\var<T<+\infty$, we have
\begin{align}\label{eq:1.44}
\tr^{\mathrm{odd}}\left[g
\exp(-\mathbb{B}_{\var}^{2})\right]-\tr^{\mathrm{odd}}\left[g
\exp(-\mathbb{B}_{T}^{2})\right]
=d^B\int_{\var}^T\tr^{\mathrm{even}}\left[g
\frac{\partial \mathbb{B}_{u}
}{\partial u}
\exp(-\mathbb{B}_{u}^{2})\right]du.
\end{align}
Following the same arguments in the proof of (\ref{eq:1.40}),
we have
\begin{align}\label{eq:1.45} 
\begin{split}
&\tr^{\mathrm{even}}\left[g\frac{\partial \mathbb{B}_{u}
}{\partial u}
\exp(-\mathbb{B}_{u}^{2})\right]=\mathcal{O}(u^{-1/2})\quad
\hbox{as $u\rightarrow 0$,}
\\
&\tr^{\mathrm{even}}\left[g\frac{\partial \mathbb{B}_{u}
}{\partial u}
\exp(-\mathbb{B}_{u}^{2})\right]=\mathcal{O}(u^{-3/2})\quad
\hbox{as $u\rightarrow +\infty$.}
\end{split}
\end{align} 
 In this case, by (\ref{eq:1.34}) and (\ref{eq:1.45}),
the equivariant $\eta$-form is defined by
\begin{align}\label{eq:1.46}
\tilde{\eta}_g=\int_0^{+\infty} \left.\frac{1}{
	\sqrt{\pi}}\psi_{B}
\tr^{\mathrm{even}}\right.\left[g\left.\frac{\partial 
	\mathbb{B}_{u}}{\partial u}
\right.\exp(-\mathbb{B}_{u}^{2})\right] du\in 
\Omega^{\mathrm{even}}(B,\C).
\end{align}
From (\ref{eq:1.35}), (\ref{eq:1.43}), 
(\ref{eq:1.44}) and (\ref{eq:1.46}), we get
\begin{align}\label{eq:1.47}
d^B\tilde{\eta}_g=\int_{X^g}\widehat{
	\mathrm{A}}_{g}(TX,\nabla^{TX})
\ch_{g}(\mE/\mS, \nabla^{\mE}).
\end{align}

We write the definition of the equivariant 
$\eta$-form (\ref{eq:1.41})
and (\ref{eq:1.46}) in a uniform 
way using the notation $\{\cdot  \}^{du}$ as in (\ref{eq:0.06}).

\begin{defn}\label{defn:1.03}\cite[Definition 2.3]{Liu17a}
For $g\in G$ fixed, under Assumption \ref{ass:1.02}, 
the equivariant 
Bismut-Cheeger $\eta$-form is defined by
\begin{align}\label{eq:1.48}
\tilde{\eta}_g:=-\int_0^{+\infty}\left\{\psi_{\R\times B}\left.
\widetilde{\tr}\right.\left[g\exp\left(-\left(\mathbb{B}_{u}
+du\wedge\frac{\partial}{\partial u}\right)^2\right)\right]
\right\}^{du}du\in \Omega^{\bullet}(B,\C).
\end{align}	
\end{defn}

If $g=e$ the identity element of $G$,
(\ref{eq:1.48}) is exactly the Bismut-Cheeger 
$\eta$-form defined in \cite{BC89}. If $B$ is noncompact,
(\ref{eq:1.40}) and (\ref{eq:1.45}) hold uniformly on 
any compact subset of $B$, 
thus Definition \ref{defn:1.03}, (\ref{eq:1.42}) and 
(\ref{eq:1.47}) still hold.

\section{Equivariant infinitesimal $\eta$-forms}\label{s02}
In this section, we state the family Kirillov formula and define 
the equivariant infinitesimal $\eta$-form. In Section 
\ref{s0201}, we state the families version of the Kirillov 
formula. In Section \ref{s0202}, we define the equivariant 
infinitesimal $\eta$-form, and establish Theorem \ref{thm:0.1}
modulo some technical  details.

In this section, we use the same notations and assumptions in Section 1. 
Especially, $TX$ is $G$-equivariant oriented and 
Assumption \ref{ass:1.02} holds in this section.

\subsection{Moment maps and the family Kirillov formula}\label{s0201}

Let $|\cdot|$ be a $G$-invariant norm on the Lie algebra 
$\mathfrak{g}$ of $G$. For $K\in \mathfrak{g}$, let 
\begin{align}\label{eq:2.01}
K^X(x)=\left.\frac{\partial}{\partial t}\right|_{t=0}
e^{tK}\cdot x\quad \text{for}\ x\in W
\end{align}
be the induced vector field 
on $W$. Since $G$ acts 
fiberwise on $W$, $K^X\in \cC^{\infty}(W,TX)$ and
\begin{align}\label{eq:2.01b} 
[K^X,K'^X]=-[K,K']^X\quad\quad \text{for any}\ K,K'\in \mathfrak{g}.
\end{align}
For $K\in \mathfrak{g}$, let $\mL_{K}$ be 
the corresponding Lie derivative 
given by
\begin{align}\label{eq:2.02b} 
\mL_Ks=\left.\frac{\partial}{\partial t}\right|_{t=0}
\left(e^{-tK}. s\right),
\end{align}
for $s\in \cC^{\infty}(W,\mE)$ (cf. (\ref{eq:1.19b})).
The associated moment maps 
$m^{TX}(\cdot)$, $m^{\mE}(\cdot)$ 
are defined by \cite[Definition 7.5]{BGV}
(see also \cite[Definition 2.1]{BG00}),
\begin{align}\label{eq:2.02} 
\begin{split}
m^{TX}(K)&:=\nabla^{TX}_{K^X}
-\mL_{K}\vert_{TX}\in \cC^{\infty}(W,\End(TX)),
\\
m^{\mE}(K)&:=\nabla^{\mE}_{K^X}-\mL_{K}\vert_{\mE}\in 
\cC^{\infty}(W,\End(\mE)).
\end{split}
\end{align}
Since the vector field $K^X$ is Killing and $\nabla^{TX}$,
$\nabla^{\mE}$ preserve the corresponding metrics, 
$m^{TX}(K)$ and $m^{\mE}(K)$ are skew-adjoint actions of $\End(TX)$ 
and $\End(\mE)$ respectively.
By Proposition 
\ref{prop:1.01}, the connection $\nabla^{TX}$ 
is the Levi-Civita connection of $(TX, g^{TX})$ when
it is restricted 
on a fiber. Since the $G$-action is along the fiber, we have
\begin{align}\label{eq:2.03}
m^{TX}(K)=\nabla^{TX}_{\cdot}K^X\in \cC^{\infty}(W,\End(TX)).
\end{align}

Since the connection $\nabla^{TX}$ is $G$-invariant, 
from (\ref{eq:2.02})
(cf. \cite[(7.4)]{BGV} or \cite[(2.8)]{BG00}),
\begin{align}\label{eq:2.05} 
\nabla^{TX}_{\cdot}m^{TX}(K)+i_{K^X}R^{TX}=0.
\end{align}
We denote by $m^{\mS}(K)\in \End(\mE)$ by
\begin{align}\label{eq:2.06}
m^{\mS}(K):=\frac{1}{4}\la m^{TX}(K)e_i, e_j\ra c(e_i)c(e_j).
\end{align}
If $TX$ is spin, $m^{\mS}(K)$ is just the moment map
of the spinor.
Set
\begin{align}\label{eq:2.07}
m^{\mE/\mS}(K):=m^{\mE}(K)-m^{\mS}(K).
\end{align}

From (\ref{eq:1.28}), we set (cf. \cite[(2.30)]{BG00})
\begin{align}\label{eq:2.08} 
R_K^{TX}=R^{TX}-2i\pi\, m^{TX}(K),\quad 
R_K^{\mE/\mS}=R^{\mE/\mS}-2i\pi\, m^{\mE/\mS}(K).
\end{align}
Then $R_K^{TX}$ (resp. $R_K^{\mE/\mS}$) is called the equivariant 
curvature of $TX$ (resp. equivariant twisted curvature of $\mE$).

Let $Z(g)\subset G$ be the centralizer of $g\in G$ with
 Lie algebra $\mathfrak{z}(g)$. Then in the sense of the 
adjoint action,
\begin{align}\label{eq:2.09}
\mathfrak{z}(g)=\{K\in \mathfrak{g}: g.K=K\}.
\end{align}

We fix $g\in G$ from now on.
In the sequel, we always take $K\in \mathfrak{z}(g)$. 
Put
\begin{align}\label{eq:2.10}
W^K=\{x\in W: K^X(x)=0 \}.
\end{align}
Then $W^K$, which is the fixed point set of the group generated 
by $K$, is a totally geodesic submanifold along each fiber $X$. 
Set
\begin{align}\label{eq:2.11}
W^{g,K}=W^g\cap W^K.
\end{align}
Then $W^{g,K}$ is also a totally geodesic submanifold along
each fiber $X$. 
Moreover, if $K_0\in \mathfrak{z}(g)$ and $z\in \R$,
for $z$ small enough, we have
\begin{align}\label{eq:2.12}
W^{g,zK_0}=W^{ge^{zK_0}}.
\end{align}

Since the $G$-action is trivial on $B$, $W^K\rightarrow B$ and 
$W^{g,K}\rightarrow B$ are fibrations with compact fiber $X^K$ 
and $X^{g,K}$.
As in (\ref{eq:1.24}), by an abuse of notation, we will often 
simply denote by 
\begin{align}\label{eq:2.13}
\dim X^{g,K}=\ell'.
\end{align}

Observe that $m^{TX}(K)|_{X^g}$ acts on $TX^g$ and $N_{X^g/X}$. 
Also it preserves the splitting (\ref{eq:1.30}). Let 
$m^{TX^g}(K)$ and $m^{N(\theta)}(K)$ be the restrictions of 
$m^{TX}(K)|_{X^g}$ to $TX^g$ and $N(\theta)$. 
We define the corresponding equivariant curvatures 
$R_K^{TX^g}$, $R^{N(\theta)}_K$ as in (\ref{eq:2.08}). 

For $K\in \mathfrak{z}(g)$ with $|K|$ small enough, comparing 
with (\ref{eq:1.31}), set
\begin{multline}\label{eq:2.14}
\widehat{\mathrm{A}}_{g,K}(TX,\nabla^{TX})
=\mathrm{det}^{\frac{1}{2}}\left(\frac{\frac{ i}{4\pi}
R_K^{TX^g}}{\sinh\left(\frac{ i}{4\pi}R_K^{TX^g}\right)}
\right)
\\
\cdot\prod_{k>0}\left( i^{\frac{1}{2}\dim 
N(\theta)}\mathrm{det}^{\frac{1}{2}}\left(1-g\exp\left(
\frac{ i}{2\pi}R_K^{N(\theta)}\right)\right)\right)^{-1}
\in \Omega^{2\bullet}(W^g,\C).
\end{multline}	
Note that $W$ compact and $|K|$ small guarantee that 
the denominator in (\ref{eq:2.14}) is invertible. 
Comparing with (\ref{eq:1.32}), set
\begin{align}\label{eq:2.15}
\ch_{g,K}(\mE/\mS, \nabla^{\mE}):=\frac{2^{(n-\ell)/2}}{
\mathrm{det}^{1/2}(1-g|_{N_{X^g/X}})}\tr^{\mE/\mS}\left[
\sigma_{n-\ell}(g^{\mE})\exp\left(-\frac{R_K^{\mE/\mS}|_{W^g}}{
2i\pi}\right)\right].
\end{align}	
As in (\ref{eq:1.33}), if $TX$ has a $G$-equivariant spin structure, 
$\ch_{g,K}(\mE/\mS, \nabla^{\mE})$ is just the equivariant 
infinitesimal Chern character in \cite[Definition 2.7]{BG00},
\begin{align}\label{eq:2.16}
\ch_{g,K}(E, \nabla^{E})=\tr^{E}\left[g
\exp\left(-\frac{R_K^E|_{W^g}}{
	2i\pi}\right)\right]\in \Omega^{2\bullet}(W^g,\C),
\end{align}
where $m^E(K)=\nabla_{K^X}^E-\mL_K$, $R_K^E:=R^E-2i\pi m^E(K)$ 
as in (\ref{eq:2.02}) and (\ref{eq:2.08}).

Set
\begin{align}\label{eq:2.17} 
d_K=d-2i\pi\ i_{K^X}.
\end{align}
Then by (\ref{eq:2.05}) (cf. \cite[Theorem 7.7]{BGV}),
\begin{align}\label{eq:2.18} 
d_K\widehat{\mathrm{A}}_{g,K}(TX,\nabla^{TX})=0,
\quad d_K\ch_{g,K}(\mE/\mS, \nabla^{\mE})=0.
\end{align}

Recall that $\mathbb{B}_t$ is the rescaled Bismut superconnection
in (\ref{eq:1.22}). Set
\begin{align}\label{eq:2.19} 
\mathbb{B}_{K,t}=\mathbb{B}_{t}+\frac{c(K^X)}{4\sqrt{t}}.
\end{align}
Then $\mathbb{B}_{K,t}^2$ is a $2$nd-order 
elliptic differential operator along the fiber $X$ 
acting on $\Lambda(T^*B)\widehat{\otimes}\mathbb{E}$. 
If the base $B$ is a point, then the operator $\mathbb{B}_{K,t}$
is $\sqrt{t}D+\frac{c(K^X)}{4\sqrt{t}}$, and it was introduced
by Bismut \cite{Bi85} in his heat kernel proof of the Kirillov
formula for the equivariant index. As observed by Bismut 
\cite[\S 1d), \S 3b)]{Bi86} (cf. also \cite[\S 10.7]{BGV}),
its square plus $\mL_{K^X}$ is the square of the Bismut
superconnection for a fibration with compact structure group,
by replacing $K^X$  by the curvature of the fibration. Thus
we can roughly interpret $\mathbb{B}_{K,t}$ as the Bismut 
superconnection by extending our fibration by a fibration with
compact structure group.

Now we state the families version of the Kirillov formula and 
delayed a heat kernel proof of it to Section 5.

\begin{thm}\label{thm:2.01}
	For any $K\in \mathfrak{z}(g)$ and $|K|$ small,
\begin{itemize}
	\item { if $n$ is even, for $t>0$,  the differential 
	form 
	$$
	\psi_B\tr_s\left[g\exp\left(-\mathbb{B}_{K,t}^2- 
	\mL_K\right)\right]\in \Omega^{\mathrm{even}}(B, \C)
	$$
	is closed, the cohomology class defined by it is independent of $t$ and 
	\begin{align}\label{eq:2.20}
	\lim_{t\rightarrow 0}\psi_B\tr_s
	\left[g\exp\left(-\mathbb{B}_{K,t}^2-\mL_{K}\right)\right]
	=\int_{X^g}\widehat{\mathrm{A}}_{g,K}(TX,\nabla^{TX})\,
	\ch_{g,K}(\mE/\mS,\nabla^{\mE}).
	\end{align}
    }
\item { if $n$ is odd, for $t>0$,  the differential 
	form 
	$$
	\psi_B\tr^{\mathrm{odd}}\left[g\exp\left(-
	\mathbb{B}_{K,t}^2- \mL_K\right)\right]\in 
	\Omega^{\mathrm{odd}}(B, \C)
	$$
	is closed, the cohomology class defined by it is independent of $t$ and 
	\begin{align}\label{eq:2.21}
	\lim_{t\rightarrow 0}\psi_B\tr^{\mathrm{odd}}
	\left[g\exp\left(-\mathbb{B}_{K,t}^2-\mL_{K}\right)\right]
	=\int_{X^g}\widehat{\mathrm{A}}_{g,K}(TX,\nabla^{TX})\,
	\ch_{g,K}(\mE/\mS,\nabla^{\mE}).
	\end{align}
}
\end{itemize}

\end{thm}

If $B$ is a point and $g=e$, this heat kernel proof of the 
Kirillov formula is given by Bismut in \cite{Bi85} (see also 
\cite[Theorem 8.2]{BGV}). 
If $B$ is a point, (\ref{eq:2.20}) is established in 
\cite{BG00}. For $g=e$, (\ref{eq:2.20}) is 
obtained in \cite{Wa14}. 

\subsection{Equivariant infinitesimal $\eta$-forms: 
Theorem \ref{thm:0.1}}\label{s0202}	

For $t>0$, set
\begin{align}\label{eq:2.22}
\mB_{K,t}=\mathbb{B}_{K,t}+
dt\wedge \frac{\partial}{\partial t}.
\end{align}
Then by (\ref{eq:2.19}),
\begin{align}\label{eq:2.23}
\mB_{K,t}^2=\mathbb{B}_{K,t}^2
+ dt\wedge \frac{\partial \mathbb{B}_{K,t}}{\partial
	t}
=\left(\mathbb{B}_{t}+\frac{c(K^X)}{4\sqrt{t}}
\right)^2+ dt\wedge \frac{\partial}{\partial
t}\left(\mathbb{B}_{t}+\frac{c(K^X)}{4\sqrt{t}}\right).
\end{align}

\begin{thm}\label{thm:2.02} 
There exist $\beta>0$, 
$\delta,\delta'>0$, $C>0$, such that if $K\in \mathfrak{z}(g)$,
 $z\in \C$, $|zK|\leq \beta$, 
	
a) for any $t\geq 1$,
\begin{align}\label{eq:2.24} 
\left|\left\{\wi{\tr}\Big[g\exp\big(-\mB_{zK,t}^2-z\mL_{K}\big)
\Big]\right\}^{dt}\right|\leq \frac{C}{t^{1+\delta}};
\end{align}
	
b) for any $0<t\leq 1$,
\begin{align}\label{eq:2.25} 
\left|\left\{\wi{\tr}\Big[g\exp\big(-\mB_{zK,t}^2-z\mL_{K}\big)
\Big]\right\}^{dt}\right|\leq C\,t^{\delta'-1}.
\end{align}
\end{thm}

We delay the proof of Theorem \ref{thm:2.02} to Section 5. 

$\bullet$ If $n=\dim X$ is even, 
then for $t>0$, as $\mathbb{B}_{K,t}$ commutes with $g$, $\mL_K$,
by \cite[Lemma 9.15]{BGV},
\begin{align}\label{eq:2.26b}
d^B\tr_s\left[g
\exp(-\mathbb{B}_{K,t}^{2}-\mL_K)\right]
=\tr_s\Big[\big[\mathbb{B}_{K,t}, g
\exp(-\mathbb{B}_{K,t}^{2}-\mL_K)\big]\Big]
=0.
\end{align}
As in (\ref{eq:1.37}) 
(cf. \cite[Proposition 8.11 and Theorem 9.19]{BGV}), we have
\begin{align}\label{eq:2.26}
\lim_{t\rightarrow +\infty}\psi_B\tr_s\Big[g\exp\big(
-\mathbb{B}_{K,t}^2-\mL_{K}
\big)\Big]
=\ch_{ge^K}(\Ker(D), \nabla^{\Ker(D)}).
\end{align}
As in (\ref{eq:1.38}),
\begin{multline}\label{eq:2.27}
\frac{\partial}{\partial t}\tr_s\left[g
\exp(-\mathbb{B}_{K,t}^{2}-\mL_K)\right]
=-d^B\tr_s\left[g\frac{\partial \mathbb{B}_{K,t}
}{\partial t}
\exp(-\mathbb{B}_{K,t}^{2}-\mL_K)\right]
\\
=d^B\left\{\tr_s\left[g
\exp(-\mB_{K,t}^{2}-\mL_K)\right]\right\}^{dt}.
\end{multline}
Thus from (\ref{eq:2.27}),  for $0<\var<T<+\infty$,
\begin{multline}\label{eq:2.28}
\tr_s\left[g
\exp(-\mathbb{B}_{K,T}^{2}-\mL_K)\right]-\tr_s\left[g
\exp(-\mathbb{B}_{K,\var}^{2}-\mL_K)\right]
\\
=d^B\int_{\var}^T\left\{\tr_s\left[g
\exp(-\mB_{K,t}^{2}-\mL_K)\right]\right\}^{dt}dt.
\end{multline}
In this case, for $|K|\leq \beta$, by Theorem \ref{thm:2.02}, 
the equivariant infinitesimal $\eta$-form is defined by
\begin{multline}\label{eq:2.29}
\tilde{\eta}_{g,K}=-\int_0^{+\infty}\frac{1}{2  i
	\sqrt{\pi}}\psi_{B}
\left\{\tr_s\left[g
\exp(-\mB_{K,t}^{2}-\mL_K)\right]\right\}^{dt}dt
\\
=\int_0^{+\infty}\frac{1}{2  i
	\sqrt{\pi}}\psi_{B}
\tr_s\left[g\frac{\partial \mathbb{B}_{K,t}
}{\partial t}
\exp(-\mathbb{B}_{K,t}^{2}-\mL_K)\right]dt
\in 
\Omega^{\mathrm{odd}}(B,\C).
\end{multline}
By (\ref{eq:2.20}), (\ref{eq:2.28})
and (\ref{eq:2.29}), we have
\begin{align}\label{eq:2.30}
d^B\tilde{\eta}_{g,K}=\int_{X^g}\widehat{
	\mathrm{A}}_{g,K}(TX,\nabla^{TX})
\ch_{g,K}(\mE/\mS, \nabla^{\mE})-\ch_{ge^K}
(\Ker (D), \nabla^{\Ker (D)}).
\end{align}

$\bullet$ If $n$ is odd, 
then for $t>0$, as $\mathbb{B}_{K,t}$ commutes with $g$, $\mL_K$,
again by the argument in \cite[Lemma 9.15]{BGV},
\begin{align}\label{eq:2.31b}
d^B\tr^{\mathrm{odd}}\left[g
\exp(-\mathbb{B}_{K,t}^{2}-\mL_K)\right]
=\tr^{\mathrm{even}}\Big[\big[\mathbb{B}_{K,t}, g
\exp(-\mathbb{B}_{K,t}^{2}-\mL_K)\big]\Big]
=0.
\end{align}
As the same argument in (\ref{eq:1.43}),
\begin{align}\label{eq:2.31}
\lim_{t\rightarrow +\infty}\tr^{\mathrm{odd}}
\Big[g\exp\big(
-\mathbb{B}_{K,t}^2-\mL_{K}
\big)\Big]=0.
\end{align} 
Comparing with (\ref{eq:1.38}) and (\ref{eq:2.27}), we have
\begin{multline}\label{eq:2.32}
\frac{\partial}{\partial t}\tr^{\mathrm{odd}}\left[g
\exp(-\mathbb{B}_{K,t}^{2}-\mL_K)\right]
=-d^B\tr^{\mathrm{even}}\left[g\frac{\partial 
\mathbb{B}_{K,t}}{\partial t}
\exp(-\mathbb{B}_{K,t}^{2}-\mL_K)\right]
\\
=d^B\left\{\tr^{\mathrm{odd}}\left[g
\exp(-\mB_{K,t}^{2}-\mL_K)\right]\right\}^{dt}.
\end{multline}
From Theorem \ref{thm:2.02},
in this case, for $|K|\leq \beta$,
the equivariant infinitesimal $\eta$-form is defined by
\begin{multline}\label{eq:2.33}
\tilde{\eta}_{g,K}=-\int_0^{+\infty}\frac{1}{
	\sqrt{\pi}}\psi_{B}
\left\{\tr^{\mathrm{odd}}\left[g
\exp(-\mB_{K,t}^{2}-\mL_K)\right]\right\}^{dt} dt
\\
=\int_0^{+\infty}\frac{1}{
	\sqrt{\pi}}\psi_{B}
\tr^{\mathrm{even}}\left[g\frac{\partial 
	\mathbb{B}_{K,t}}{\partial t}
\exp(-\mathbb{B}_{K,t}^{2}-\mL_K)\right] dt
\in 
\Omega^{\mathrm{even}}(B,\C).
\end{multline}
As in (\ref{eq:1.47}), by (\ref{eq:2.21}), (\ref{eq:2.31}), 
(\ref{eq:2.32}) and (\ref{eq:2.33}), we get
\begin{align}\label{eq:2.34}
d^B\tilde{\eta}_{g,K}=\int_{X^g}\widehat{
	\mathrm{A}}_{g,K}(TX,\nabla^{TX})
\ch_{g,K}(\mE/\mS, \nabla^{\mE}).
\end{align}

\begin{defn}\label{defn:2.03} 
For $K\in \mathfrak{z}(g)$, $|K|\leq \beta$,   
determined in Theorem \ref{thm:2.02}, under
Assumption \ref{ass:1.02},
the equivariant infinitesimal Bismut-Cheeger
$\eta$-form is defined by
\begin{align}\label{eq:2.35} 
\wi{\eta}_{g,K}=-\int_{0}^{+\infty}\left\{\psi_{\R\times B}
\wi{\tr}\Big[g\exp\left(-\mB_{K,t}^2-\mL_{K}
\right) \Big]\right\}^{dt}dt.
\end{align}
\end{defn}
By (\ref{eq:0.05}) and (\ref{eq:1.34}),  (\ref{eq:2.35})
is a reformulation of (\ref{eq:2.29}) and (\ref{eq:2.33}).
From (\ref{eq:2.30}) and (\ref{eq:2.34}), we establish the
first part of Theorem \ref{thm:0.1}.

Remark that the compactness of $B$ guarantees the existence of 
the constant $\beta>0$ in Definition \ref{defn:2.03}.

From (\ref{eq:2.29}) and (\ref{eq:2.33}),
it is obvious that if $K=0$, $\wi{\eta}_{g,K}=\wi{\eta}_{g}$ 
in (\ref{eq:1.48}).

From the Duhamel's formula (cf. e.g., \cite[Theorem 2.48]{BGV}), 
we have
\begin{align}\label{eq:2.36}
\frac{\partial}{\partial\bar{z}}\wi{\tr}\Big[g\exp\big(
-\mB_{zK,t}^2-z\mL_{K}\big)\Big]
=-\wi{\tr}\left[g\frac{\partial (\mB_{zK,t}^2+z\mL_{K})}{\partial 
	\bar{z}}\exp\big(-\mB_{zK,t}^2-z\mL_{K}\big)\right]=0.
\end{align}
Thus, $\wi{\tr}\Big[g\exp\big(-\mB_{zK,t}^2-z\mL_{K}\big)\Big]$ 
is $\cC^{\infty}$ on $t>0$ and holomorphic on $z\in \C$.

We fix $K\in \mathfrak{z}(g)$. 	Thus for 
$0<\var<T<+\infty$, the function 
$$\int_{\var}^{T}\left\{
\psi_{\R\times B}\wi{\tr}\Big[g\exp\left(-\mB_{zK,t}^2-z\mL_{K}
\right) \Big]\right\}^{dt}dt$$ 
is holomorphic on $z$.
By Theorem \ref{thm:2.02} and the dominated convergence theorem, 
we have
\begin{align}\label{eq:2.37} 
\wi{\eta}_{g,zK}:=-\int_{0}^{+\infty}\left\{\psi_{\R\times B}
\wi{\tr}\Big[g\exp\left(-\mB_{zK,t}^2-z\mL_{K}
\right) \Big]\right\}^{dt}dt
\end{align}
is holomorphic on $z\in \C$, $|zK|< \beta$. Thus we get the last 
part of Theorem \ref{thm:0.1}.

The proof of Theorem \ref{thm:0.1} is completed. 

\section{Comparison of two equivariant $\eta$-forms}\label{s03}

In this section, we state our main result. We use the same 
notations and assumptions in Sections 1 and 2.

Let $\vartheta_K\in T^*X$ be the $1$-form which is dual to 
$K^X$ by the metric $g^{TX}$, i.e., for any $U\in TX$,
\begin{align}\label{eq:3.01}
\vartheta_K(U)=\la K^X, U\ra.
\end{align}
We identify $\vartheta_K$ to a vertical 
$1$-form on $W$, i.e., to a $1$-form which vanishes on $T^HW$. 
Then by (\ref{eq:2.17}) and (\ref{eq:3.01}), we have
\begin{align}\label{eq:3.02}
d_K\vartheta_{K}=d\vartheta_K-2i\pi\,|K^X|^2.
\end{align}
Let $d^X$ be the exterior differential operator along the 
fiber $X$. By (\ref{eq:2.03}) and (\ref{eq:3.01}) 
(cf. \cite[Lemma 7.15 (1)]{BGV}), for $U,U'\in TX$, we have
\begin{align}\label{eq:3.03} 
d^X\vartheta_K(U,U')=2\la \nabla^{TX}_UK^X,U'\ra
=2\la m^{TX}(K)U,U'\ra.
\end{align}

From (\ref{eq:1.11}) and (\ref{eq:1.12}), set
\begin{align}\label{eq:3.04} 
\wi{T}=2T(f_{p}^H,e_i)f^p\wedge e^i\wedge 
+\frac{1}{2}T(f_{p}^H,f_q^H)f^p\wedge f^q\wedge.
\end{align}
From \cite[Proposition 10.1]{BGV} or 
\cite[(3.61) and (3.94)]{BG04},
\begin{align}\label{eq:3.05} 
d\vartheta_K=d^X\vartheta_K+\la \wi{T},K^X\ra
=d^X\vartheta_K+\vartheta_K(\wi{T}).
\end{align}

For $K\in \mathfrak{z}(g)$, $|K|$ small, $v>0$, set
\begin{align}\label{eq:3.07} \begin{split}
	\alpha_K&=\widehat{\mathrm{A}}_{g,K}(TX,\nabla^{TX})
\ch_{g,K}(\mE/\mS,\nabla^{\mE})\in \Omega^{2\bullet}\big(W^g,
\det N_{X^g/X}\big),\\
\wi{e}_v&=-\int_{X^g}\frac{\vartheta_K}{8vi\pi}\exp\left(\frac{d_{K}
	\vartheta_K}{8vi\pi}\right)\alpha_K\in \Omega^{\bullet}(B,\C).
\end{split}\end{align}		 
Note that if $W^{g,K}=W^g$, as $\vartheta_K=0$ on $W^{g,K}$, 
we get $\wi{e}_v=0$.

\begin{lemma}\label{lem:3.01}
If $W^{g,K}\varsubsetneq W^g$. Then $\wi{e}_v=\mathcal{O}(v^{-1})$ 
as $v\rightarrow +\infty$ and  $\wi{e}_v=\mathcal{O}(v^{1/2})$ as 
$v\rightarrow 0$.
\end{lemma}
\begin{proof}
By (\ref{eq:3.02}) and (\ref{eq:3.07}), we have
\begin{align}\label{eq:3.08} 
\wi{e}_v=-\sum_{j=0}^{\lfloor \dim W^g/2\rfloor}\frac{1}{j!}
\left(\frac{1}{2i\pi}\right)^{j+1}
\int_{X^g}\frac{\vartheta_K}{4v}\left(\frac{d\vartheta_K}{4v}
\right)^{j}\exp\left(-\frac{|K^X|^2}{4v} \right)\cdot\alpha_K.
\end{align}
Thus when $v\rightarrow +\infty$, 
$\wi{e}_v=\mathcal{O}(v^{-1})$.

For $v\rightarrow 0$, we follow the argument in the proof of 
\cite[Theorem 1.3]{Bi86a}.	
For $x\in W^g$, if $K^X_x\neq 0$, when
$v\rightarrow 0$, the integral term in (\ref{eq:3.08}) at $x$ is of
exponential decay. So the integral in (\ref{eq:3.08}) could be 
localized on a neighbourhood of $W^{g,K}$. 

Let $N_{X^{g,K}/X^g}$ be the normal bundle of $W^{g,K}$ in $W^g$,
and we identify it as the orthogonal complement of 
$TX^{g,K}=TX^g|_{W^{g,K}}\cap TX^K|_{W^{g,K}}$
in $TX^g|_{W^{g,K}}$. 
Recall that as $K^X$ is a Killing vector field, for any 
$b\in B$, $X^{g,K}_b$ is totally geodesic in $X^g_b$, 
and as the same argument in Section \ref{s0201}, 
$\nabla^{TX^g}$,
$m^{TX}(K)$ preserve the splitting
\begin{align}\label{eq:3.18b} 
TX^g=TX^{g,K}\oplus N_{X^{g,K}/X^g} \quad \text{ on }
W^{g,K}
\end{align}
and $m^{TX}(K)=0$ on $TX^{g,K}$. In particular,
\begin{align}\label{eq:3.09b}
m^{N_{X^{g,K}/X^g}}(K)=m^{TX}(K)|_{N_{X^{g,K}/X^g}}\in 
\End(N_{X^{g,K}/X^g}) \text{ is skew-adjoint and invertible.}
\end{align}
 Combining with (\ref{eq:2.05}),
it implies that $N_{X^{g,K}/X^g}$ is orientable, and we fix 
an orientation. Then the orientations on $TX$, $N_{X^{g,K}/X^g}$
induce the identifications over $W^{g,K}$,
\begin{align}\label{eq:3.10b}
\det(N_{X^{g}/X})\simeq \det(TX^g)\simeq \det(TX^{g,K}).
\end{align} 

Given 
$\var>0$, let $\mU_{\var}''$ be the $\var$-neighborhood of 
$W^{g,K}$ in $N_{X^{g,K}/X^g}$. There exists $\var_0$ such 
that for $0<\var\leq \var_0$, the fiberwise exponential map 
$(y,Z)\in N_{b, X^{g,K}/X^g}\rightarrow \exp_{y}^{X}(Z)\in X_b^g$ 
is a diffeomorphism from $\mU_{\var}''$ into the tubular 
neighborhood $\mV_{\var}''$ of $W^{g,K}$ in $W^g$. 
We denote $\wi{\mV}_{\var}''$ the fiber of the 
fibration $\mV_{\var}''
\rightarrow B$.
With this identification, let $\bar{k}(y,Z)$ be the function such that
\begin{align}\label{eq:3.09}
dv_{X^g}(y,Z)=\bar{k}(y,Z)dv_{X^{g,K}}(y)dv_{N_{X^{g,K}/X^g}}(Z).
\end{align}
Here $dv_{X^g}\in \Lambda^{\max}(T^*X^g)\otimes \det(T^*X^g)$, 
$dv_{X^{g,K}}\in \Lambda^{\max}(T^*X^{g,K})
\otimes \det(T^*X^{g,K})$ 
are the Riemannian volume forms
of $X^g$, $X^{g,K}$ and $dv_{N_{X^{g,K}/X^g}}$ is the Euclidean
volume form on $N_{X^{g,K}/X^g}$.

Let $e^1,\cdots,e^{\ell}$ be a locally orthonormal frame of
$T^*X^g$. For $\beta\in \Omega^{\bullet}
\big(W^g,\det (N_{X^g/X})\big)$, 
let $[\beta]^{\max}$ be the 
coefficient of $e^1\wedge\cdots\wedge e^{\ell}\otimes 
\widehat{e^1\wedge\cdots\wedge e^{\ell}}$ of $\beta$,
here $\widehat{e^1\wedge\cdots\wedge e^{\ell}}$ means
the local frame of $\det(N_{X^g/X})$ induced by 
$e^1\wedge\cdots\wedge e^{\ell}$ via (\ref{eq:3.10b}).
Consider the dilation $\delta_v$, $v>0$, of $N_{X^{g,K}/X^g}$
by $\delta_v(y,Z)=(y,\sqrt{v}Z)$.
We have
\begin{multline}\label{eq:3.10}
\int_{\wi{\mV}_{\var}''}\frac{\vartheta_K}{4v}\left(\frac{d\vartheta_K
}{4v}\right)^{j}\exp\left(-\frac{|K^X|^2}{4v} \right)\alpha_K
\\
=\int_{X^{g,K}}\int_{Z\in N_{X^{g,K}/X^g}, |Z|< \var}\left[\frac{
\vartheta_K|_{(y,Z)}}{4v}\left(\frac{d\vartheta_K
|_{(y,Z)}}{4v}\right)^{j}
\exp\left(-\frac{|K^X(y,Z)|^2}{4v}\right)\alpha_K(y,Z)
\right]^{\max}
\\
\cdot \bar{k}(y,Z)dv_{X^{g,K}}(y)dv_{N_{X^{g,K}/X^g}}(Z)
\\
=\int_{X^{g,K}}\int_{Z\in N_{X^{g,K}/X^g}, |Z|< \var/\sqrt{v}}
\bigg[\frac{\delta_v^*\vartheta_K|_{(y,Z)}}{4v}\left(\frac{
d\delta_v^*\vartheta_K|_{(y,Z)}}{4v}\right)^{j}\exp\left(-
\frac{|K^X(y,\sqrt{v}Z)|^2}{4v}\right)
\\
\cdot\delta_v^*\alpha_K|_{(y,Z)}\Bigg]^{\max}
\cdot \bar{k}(y,\sqrt{v}Z)dv_{X^{g,K}}(y)
dv_{N_{X^{g,K}/X^g}}(Z).
\end{multline}
	
Let $\nabla^{N_{X^{g,K}/X^g}}$ be the connection on 
$N_{X^{g,K}/X^g}$ induced by $\nabla^{TX}$ 
as explained after (\ref{eq:1.30}). 
Let $\pi_N:N_{X^{g,K}/X^g}
\rightarrow W^{g,K}$ be the obvious projection. With respect 
to $\nabla^{N_{X^{g,K}/X^g}}$, we have the canonical splitting 
of bundles over $N_{X^{g,K}/X^g}$,
\begin{align}\label{eq:3.11} 
TN_{X^{g,K}/X^g}=T^HN_{X^{g,K}/X^g}
\oplus \pi_N^*N_{X^{g,K}/X^g}.
\end{align}
By (\ref{eq:1.04}) and (\ref{eq:3.11}), we have
\begin{align}\label{eq:3.11b} 
T^HN_{X^{g,K}/X^g}\simeq \pi_N^*TW^{g,K}
\simeq 
\pi_N^*(T^HW\oplus TX^{g,K}).
\end{align}
On $N_{X^{g,K}/X^g}$, by (\ref{eq:3.11}) and (\ref{eq:3.11b}),
 we have
\begin{multline}\label{eq:3.12} 
\Lambda(T^*N_{X^{g,K}/X^g})=\Lambda
(T^{H*}N_{X^{g,K}/X^g})\widehat{\otimes} \pi_N^*\Lambda
(N_{X^{g,K}/X^g}^*)
\\
\simeq \pi_N^*\left(
\Lambda(T^*W^{g,K})\widehat{\otimes} \Lambda
(N_{X^{g,K}/X^g}^*)\right).
\end{multline}

For $y\in W^{g,K}$ fixed, we take $Y_1, Y_1'\in T_yW^{g,K}$, 
$Y^V, Y'^V\in N_{X^{g,K}/X^g,y}$, then 
$Y=Y_1+Y^V$, $Y'=Y_1'+Y'^V$
are sections of $TN_{X^{g,K}/X^g}$ along $N_{X^{g,K}/X^g,y}$ 
under our identification (\ref{eq:3.11}), i.e.,
\begin{align}\label{eq:3.13}
Y_{(y,Z)}=Y_1^H(y,Z)+Y^V, \quad Y_{(y,Z)}'=Y_1'^H(y,Z)+Y'^V.
\end{align}
Here $Y_1^H, Y_1'^H\in T^HN_{X^{g,K}/X^g}$ are the lifts of 
$Y_1, Y_1'$.

Let $\theta_0$ be the one form on total space $\mN$ of 
$N_{X^{g,K}/X^g}=N_{W^{g,K}/W^g}$ given by 
\begin{align}\label{eq:3.16}
\theta_0(Y)_{(y,Z)}=\la m^{TX}(K)Z, Y^V\ra_{y}
\quad \text{for}\ 
Y=Y_1^H+Y^V\in T^HN_{X^{g,K}/X^g}
\oplus (\pi_N^*N_{X^{g,K}/X^g}).
\end{align}
By \cite[Lemma 7.15 (2)]{BGV},
we have
\begin{align}\label{eq:3.17}
\frac{1}{v}\delta_v^*\vartheta_K=\theta_0+\mO(v^{1/2}).
\end{align}
From (\ref{eq:3.17}), we get
\begin{align}\label{eq:3.18}
\frac{1}{v}\delta_v^*d\vartheta_K=
\frac{1}{v}d\delta_v^*\vartheta_K=d\theta_0+\mO(v^{1/2}).
\end{align}
As in the argument before \cite[p218, Lemma 7.16]{BGV}, 
by (\ref{eq:3.18b}), we calculate that for $(y,Z)\in 
N_{X^{g,K}/X^g}$,
\begin{align}\label{eq:3.21} 
d\theta_0(Y,Y')_{(y,Z)}=2\la m^{TX}(K)Y^V, Y'^V\ra_{y}
-\la R^{TX}(Y_1^H,Y_1'^H)(m^{TX}(K)Z), Z\ra_{y}.
\end{align}
By (\ref{eq:2.03}) and (\ref{eq:2.11}), for $y\in W^{g,K}$,
\begin{align}\label{eq:3.19}
\frac{1}{v}|K^X(y,\sqrt{v}Z)|^2=|m^{TX}(K)Z|^2+\mO(v^{1/2}).
\end{align}
From (\ref{eq:3.10}), (\ref{eq:3.17}), (\ref{eq:3.18}) and (\ref{eq:3.19}),
for any $\alpha\in \Omega^{\bullet}
\big(W^g,\det (N_{X^g/X})\big)$, as $v\rightarrow 0$,
\begin{multline}\label{eq:3.20}
\int_{\wi{\mV}_{\var}''}\frac{\vartheta_K}{4v}\left(\frac{d\vartheta_K
}{4v}\right)^{j}\exp\left(-\frac{|K^X|^2}{4v} \right)\alpha
\\
=\int_{X^{g,K}}\alpha_y\int_{Z\in N_{X^{g,K}/X^g}}
\frac{\theta_0}{4}\left(\frac{d\theta_0
}{4}\right)^{j}\exp\left(-\frac{|m^{TX}(K)Z|^2}{4} \right)
+\mO(v^{1/2}).
\end{multline}
From (\ref{eq:3.21}), $d\theta_0$ is an even polynomial in $Z$. 
However from (\ref{eq:3.16}), $\theta_0$
is linear in $Z$. Thus the last integral in 
(\ref{eq:3.20}) is zero. Therefore, as $v\rightarrow 0$, we have
\begin{align}\label{eq:3.25}
\int_{X^g}\frac{\vartheta_K}{4v}\left(\frac{d\vartheta_K
}{4v}\right)^{j}\exp\left(-\frac{|K^X|^2}{4v} \right)\alpha_K=
\mathcal{O}(v^{1/2}).
\end{align}
The proof of Lemma \ref{lem:3.01} is completed.
\end{proof}	

Remark that when $B$ is a point, for $g=e$, Lemma \ref{lem:3.01} 
is proved in \cite[Proposition 2.2]{Go09}.

From Lemma \ref{lem:3.01} and (\ref{eq:3.07}), the following 
integral is well-defined,
\begin{align}\label{eq:3.26} 
\mM_{g,K}:=\int_0^{+\infty}\wi{e}_v\frac{dv}{v}.
\end{align}

\begin{prop}\label{prop:3.02} 
	For any $K_{0}\in \mathfrak{z}(g)$, there exists $\beta>0$
	such that for $K=zK_{0}$, $-\beta<z<\beta$, we have
\begin{multline}\label{eq:3.27a}
d^B\mM_{g,K}=\int_{X^{g}}\widehat{\mathrm{A}}_{g,K}(TX,\nabla^{TX})
\ch_{g,K}(\mE/\mS,\nabla^{\mE})
\\
- \int_{X^{ge^K}}
\widehat{\mathrm{A}}_{ge^K}(TX,\nabla^{TX})
\ch_{ge^K}(\mE/\mS,\nabla^{\mE}).
\end{multline}	
And there exist $c_j(K)\in \Omega^{\bullet}(B,\C)$ 
for $1\leq j\leq \lfloor(\dim W^g+1)/2\rfloor$ such that
$\mM_{g,tK}$ is smooth on $|t|< 1$, $t\neq 0$ and 
as $t\rightarrow 0$, we have
\begin{align}\label{eq:3.27} 
\mM_{g,tK}=\sum_{j=1}^{\lfloor(\dim 
	W^g+1)/2\rfloor}c_j(K)t^{-j}+\mathcal{O}(t^0).
\end{align}
Moreover, $t^{\lfloor(\dim W^g+1)/2\rfloor}\mathcal{M}_{g,tK}
$ is real analytic in $t$ for $|t|<1$.
\end{prop}
\begin{proof} 
By (\ref{eq:2.17}), $d_K^2=-2i\pi\mL_K$, and $\vartheta_K$
is $K$-invariant. We know
\begin{align}\label{eq:3.22b}
\frac{\partial}{\partial v}\left(\exp\left(
\frac{d_K\vartheta_K}{2vi\pi}\right) \right)
=-\frac{1}{v^2}d_K\left(
\frac{\vartheta_K}{2i\pi}
\exp\left(\frac{d_K\vartheta_K}{2vi\pi}
\right)\right).
\end{align}

We define the corresponding equivariant curvature
$R^{N_{X^{g,K}/X^g}}_K$ as in (\ref{eq:2.08}) via (\ref{eq:3.18b}). 
By the proof of (\ref{eq:3.25}) and \cite[Theorem 1.3]{Bi86a}, 
we know that there exists $C>0$,
such that for any $v\in (0,1]$, $\alpha\in 
\Omega^{\bullet}(W^g,\det (N_{X^g/X}))=
\Omega^{\bullet}(W^g,o(TX^g))$,
\begin{align}\label{eq:3.23b}
\begin{split}
&\left|\int_{X^g}\exp\left(
\frac{d_K\vartheta_K}{2vi\pi}\right) \alpha
-\int_{X^{g,K}}\frac{i^{-(\ell-\ell')/2}\alpha}{
	\mathrm{det}^{1/2}\left(
	R^{N_{X^{g,K}/X^g}}_K/(2i\pi)\right)}\right|
\leq C\sqrt{v}\|\alpha\|_{\cC^1(W^g)},
\\
&\left|\int_{X^g}\frac{\vartheta_K}{2vi\pi}
\exp\left(\frac{d_K\vartheta_K}{2vi\pi}\right) \alpha
\right|
\leq C\sqrt{v}\|\alpha\|_{\cC^1(W^g)}.
\end{split}
\end{align}

Let $Q_K$ be the current on $W^g$ such that if
$\alpha\in 
\Omega^{\bullet}(W^g,\det (N_{X^g/X}))$, then 
\begin{align}\label{eq:3.24b} 
\int_{X^g}Q_K\alpha=-\int_0^{+\infty}\int_{X^g}
\frac{\vartheta_K}{2vi\pi}
\exp\left(\frac{d_K\vartheta_K}{2vi\pi}\right)\alpha
\frac{dv}{v}.
\end{align}
From (\ref{eq:3.08}),  the second equation of (\ref{eq:3.23b}), 
we know (\ref{eq:3.24b}) is well-defined.
From (\ref{eq:3.22b})-(\ref{eq:3.24b}),
the following equality of currents on $W^g$ holds 
(cf. \cite[Theorem 1.8]{B11a}):
\begin{align}\label{eq:3.25b}
d_KQ_K=1-\frac{i^{-(\ell-\ell')/2}\delta_{W^{g,K}}}{
	\mathrm{det}^{1/2}\left(
		R^{N_{X^{g,K}/X^g}}_K/(2i\pi)\right)},
\end{align}
where $\delta_{W^{g,K}}$ is the current of integration on $W^{g,K}$. 
From (\ref{eq:3.07}), (\ref{eq:3.26}) and (\ref{eq:3.24b}), we get
\begin{multline}\label{eq:3.26b} 
\mM_{g,K}=\int_{X^g}Q_K\alpha_K
\\
=-\int_0^{+\infty}\int_{X^g}\frac{\vartheta_K}{2vi\pi}
\exp\left(\frac{d_K\vartheta_K}{2vi\pi}
\right)\widehat{\mathrm{A}}_{g,K}(TX,\nabla^{TX})
\ch_{g,K}(\mE/\mS,\nabla^{\mE})\frac{dv}{v}.
\end{multline}

For $x\in W^g$, $K\in \mathfrak{z}(g)$, we have $K^X(x)\in T_xX^g$.
From \cite[(1.7)]{BGV}, for $\sigma\in \Omega^{\bullet}(W^g,o(TX^g))$,
using the sign convention in (\ref{eq:0.11}), we have
\begin{align}\label{eq:3.23}
d^B\int_{X^g}\sigma=\int_{X^g}d\sigma=\int_{X^g}d_K\sigma.
\end{align}

From (\ref{eq:2.11}), proceeding as the same calculation 
in the proof of \cite[Theorem 8.2]{BGV},
we get, as elements in $\Omega^{\bullet}(W^{g,K},
\det (N_{X^g/X}))$,
\begin{align}\label{eq:3.26a} 
\frac{\widehat{\mathrm{A}}_{g,K}(TX,\nabla^{TX})
	\ch_{g,K}(\mE/\mS,\nabla^{\mE})
}{i^{(\ell-\ell')/2}
	\mathrm{det}^{1/2}\left(
	R^{N_{X^{g,K}/X^g}}_K/(2i\pi)\right)}=
\widehat{\mathrm{A}}_{ge^K}(TX,\nabla^{TX})
\ch_{ge^K}(\mE/\mS,\nabla^{\mE}).
\end{align}

As $\alpha_K$ is $d_K$-closed, by (\ref{eq:2.12})
and (\ref{eq:3.24b})-(\ref{eq:3.26a}), we get (\ref{eq:3.27a}).

For $t\neq 0$, by (\ref{eq:3.26b}) and changing the variables
$v\mapsto vt^2$, we have
\begin{align}\label{eq:3.29} 
\mM_{g,tK}
=-\int_0^{+\infty}\int_{X^g}\frac{\vartheta_K}{2vi\pi t}
\sum_{k=0}^{\lfloor(\dim W^g-1)/2\rfloor}\left(\frac{(d
	\vartheta_K)^k}{(2vi\pi t)^k k! } \right)
\exp\left(-\frac{|K^X|^2}{v}\right)\alpha_{tK}
\frac{dv}{v}.
\end{align}
From the arguments in the proof of (\ref{eq:3.25}), we 
get (\ref{eq:3.27}). From (\ref{eq:2.14}), (\ref{eq:2.15}) 
and (\ref{eq:3.07}), we see that $\alpha_{tK}$ is real analytic 
on $t$ for $|t|<1$. Following 
the proof of (\ref{eq:3.25}), 
$$\int_0^{+\infty}\int_{X^g}\frac{\vartheta_K}{v}
\left(\frac{d\vartheta_K}{v}\right)^k
\exp\left(-\frac{|K^X|^2}{v}\right)\alpha_{tK}
\frac{dv}{v}$$
is uniformly absolutely integrable on $v$ for 
$|t|<1$. Thus $t^{\lfloor(\dim W^g+1)/2\rfloor}
\mathcal{M}_{g,tK}$ is real analytic on $t$ for $|t|<1$.

The proof of Proposition \ref{prop:3.02} is completed.
\end{proof}

From Proposition \ref{prop:3.02}, 
we could state our main result, Theorem 
\ref{thm:0.2} as follows.

\begin{thm}\label{thm:3.03} 
For any $g\in G$, $K_{0}\in \mathfrak{z}(g)$, there exists $\beta>0$
such that for $K=zK_{0}$, $-\beta<z<\beta$, $K\neq 0$, we have
\begin{align}\label{eq:3.30} 
\tilde{\eta}_{g,K}=\tilde{\eta}_{ge^K}+\mM_{g,K}\in 
\Omega^{\bullet}(B,\C)/d\Omega^{\bullet}(B,\C).
\end{align}
\end{thm}	

Observe that by (\ref{eq:2.37}), $\tilde{\eta}_{g,tK}$ is 
analytic on $t$ for $t$ small. By (\ref{eq:3.30}), 
when $t\rightarrow 0$, 
modulo exact forms, the 
singularity of $\tilde{\eta}_{ge^{tK}}$ is the same 
as that of $-\mM_{g,tK}$ in (\ref{eq:3.27}).

Note that Theorem \ref{thm:3.03} is compatible with (\ref{eq:1.42}),
(\ref{eq:1.47}), (\ref{eq:2.30}), (\ref{eq:2.34}) and (\ref{eq:3.27a}). 

\begin{rem}\label{rem:3.04}
For $K\in \mathfrak{z}(g)$,
 $M=\lfloor(\dim W^g-1)/2
\rfloor$, on $W^g\setminus \{K^X=0\}$, we have
\begin{multline}\label{eq:3.31}
Q_K=-\sum_{j= 0}^M
\frac{1}{j!}\left(\frac{1}{2i\pi}
\right)^{j+1}\int_0^{+\infty}\frac{\vartheta_K}{v}
\left(\frac{d\vartheta_K}{v}\right)^j
\exp\left(-\frac{|K^X|^2}{v}\right)\frac{dv}{v}
\\
=-\sum_{j= 0}^M
\frac{1}{j!}\left(\frac{1}{2i\pi}
\right)^{j+1}\frac{\vartheta_K}{|K^X|^2}
\frac{(d\vartheta_K)^j}{|K^X|^{2j}}
\int_0^{+\infty}v^je^{-v}dv
\\
=-\sum_{j= 0}^M  
\frac{\vartheta_K(d\vartheta_K)^j}{
	(2i\pi)^{j+1}|K^X|^{2j+2}}
=-\frac{\vartheta_K}{2i\pi|K^X|^2}\left(1- 
\frac{d\vartheta_K}{2i\pi|K^X|^{2}}\right)^{-1}.
\end{multline}
From (\ref{eq:3.17})-(\ref{eq:3.19}), we know that there exists
$C>0$ such that 
\begin{align}\label{eq:3.34b}
\left|K^X(y,Z) \right|^2\geq C|Z|^2,
\end{align}
and for $Y_1\in T_yW^{g,K}$,
\begin{align}\label{eq:3.34c}
i_{Y_1^H}\vartheta_{K}=\mathcal{O}(|Z|^3),
\quad i_{Y_1^H}d\vartheta_{K}=\mathcal{O}(|Z|^2).
\end{align}
 From (\ref{eq:3.31})-(\ref{eq:3.34c}) and
the rank $\ell-\ell'$ of $N_{X^{g,K}/X^g}$ is even,
we know that near $W^{g,K}$,
\begin{align}\label{eq:3.35b}
Q_K(y,Z)=\mathcal{O}(|Z|^{1-(\ell-\ell')}).
\end{align}
Thus as a current over $W^g$, $Q_K$ is in fact locally integrable
over $W^g$ and given by (\ref{eq:3.31}). 
For $g=e$, and $B=\mathrm{pt}$,
this is exactly \cite[Proposition 2.2]{Go09}.

Assume now $K^X$ has no zeros, for $t\neq 0$ small enough,
by (\ref{eq:3.07}), (\ref{eq:3.26}), (\ref{eq:3.30}) and (\ref{eq:3.31}), 
we have
\begin{align}\label{eq:3.33}
\wi{\eta}_{g,tK}=\wi{\eta}_{ge^{tK}}
-
\int_{X^g}\frac{\vartheta_{K}}{2i\pi t|K^X|^2}\left(1- 
\frac{d\vartheta_{K}}{2i\pi t|K^X|^{2}}\right)^{-1}  \alpha_{tK}
\in \Omega^{\bullet}(B,\C)/d\Omega^{\bullet}(B,\C).
\end{align}
In particular, for $g=e$ and $B=\mathrm{pt}$, 
(\ref{eq:3.33}) as Taylor expansion at $t=0$ is 
\cite[Theorem 0.5]{Go00}.
\end{rem}

\section{A proof of Theorem \ref{thm:3.03}}\label{s04}

In this section, we state some intermediate results and prove 
Theorem \ref{thm:3.03}. The proofs of the intermediate results 
are delayed to Section 6.

\subsection{Some intermediate results}\label{s0401}

For $t>0$, $v>0$, set
\begin{align}\label{eq:4.01} 
\mC_{v,t}=\mathbb{B}_t+\frac{\sqrt{t}c(K^X)}{4}
\left(\frac{1}{t}-\frac{1}{v}\right)
+dt\wedge \frac{\partial}{\partial t}
+dv\wedge \frac{\partial}{\partial v}.
\end{align}
Then $\mC_{v,t}$ is a superconnection associated
 with the fibration 
$(\R_+^*)^2\times W \rightarrow (\R_+^*)^2\times B$. 
From the argument in the proof of \cite[Theorem 9.17]{BGV}, we have
\begin{align}\label{eq:4.02} 
d^{\R^2\times B}\wi{\tr}[g\exp(-\mC_{v,t}^2
-\mathcal{L}_{K})]=0.
\end{align}

For $\alpha\in \Lambda(T^*(\R^2\times B))$, 
\begin{align}\label{eq:4.03}
\alpha=\alpha_0+dv\wedge \alpha_1+dt\wedge \alpha_2+ dv\wedge 
dt\wedge \alpha_3,\quad \alpha_i\in 
\Lambda (T^*B), i=0,1,2,3, 
\end{align}
as in (\ref{eq:0.06}), set
\begin{align}\label{eq:4.04}
[\alpha]^{dv}:=\alpha_1,\quad  [\alpha]^{dt}:=\alpha_2,\quad 
[\alpha]^{dv\wedge dt}:=\alpha_3.
\end{align}

\begin{defn}\label{defn:4.01}
We define $\beta_{g,K}$ to be the part of $-\psi_{\R^2\times B}
\wi{\tr}[g\exp(-\mC_{v,t}^2-\mathcal{L}_{K})]$
of degree one with respect to the coordinates $(v,t)$.
We denote by 
\begin{align}\label{eq:4.06}
\alpha_{g,K}=-\left\{\psi_{\R^2\times B}
\wi{\tr}[g\exp(-\mC_{v,t}^2-\mathcal{L}_{K})]\right\}^{
dv\wedge dt}.
\end{align}
\end{defn}	

From comparing the coefficient of $dv\wedge dt$ part of 
(\ref{eq:4.02}), we have
\begin{align}\label{eq:4.07}
\left(dv\wedge\frac{\partial}{\partial v}+dt\wedge\frac{
\partial}{\partial t}\right)\beta_{g,K}=-dv\wedge dt\wedge 
d^B\alpha_{g,K}.
\end{align}

Take $a,A$, $0<a\leq 1\leq A<+\infty$. Let $\Gamma=
\Gamma_{a,A}$ be the oriented contour in $\R_{+,v}
\times\R_{+,t}$:
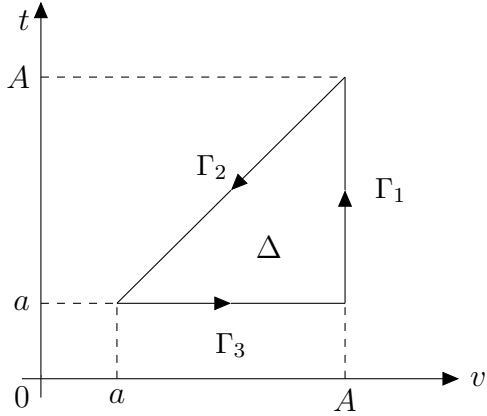
\begin{equation*}
\begin{tikzpicture}
\draw[->][ -triangle 45] (-0.25,0) -- (5.5,0);
\draw[->][ -triangle 45] (0,-0.25) -- (0,5);
\draw[->][ -triangle 45] (1,1) -- (2.5,1);
\draw (2.5,1) -- (4,1);
\draw[->][ -triangle 45] (4,1) -- (4,2.5);
\draw (4,2.5) -- (4,4);
\draw[->][ -triangle 45] (4,4) -- (2.5,2.5);
\draw (2.5,2.5) -- (1,1);
\draw[dashed] (0,1) -- (1,1);
\draw[dashed] (0,4) -- (4,4);
\draw[dashed] (1,0) -- (1,1);
\draw[dashed] (4,0) -- (4,1);
\foreach \x in {0}
\draw (\x cm,1pt) -- (\x cm,1pt) node[anchor=north east] {$\x$};
\draw
(3,1.75)  node {$\Delta$}(3,1.75);
\draw
(0,4.75)  node[anchor=east] {$t$}(0,4.75);
\draw
(5.5,0)  node[anchor=west] {$v$}(5.5,0);
\draw
(0,1)  node[anchor=east] {$a$}(0,1);
\draw
(0,4)  node[anchor=east] {$A$}(0,4);
\draw
(1,0)  node[anchor=north] {$a$}(1,0);
\draw
(4,0)  node[anchor=north] {$A$}(4,0);
\draw
(4.25,2.5)  node[anchor=west] {\small{$\Gamma_1$}}(4.25,2.5);
\draw
(2.5,0.75)  node[anchor=north] {\small{$\Gamma_3$}}(2.5,0.75);
\draw
(2.25,2.5)  node[anchor=south] {\small{$\Gamma_2$}}(2.25,2.5);
\end{tikzpicture}
\end{equation*}

The contour $\Gamma$ is made of three oriented pieces $\Gamma_1, 
\Gamma_2, \Gamma_3$ indicated in the above picture.
For $1\leq k\leq 3$, set $I_k^0=\int_{\Gamma_k}\beta_{g,K}$.
Also $\Gamma$ bounds an oriented triangular domain $\Delta$.

By Stocks' formula and  (\ref{eq:4.07}),
\begin{align}\label{eq:4.08}
\sum_{k=1}^3I_k^0=\int_{\partial \Delta}\beta_{g,K}=
\int_{\Delta}\left(dv\wedge\frac{\partial}{\partial v}
+dt\wedge\frac{\partial}{\partial t}\right)\beta_{g,K}
=-d^B\left(\int_{\Delta} \alpha_{g,K} dv\wedge dt\right).
\end{align}

The proof of the following theorem is left to Section \ref{s0511}.

\begin{thm}\label{thm:4.02} 
For $K\in \mathfrak{z}(g)$, $|K|$ small enough, 
there exist $\delta>0$, $C>0$
such that  for any $t\geq 1$, $v\geq t$, we have
\begin{align}\label{eq:4.09} 
\left|[\beta_{g,K}(v,t)]^{dt}\right|\leq \frac{C}{t^{1+\delta}}.
\end{align}
\end{thm}

For $\alpha\in \Omega^{j}(B,\C)$,
we define 
\begin{align}\label{eq:4.10}
\phi(\alpha):=\{\psi_{\R\times B}(dv\wedge \alpha) \}^{dv}=\left\{
\begin{array}{ll}
\pi^{-\frac{1}{2}}\left(2i\pi\right)^{-\frac{j}{2}}\cdot 
\alpha\hspace{10mm} & \hbox{if $j$ is even;} \\
\left(2i\pi\right)^{-\frac{j+1}{2}}\cdot \alpha
\hspace{10mm} & \hbox{if $j$ is odd.}
\end{array}
\right.
\end{align}
Comparing with (\ref{eq:1.26}), we set
\begin{align}\label{eq:4.10a}
\widetilde{\tr}'=\left\{
\begin{array}{ll}
\tr_s
& \hbox{if $n$ is even;} \\
\tr^{\mathrm{even}} & \hbox{if $n$ is odd.}
\end{array}
\right.
\end{align}
For $0<t\leq v$, set 
\begin{align}\label{eq:4.11} 
\mB_{K,t,v}=\left(\mathbb{B}_t+\frac{\sqrt{t}c(K^X)}{4}\left(
\frac{1}{t}-\frac{1}{v} \right)\right)^2+\mathcal{L}_{K}.
\end{align}

Then by Definition \ref{defn:4.01}, (\ref{eq:4.01}) and (\ref{eq:4.11}), 
we have 
\begin{align}\label{eq:4.05}
\begin{split}
&[\beta_{g,K}(v,t)]^{dt}\\
&\hspace{5mm}=-\left\{\psi_{\R_t\times B}
\wi{\tr}\left[g\exp\left(-\mB_{K,t,v}
-dt\wedge \frac{\partial}{\partial t}
\left(\mathbb{B}_t+\frac{\sqrt{t}
c(K^X)}{4}\left(\frac{1}{t}-\frac{1}{v}\right)\right)\right) 
\right]\right\}^{dt}
\\
&\hspace{5mm}=\phi\wi{\tr}'\left[g\frac{\partial}{\partial t}
\left(\mathbb{B}_t+\frac{\sqrt{t}
	c(K^X)}{4}\left(\frac{1}{t}-\frac{1}{v}\right)\right)
\exp\left(-\mB_{K,t,v}\right) 
\right],
\\
&[\beta_{g,K}(v,t)]^{dv}=-\left\{\psi_{\R_v\times B}
\wi{\tr}\left[g\exp\left(-\mB_{K,t,v}-dv 
\frac{\sqrt{t}c(K^X)}{4v^2}\right) \right]\right\}^{dv}
\\
&\hspace{5mm}=\phi\wi{\tr}'\left[g\frac{\sqrt{t}c(K^X)}{4v^2}
\exp\left(-\mB_{K,t,v}\right) 
\right].
\end{split}
\end{align}
Thus as $\mB_{K,t,t}=\mathbb{B}_t^2+\mL_K$, 
by (\ref{eq:4.05}), on $\Gamma_2$, we have 
\begin{multline}\label{eq:4.13b} 
\beta_{g,K}(v,t)=dt\wedge \phi\wi{\tr}'\left[
g\frac{\partial\mathbb{B}_t}{\partial t}
\exp\left(-\mathbb{B}_t^2-\mL_K\right) 
\right]
\\
=-dt\wedge \left\{\psi_{\R\times B}
\left.\widetilde{\tr}\right.\left[g\exp\left(-\left(
\mathbb{B}_{t}+dt\wedge\frac{\partial}{\partial 
	t}\right)^2-\mL_K\right)\right]\right\}^{dt}.
\end{multline}

In the 
remainder of this section, we use Theorem \ref{thm:4.02} and 
the following estimates 
to  prove Theorem \ref{thm:3.03}. The proofs of these 
estimates are delayed to Section 6.
Recall that $\widetilde{e}_v$ is defined in (\ref{eq:3.07}).

\begin{thm}\label{thm:4.03} 
For $K_{0}\in \mathfrak{z}(g)$, there exists $\beta>0$
	such that for $K=zK_{0}$, $-\beta<z<\beta$, $K\neq 0$,
	
a) when $t\rightarrow 0$,
\begin{align}\label{eq:4.12} 
\phi\wi{\tr}'\left[g\frac{\sqrt{t}c(K^X)}{4v}\exp\left(-
\mB_{K,t,v}\right) \right]\rightarrow -\wi{e}_v;
\end{align}
	
b) there exist $C>0$, $\delta\in (0,1]$, such that for 
$t\in (0,1]$, $v\in [t,1]$, 
\begin{align}\label{eq:4.13} 
\left|\phi\wi{\tr}'\left[g\frac{\sqrt{t}c(K^X)}{4v}\exp\left(
-\mB_{K,t,v}\right) \right]+\wi{e}_v\right|\leq C\left(\frac{t
}{v} \right)^{\delta};
\end{align}
	
c) there exists $C>0$ such that for $t\in (0,1]$, $v\geq 1$, 
\begin{align}\label{eq:4.14} 
\left|\wi{\tr}'\left[g\frac{\sqrt{t}c(K^X)}{4v}\exp\left(
-\mB_{K,t,v}\right) \right]\right|\leq \frac{C}{v};
\end{align}
	
d) for $v\geq 1$, 
\begin{align}\label{eq:4.15} 
\lim_{t\rightarrow 0}\wi{\tr}'\left[g\frac{c(K^X)}{4\sqrt{t}v}
\exp\left(-\mB_{K,t,tv}\right) \right]=0.
\end{align}
\end{thm}	

\subsection{A proof of Theorem \ref{thm:3.03}}\label{s0402}

We now finish the proof of Theorem \ref{thm:3.03} 
by using Theorems \ref{thm:4.02} and \ref{thm:4.03}. 
By (\ref{eq:4.08}), we know that $I_1^0+I_2^0+I_3^0$
is an exact form on $B$. We take the limits $A\rightarrow+\infty$ and 
then $a\rightarrow 0$ in the indicated order. We claim that the limit 
of the part $I_j^0(A,a)$ as $A\rightarrow +\infty$ exists, 
denoted by $I_j^1(a)$, and the limit of $I_j^1(a)$ as 
$a\rightarrow 0$ exists, denoted by $I_j^2$ for $j=1,2,3$. 

i) 
By (\ref{eq:4.11}) and (\ref{eq:4.05}), $[\beta_{g,K}
(v,t)]^{dt}$ is uniformly bounded for $v\geq 1$, $t\in I$,
a compact interval $I\subset (0,+\infty)$, and
\begin{align}\label{eq:4.17b}
\lim_{v\rightarrow +\infty}[\beta_{g,K}(v,t)]^{dt}
=[\beta_{g,K}(+\infty,t)]^{dt}.
\end{align}
From Theorem \ref{thm:4.02}, (\ref{eq:2.23}), (\ref{eq:4.05}) 
and the dominated convergence theorem, we see that
\begin{align}\label{eq:4.17} 
I_1^1(a)=\lim_{A\rightarrow +\infty}\int_a^A [\beta_{g,K}
(A,t)]^{dt}dt=-\int_{a}^{+\infty}\left\{\psi_{\R\times B}
\wi{\tr}\Big[g\exp\left(-\mB_{K,t}^2-\mL_{K}
\right) \Big]\right\}^{dt}dt.
\end{align}
Thus by Theorem \ref{thm:2.02} and Definition \ref{defn:2.03}, 
we have
\begin{align}\label{eq:4.18} 
I_1^2=-\int_{0}^{+\infty}\left\{\psi_{\R\times B}
\wi{\tr}\Big[g\exp\left(-\mB_{K,t}^2-\mL_{K}
\right) \Big]\right\}^{dt}dt=\wi{\eta}_{g,K}.
\end{align}

ii) From Definition \ref{defn:1.03}, (\ref{eq:2.02b}) 
and (\ref{eq:4.13b}), we have
\begin{align}\label{eq:4.19}
I_2^2=\int_0^{+\infty}\left\{\psi_{\R\times B}
\left.\widetilde{\tr}\right.\left[g\exp\left(-\left(
\mathbb{B}_{t}+dt\wedge\frac{\partial}{\partial 
t}\right)^2-\mL_K\right)\right]\right\}^{dt}dt=-\wi{
\eta}_{ge^K}.
\end{align}

iii) For the term $I_3^{0}(A,a)$, set
\begin{align}\label{eq:4.20} 
\begin{split}
&J_1=-\int_a^1\wi{e}_v\frac{dv}{v},
\\
&J_2=\int_1^{+\infty}\phi\wi{\tr}'\left[g\frac{\sqrt{a}c(K^X)}{
	4v}\exp\left(-\mB_{K,a,v}\right) \right]\frac{dv}{v},
\\
&J_3=\int_1^{1/a}\left(\phi\wi{\tr}'\left[g\frac{c(K^X)}{
4\sqrt{a}v}\exp\left(-\mB_{K,a,av}\right) \right]+\wi{e}_{av} 
\right)\frac{dv}{v}.
\end{split}
\end{align}
Clearly, by Theorem \ref{thm:4.03} c) and (\ref{eq:4.05}), we have
\begin{align}\label{eq:4.21} 
I_3^1(a)=J_1+J_2+J_3.
\end{align}
By (\ref{eq:4.12}), (\ref{eq:4.14}) and (\ref{eq:4.20}), from the 
dominated convergence theorem, we find that as 
$a\rightarrow 0$, 
\begin{align}\label{eq:4.22} 
J_2\rightarrow J_2^1=-\int_{1}^{+\infty}\wi{e}_v\frac{dv}{v}.
\end{align}
By (\ref{eq:4.13}), there exist $C>0$, $\delta\in (0,1]$ 
such that for $a\in (0,1]$, $1\leq v\leq 1/a$,
\begin{align}\label{eq:4.23} 
\left|\phi\wi{\tr}'\left[g\frac{c(K^X)}{4\sqrt{a}v}\exp\left(
-\mB_{K,a,av}\right) \right]+\wi{e}_{av}\right|\leq \frac{C
}{v^{\delta}}.
\end{align}
Using Lemma \ref{lem:3.01}, (\ref{eq:4.15}), 
(\ref{eq:4.20}), (\ref{eq:4.23}), and the dominated convergence 
theorem, as $a\rightarrow 0$, 
\begin{align}\label{eq:4.24}
J_3\rightarrow J_3^1=0.
\end{align}
By (\ref{eq:3.26}), (\ref{eq:4.20})-(\ref{eq:4.22}) and (\ref{eq:4.24}),
we have
\begin{align}\label{eq:4.25} 
I_3^2=-\int_0^{+\infty}\wi{e}_v\frac{dv}{v}=-\mM_{g,K}.
\end{align}

By \cite[\S22, Theorem 17]{dR73}, $d\Omega^{\bullet}(B,\C)$
is closed under the uniformly convergence.
Thus, by (\ref{eq:4.08}),
\begin{align}\label{eq:4.26}
\sum_{j=1}^{3}I_j^2\equiv0\ \mathrm{mod}\ d\Omega^{\bullet}(B,\C).
\end{align}

By (\ref{eq:4.18}), (\ref{eq:4.19}), (\ref{eq:4.25}) 
and (\ref{eq:4.26}),
the proof of Theorem \ref{thm:3.03} is completed.

\section{Construction of the equivariant infinitesimal 
$\eta$-forms}\label{s05}

In this section, we prove Theorems \ref{thm:2.02} and 
\ref{thm:4.02} following the lines of \cite[\S 7]{BG00}
and give a heat kernel proof of the family 
Kirillov formula, Theorem \ref{thm:2.01}. For the convenience 
to compare the arguments in this section with those in 
\cite{BG00}, especially how the extra terms for the family 
version appear, the structure of this section is formulated 
almost the same as in \cite[\S 7]{BG00}.

This section is organized as follows. In Section \ref{s0501},
we prove Theorem \ref{thm:2.02} a). In Sections 
\ref{s0502}-\ref{s0510},
we give proofs of Theorems \ref{thm:2.01} and \ref{thm:2.02} b).
In Section \ref{s0511}, we prove Theorem \ref{thm:4.02}.

\subsection{The behaviour of the trace as 
$t\rightarrow +\infty$}\label{s0501}

Set
\begin{align}\label{eq:5.01}
C_{K,t}=\mathbb{B}_{t}+\frac{c(K^X)}{4\sqrt{t}}+
t\cdot dt\wedge \frac{\partial}{\partial t}.
\end{align}
For $z\in \C$, we denote by
\begin{align}\label{eq:5.02} 
\mA_{zK,t}:=C_{zK,t}^2+z\mL_{K}.
\end{align}
Then Theorem \ref{thm:2.02} a) is implied by the following 
estimate. 

\begin{thm}\label{thm:5.01} 
For $\beta>0$ fixed, there exist 
$C>0$, $\delta>0$ such that if  $K\in \mathfrak{g}$,
$z\in \C$, $|zK|\leq \beta$, $t\geq 1$,
\begin{align}\label{eq:5.03} 
\left|\left\{\wi{\tr}[g\exp(-\mA_{zK,t})]\right\}^{dt}\right|\leq 
\frac{C}{t^{\delta}}.
\end{align}
\end{thm}
\begin{proof}
This subsection is devoted to the proof of Theorem 
\ref{thm:5.01}. 
\end{proof}

In this subsection, we fix $\beta>0$. The constants in this 
subsection may depend on $\beta$.

For $b\in B$, recall that $\mathbb{E}_{b}$ is the vector 
space of the smooth sections of $\mE$ on $X_b$.
For $\mu\in \R$, let $\mathbb{E}_{b}^{\mu}$ be the Sobolev 
spaces of the order $\mu$ of sections of $\mE$
on $X_b$. We equip $\mathbb{E}_{b}^{0}$ by the Hermitian 
product $\la\ ,\ \ra_0$ in (\ref{eq:1.19}). 
Let $\|\cdot\|_0$ be the corresponding norm of 
$\mathbb{E}_{b}^{0}$. For $\mu\in \Z$, let $\|\cdot\|_{\mu}$ 
be the Sobolev norm of $\mathbb{E}_{b}^{\mu}$ induced 
by $\nabla^{TX}$ and $\nabla^{\mE}$.
	
Recall that we assume that the kernels $\Ker (D)$ form a vector bundle 
over $B$. We denote by $P$ the orthogonal projection from 
$\mathbb{E}_{}^0$ to $\Ker (D)$ and let $P^{\bot}=1-P$.
	
Recall that $P^{TX}:TW=T^HW\oplus TX\rightarrow TX$
is the projection defined by (\ref{eq:1.04}). 
For $s,s'\in \mathbb{E}$, $t\geq 1$, we set
\begin{align}\label{eq:5.04}
\begin{split}
|s|_{t,0}^2:&=\|s\|_0^2\,,             \\
|s|_{t,1}^2:&=\|Ps\|_{0}^2+t\|P^\bot s\|_0^2
+t\|\nabla^{\mE}_{P^{TX}\cdot}P^\bot s\|_0^2\,.
\end{split}
\end{align}
Set
\begin{align}\label{eq:5.06}
|s|_{t,-1}=\sup_{0\neq s'\in \mathbb{E}^1}\frac{|\langle 
s,s'\rangle_{0}|}{|s'|_{t,1}}.
\end{align}
Then (\ref{eq:5.04}) and (\ref{eq:5.06}) define Sobolev 
norms on $\mathbb{E}^1$ and $\mathbb{E}^{-1}$. Since 
$\nabla^{\mE}_{P^{TX}\cdot}P$ is an operator along the fiber $X$ 
with smooth kernel, 
we know that $|\cdot|_{t,1}$ (resp. $|\cdot|_{t,-1}$) is 
equivalent to $\|\cdot\|_1$ (resp. $\|\cdot\|_{-1}$) 
on $\mathbb{E}^{1}$ (resp. $\mathbb{E}^{-1}$).
	
Let $\mA_{zK,t}^{(0)}$ be the piece of $\mA_{zK,t}$ which 
has degree 0 in $\Lambda(T^*(\R\times B))$.	
	
\begin{lemma}\label{lem:5.02}
There exist $c_1, c_2, c_3, c_4>0$, such that for
any $t\geq 1$, $K\in\mathfrak{g}$, $z\in \C$, $|zK|\leq 
\beta$, $s,s'\in\mathbb{E}$,
\begin{align}\label{eq:5.07}
\begin{split}
\Re\left\langle \mA_{zK,t}^{(0)}s,s\right\ra_{0}&\geq 
c_1|s|_{t,1}^2-c_2|s|_{t,0}^2,
\\
\left|\Im\left\langle \mA_{zK,t}^{(0)}s,s\right\ra_{0}
\right|&\leq c_3|s|_{t,1}|s|_{t,0},
\\
\left|\left\langle \mA_{zK,t}^{(0)}s,s'\right\ra_{0}
\right|&\leq c_4|s|_{t,1}|s'|_{t,1}.
\end{split}
\end{align}
\end{lemma}
\begin{proof}
From (\ref{eq:1.22}), (\ref{eq:5.01}) and (\ref{eq:5.02}), we have
\begin{align}\label{eq:5.08}
\mA_{zK,t}^{(0)}=tD^{2}+\frac{z}{4}\left[D,c(K^X)\right]
-z^2\frac{|K^X|^2}{16t}+z\mL_K.
\end{align}
So we have
\begin{align}\label{eq:5.09}
\begin{split}
&\Re \la \mA_{zK,t}^{(0)}s,s\ra_0=\left\la\left(tD^{2}
+\Im(z)i\left(\frac{1}{4}\left[D,c(K^X)\right]+\mL_K\right)
-\Re(z^2)
\frac{|K^X|^2}{16t}\right)s,s\right\ra_0,
\\
&\Im \la \mA_{zK,t}^{(0)}s,s\ra_0=\left\la\left(
-\Re(z)i\left(\frac{1}{4}
\left[D,c(K^X)\right]+\mL_K\right)-\Im(z^2)
\frac{|K^X|^2}{16t}\right)s,s\right\ra_0.
\end{split}
\end{align}
From (\ref{eq:5.04}), there exist $c_1', c_2', c_3', 
c_4'>0$ such that for any $t\geq 1$, $|zK|\leq \beta$, 
$\epsilon>0$, 
\begin{align}\label{eq:5.10}
\begin{split}
\left\la\left(tD^{2}
-\Re(z^2)
\frac{|K^X|^2}{16t}\right)s,s\right\ra_{0}&\geq 
c_1'|s|_{t,1}^2-c_2'|s|_{t,0}^2,
\\
\left|\left\la \frac{\Im(z)}{4}\left[D,c(K^X)\right]s,
s\right\ra_{0}\right|&\leq c_3'|s|_{t,1}|s|_{t,0}\leq c_3'
\epsilon|s|_{t,1}^2+\frac{c_3'}{4\epsilon}|s|_{t,0}^2,
\\
\left|\left\la |z|\mL_Ks,s\right\ra_{0}
\right|&\leq c_4'|s|_{t,1}|s|_{t,0}\leq c_4'
\epsilon|s|_{t,1}^2+\frac{c_4'}{4\epsilon}|s|_{t,0}^2.
\end{split}
\end{align}
By taking $\epsilon=\min\{c_1'/(4c_3'), c_1'/(4c_4')\}$, 
from (\ref{eq:5.09}), we 
get the first estimate of (\ref{eq:5.07}).

The other estimates in (\ref{eq:5.07}) follow directly 
from (\ref{eq:5.04}) and (\ref{eq:5.09}).
		
The proof of Lemma \ref{lem:5.02} is completed.
\end{proof}

By using Lemma \ref{lem:5.02} and exactly the same 
argument in \cite[Theorem 11.27]{BL91}, we get
\begin{lemma}\label{lem:5.03}
There exist $c, C>0$,  such that if $t\geq 1$, 
$K\in\mathfrak{g}$, $z\in \C$, $|zK|\leq \beta$,
\begin{align}\label{eq:5.11}
\lambda\in U_c:=\left\{\lambda\in\C:\Re(\lambda)\leq\frac{
\Im(\lambda)^2}{4c^2}-c^2\right\},
\end{align}
the resolvent $(\lambda-\mA_{zK,t}^{(0)})^{-1}$ exists, 
and moreover for any $s\in \mathbb{E}$,
\begin{align}\label{eq:5.12}	
\begin{split}
&|(\lambda-\mA_{zK,t}^{(0)})^{-1}s|_{t,0}\leq C|s|_{t,0}, 
\\
&|(\lambda-\mA_{zK,t}^{(0)})^{-1}s|_{t,1}\leq 
C(1+|\lambda|)^2|s|_{t,-1}.
\end{split}	
\end{align}
\end{lemma}

The following lemma is the analogue of \cite[Theorem 9.15]{Bi97}.
\begin{lemma}\label{lem:5.04}
There exist $C>0, k\in \N$, such that for $t\geq 1$, 
$K\in\mathfrak{g}$, $z\in \C$, $|zK|\leq \beta$, $\lambda\in 
U_{c}$, with $c$ in Lemma \ref{lem:5.03},
the resolvent $(\lambda-\mA_{zK,t})^{-1}$ exists, extends to a
continuous linear operator from $\Lambda(T^*(\R\times B))
\otimes\mathbb{E}^{-1}$ into 
$\Lambda(T^*(\R\times B))\otimes\mathbb{E}^1$, and 
moreover for $s\in \mathbb{E}$,  
\begin{align}\label{eq:5.16}
|(\lambda-\mA_{zK,t})^{-1}s|_{t,1}\leq C(1+|\lambda|)^k|s|_{t,-1}.
\end{align}
\end{lemma}
\begin{proof}
From (\ref{eq:1.01}), (\ref{eq:1.22}), (\ref{eq:5.01}) 
and  (\ref{eq:5.02}),
\begin{multline}\label{eq:5.13}
\mA_{zK,t}-\mA_{zK,t}^{(0)}=\sqrt{t}\left([D, 
\nabla^{\mathbb{E},u}]+\frac{1}{2}dt\wedge D \right)
+\left(\nabla^{\mathbb{E},u}\right)^2 -\frac{1}{4}[D,c(T^H)]
\\
+\frac{1}{8\sqrt{t}}\Big(2[\nabla^{\mathbb{E},u}, z c(K^X)
-c(T^H)]-dt\wedge (z c(K^X)-c(T^H))\Big)
\\
+\frac{1}{16t}\Big(2z \la K^X,T^H\ra+c(T^H)^2\Big).
\end{multline}
By \cite[Theorem 2.5]{Bi86}, $[D, \nabla^{
	\mathbb{E},u}]$ and $\left(\nabla^{\mathbb{E},u}\right)^2$ 
are first order differential operators along the fiber.
From $P[D,\nabla^{\mathbb{E},u}]P=0$, we get 
\begin{align}\label{eq:5.14} 
\Big|\la \sqrt{t}[D, \nabla^{\mathbb{E},u}]s, s'\ra_0\Big|\leq 
C(|s|_{t,0}|s'|_{t,1}+|s|_{t,1}|s'|_{t,0}).
\end{align}  
By (\ref{eq:5.13}) and (\ref{eq:5.14}), 
there exists $C'>0$ such that for any 
$t\geq 1$, we have
\begin{align}\label{eq:5.15}
|(\mA_{zK,t}-\mA_{zK,t}^{(0)})s|_{t,-1}\leq C'|s|_{t,1}.
\end{align}	
	
Take $\lambda\in U_{c}$. Then since $\mA_{zK,t}-\mA_{zK,t}^{(0)}$ 
has positive degree in $\Lambda(T^*(\R\times B))$, we have
\begin{align}\label{eq:5.17}
(\lambda-\mA_{zK,t})^{-1}=\sum_{m=0}^{1+\dim B}(\lambda-
\mA_{zK,t}^{(0)})^{-1}\Big((\mA_{zK,t}-\mA_{zK,t}^{(0)})
(\lambda-\mA_{zK,t}^{(0)})^{-1}\Big)^m.
\end{align}
Therefore, by (\ref{eq:5.12}), (\ref{eq:5.15}) and (\ref{eq:5.17}), 
we obtain (\ref{eq:5.16}).
	
The proof of Lemma \ref{lem:5.04} is completed.
\end{proof}

\begin{prop}\label{prop:5.05}
There exists $C>0$, such that for $t\geq 1$, 
$K\in\mathfrak{g}$, $z\in \C$, $|zK|\leq \beta$, $s\in \mathbb{E}$, 
\begin{align}\label{eq:5.18}
\left\|\Big(\exp(-\mA_{zK,t})-\exp(-\mathbb{B}_{zK,t}^2-z\mL_K
)\Big)s\right\|_0\leq \frac{C}{\sqrt{t}}
\|s\|_0.
\end{align}
\end{prop}
\begin{proof}	
From (\ref{eq:5.04}) and (\ref{eq:5.06}), we know for $s\in \mathbb{E}$,
\begin{align}\label{eq:5.18b}
|P^{\bot}s|_{t,-1}=\sup_{0\neq s'\in \mathbb{E}^1,\, Ps'=0}
\frac{|\la P^{\bot}s,s'\ra_0|}{|s'|_{t,1}}
=\frac{1}{\sqrt{t}}\|P^{\bot}s\|_{-1}
\leq \frac{1}{\sqrt{t}}\|P^{\bot}s\|_0.
\end{align}	
Note that from (\ref{eq:2.19}), (\ref{eq:5.01}) and (\ref{eq:5.02}),
we have
\begin{align}\label{eq:5.19b} 
\mA_{zK,t}=\mathbb{B}_{zK,t}^2+z\mL_K+dt\wedge \left(\frac{1}{2}
\sqrt{t}D-\frac{1}{8\sqrt{t}}\big(zc(K^X)-c(T^H)\big) \right).
\end{align}
Thus $\mathbb{B}_{zK,t}^2+z\mL_K$ has the same spectrum 
as $\mA_{zK,t}$
and by omitting $dt$ part, we know Lemma \ref{lem:5.04} holds for
$\mathbb{B}_{zK,t}^2+z\mL_K$. Thus from (\ref{eq:5.16}) and
(\ref{eq:5.18b}), for $\lambda\in U_c$, we have
\begin{multline}\label{eq:5.20b}
\left\|(\lambda-\mA_{zK,t})^{-1}\sqrt{t}D\Big(\lambda-
(\mathbb{B}_{zK,t}^2+z\mL_K)\Big)^{-1}s\right\|_0
\\
\leq \frac{C}{\sqrt{t}}(1+|\lambda|)^k\left\|
\sqrt{t}DP^{\bot}\Big(\lambda-
(\mathbb{B}_{zK,t}^2+z\mL_K)\Big)^{-1}s\right\|_0
\\
\leq \frac{C}{\sqrt{t}}(1+|\lambda|)^k\left|
\Big(\lambda-
(\mathbb{B}_{zK,t}^2+z\mL_K)\Big)^{-1}s\right|_{t,1}
\\
\leq \frac{C^2}{\sqrt{t}}(1+|\lambda|)^{2k}|s|_{t,-1}\leq 
\frac{C^2}{\sqrt{t}}(1+|\lambda|)^{2k}\|s\|_0.
\end{multline}
Note that
\begin{align}\label{eq:5.21b}
\exp(-\mA_{zK,t})=\frac{1}{2i\pi}
\int_{\partial U_c}e^{-\lambda}(\lambda-\mA_{zK,t})^{-1}d\lambda,
\end{align}
and (\ref{eq:5.21b}) also holds for $\mathbb{B}_{zK,t}^2+z\mL_K$.
From (\ref{eq:5.19b}),
\begin{multline}\label{eq:5.22b}
(\lambda-\mA_{zK,t})^{-1}-(\lambda-
(\mathbb{B}_{zK,t}^2+z\mL_K))^{-1}
\\
=(\lambda-\mA_{zK,t})^{-1}\cdot
\left(dt\wedge \left(\frac{1}{2}
\sqrt{t}D-\frac{1}{8\sqrt{t}}\big(zc(K^X)-c(T^H)\big) \right)
\right)\cdot(\lambda-
(\mathbb{B}_{zK,t}^2+z\mL_K))^{-1}.
\end{multline}
Now from (\ref{eq:5.20b})-(\ref{eq:5.22b}), we get (\ref{eq:5.18}).
The proof of Proposition \ref{prop:5.05} is completed.
\end{proof}

Since $B$ is compact, there exists a family of smooth sections of 
$TX$, $U_1,\cdots,U_m$ such that for any $x\in W$, $U_1(x),\cdots, 
U_m(x)$ spans $T_xX$.	
	
Let $\mD$ be a family of operators on $\mathbb{E}$,
\begin{align}\label{eq:5.22}
\mD=\left\{P^\bot \, \nabla^{\mE}_{U_i}P^\bot\right\}.
\end{align}

From (\ref{eq:5.08}) and (\ref{eq:5.13}), by the same argument 
as the proof of
\cite[Proposition 11.29]{BL91} (see also e.g., \cite[Theorem 9.17]{Bi97},
\cite[Lemma 5.17]{Liu17a}), we get the following lemma. 

\begin{lemma}\label{lem:5.06}
For any $k\in \mathbb{N}$ fixed, there exists $C_k>0$
such that for $t\geq 1$, $K\in\mathfrak{g}$, $z\in \C$, 
$|zK|\leq \beta$, 
$Q_1,\cdots, Q_k\in\mD$ and $s,s'\in \mathbb{E}$, we have
\begin{align}\label{eq:5.23}
|\langle [Q_1, [Q_2,\cdots[Q_k, \mA_{zK,t}],\cdots]]s,s'
\ra_{0}|\leq C_k|s|_{t,1}|s'|_{t,1}.
\end{align}
\end{lemma}

For $k\in \mathbb{N}$, let $\mD^k$ be the family of operators 
$Q$ which can be written in the form
\begin{align}\label{eq:5.24}
Q=Q_1\cdots Q_k,\quad Q_i\in \mD.
\end{align}
If $k\in \mathbb{N}$, we define the Hilbert norm 
$\|\cdot \|_{k}'$ by
\begin{align}\label{eq:5.25}
\|s\|_{k}^{'2}=\sum_{\ell=0}^k\sum_{Q\in \mD^\ell}\|Qs\|^2_0.
\end{align}

Since $P \nabla_{P^{TX}\cdot}^{\mE}$ and 
$\nabla_{P^{TX}\cdot}^{\mE} P$
are operators along the fiber with smooth kernels, 
the Sobolev norm  
$\|\cdot\|_{k}'$ is equivalent to the Sobolev norm 
$\|\cdot\|_k$. Thus, we also denote the Sobolev space 
with respect to $\|\cdot \|_{k}'$ by $\mathbb{E}^k$.

By using Lemma \ref{lem:5.06}, as the proof of 
\cite[Theorem 11.30]{BL91}, we get

\begin{lemma}\label{lem:5.07}
For any $m\in \mathbb{N}$, there exist $p_m\in\mathbb{N}$ and
$C_m>0$ such that for $t\geq 1$,
$\lambda\in U_c$, $s\in \mathbb{E}$,
\begin{align}\label{eq:5.26}
\|(\lambda-\mA_{zK,t})^{-1}s\|_{m+1}'\leq
C_m(1+|\lambda|)^{p_m}\|s\|_{m}'.
\end{align}
\end{lemma}
	
Let $\exp(-\mA_{zK,t})(x,x')$, $\exp(-\mathbb{B}_{zK,t}^2
-z\mL_K)(x,x')$ be the smooth 
kernels of the operators $\exp(-\mA_{zK,t})$, 
$\exp(-\mathbb{B}_{zK,t}^2-z\mL_K)$
	associated with $dv_{X}(x')$.
By using Lemma \ref{lem:5.07}, 
following the same progress as in the proof of 
\cite[Theorem 11.31]{BL91}, we get
\begin{prop}\label{prop:5.08}
For $m\in \N$, there exists $C>0$, such that
for $b\in B$, $x,x'\in X_b$, $t\geq 1$, $K\in\mathfrak{g}$, 
$z\in \C$, $|zK|\leq \beta$,
\begin{align}\label{eq:5.27}
\sup_{|\alpha|,|\alpha'|\leq m}\left|\frac{
	\partial^{|\alpha|+|\alpha'|}}
{\partial x^{\alpha}\partial	{x}^{'\alpha'}}
\exp(-\mA_{zK,t})(x,x')\right|\leq C.
\end{align}
\end{prop}

By omitting $dt$ part, we know Proposition \ref{prop:5.08}
holds for $\exp(-\mathbb{B}_{zK,t}^2-z\mL_K)(x,x')$.		
From Propositions \ref{prop:5.05}, \ref{prop:5.08} and 
(\ref{eq:5.19b}), 
by the arguments in \cite[\S 11 p)]{BL91},
there exist $C>0$, $\delta>0$, 
such that for $t\geq 1$, $K\in \mathfrak{g}$, $|zK|\leq \beta$,
\begin{align}\label{eq:5.28} 
\left|\exp\left(-\mA_{zK,t}
\right) (x,x')-\exp(-\mathbb{B}_{zK,t}^2-z\mL_K)(x,x')\right|
\leq \frac{C}{t^\delta}.
\end{align}

From (\ref{eq:5.19b}),
\begin{align}\label{eq:5.27b} 
dt\wedge \left\{\wi{\tr}\left[g\exp\left(-\mA_{zK,t}\right)
\right]\right\}^{dt}=\wi{\tr}\left[g\Big(\exp\left(-\mA_{zK,t}\right)
-\exp(-\mathbb{B}_{zK,t}^2-z\mL_K)\Big)\right].
\end{align}
From (\ref{eq:5.27}) and (\ref{eq:5.27b}),
we get Theorem \ref{thm:5.01}.

\subsection{A proof of Theorems \ref{thm:2.01} and 
\ref{thm:2.02} b)}\label{s0502}

Section \ref{s0503} is devoted to the proof of 
the following theorem.
\begin{thm}\label{thm:5.09} 
There exist 
$\beta>0$, $C>0$, $0<\delta\leq 1$ such that if
$K\in \mathfrak{z}(g)$, $z\in \C$,
$|zK|\leq \beta$, $0<t\leq 1$,
\begin{align}\label{eq:5.32} 
\left|\psi_{\R\times B}
\wi{\tr}\left[g\exp\left(-\mA_{zK,t}\right)
\right]-\int_{X^g}\widehat{\mathrm{A}}_{g,zK}(TX,\nabla^{TX})\,
\ch_{g,zK}(\mE/\mS,\nabla^{\mE})\right|\leq Ct^{\delta}.
\end{align}
\end{thm}

Since $\int_{X^g}\widehat{\mathrm{A}}_{g,zK}(TX,\nabla^{TX})\,
\ch_{g,zK}(\mE/\mS,\nabla^{\mE})$ does not have the $dt$ term, 
we get Theorem \ref{thm:2.02} b)
from Theorem \ref{thm:5.09}, which we 
reformulate as follows. 

\begin{thm}\label{thm:5.10} 
There exist $\beta>0$, $C>0$, 
$\delta>0$, such that if 
$K\in \mathfrak{z}(g)$, $z\in \C$, $|zK|\leq \beta$, 
$0<t\leq 1$,
\begin{align}\label{eq:5.33} 
\left|\left\{\wi{\tr}\Big[g\exp\big(-\mA_{zK,t}
\big)\Big]\right\}^{dt}\right|
\leq Ct^{\delta}.
\end{align}
\end{thm}

\begin{proof}[Proof of Theorem \ref{thm:2.01}]
If we omit the $dt$ term in (\ref{eq:5.32}) and take $z=1$, it 
follows that
\begin{multline}\label{eq:5.34} 
\left|\psi_{B}\wi{\tr}\left[g\exp\left(-\left(\mathbb{B}_t
+\frac{c(K^X)}{4\sqrt{t}}\right)^2-\mathcal{L}_{K}
\right) \right]\right.
\\
\left.-\int_{X^g}\widehat{\mathrm{A}}_{g,K}(TX,\nabla^{TX})\,
\ch_{g,K}(\mE/\mS,\nabla^{\mE})\right|\leq Ct^{\delta}.
\end{multline}
From (\ref{eq:5.34}), we get (\ref{eq:2.20}) 
and (\ref{eq:2.21}).

From (\ref{eq:2.26b}), (\ref{eq:2.27}), (\ref{eq:2.31b}) 
and (\ref{eq:2.32}),
we get other parts of Theorem \ref{thm:2.01}.

The proof of Theorem \ref{thm:2.01} is completed.
\end{proof}

For simplicity, we will assume in the 
remainder of this section that 
$n=\dim X$ is even. The functional analysis part is exactly 
the same for even and odd dimensional. We only explain in
Remark \ref{rem:5.22} how to use the argument in the proof of
\cite[Theorem 2.10]{BF86II} to compute the local index in odd 
dimensional case.

\subsection{Finite propagation speed and localization}\label{s0503}
	
The proof of the following lemma is the same as Lemma 
\ref{lem:5.02}.
	
\begin{lemma}\label{lem:5.11}
Given $\beta>0$, there exist $C_1, C_2, C_2'(\beta), C_3(\beta),
C_3'(\beta), C_4, C_5(\beta)>0$ such that if $K\in
\mathfrak{g}$, $z\in \C$, $|zK|\leq \beta$,  $s, s'\in \mathbb{E}$, 
$0<t\leq 1$,
\begin{align}\label{eq:5.35}\begin{split}
\Re \la t\mA_{zK,t}^{(0)}s,s\ra_0&\geq C_1t^2\|s\|_1^2
-(C_2t^2+C_2'(\beta))\|s\|_0^2,\\
|\Im\la t\mA_{zK,t}^{(0)}s,s\ra_0|&\leq C_3(\beta)t\|s\|_1
\|s\|_0+C_3'(\beta)\|s\|_0^2,\\
|\la t\mA_{zK,t}^{(0)}s,s'\ra_0|&\leq
C_4(t\|s\|_1+C_5(\beta)\|s\|_0)(t\|s'\|_1+C_5(\beta)\|s'\|_0).
\end{split}
\end{align}
Moreover, as $\beta\rightarrow 0$, $C_2'(\beta)$, $C_3(\beta)$,
$C_3'(\beta)$, $C_5(\beta)\rightarrow 0$.
\end{lemma}

In the sequel, we take $\beta>0$ and always assume that $K\in
\mathfrak{g}$, $|zK|\leq \beta$.
	
For $c>0$, put
\begin{align}\label{eq:5.36}
\begin{split}
&V_c=\left\{\lambda\in\C : \Re (\lambda)\geq \frac{\Im 
(\lambda)^2}{4c^2}-c^2\right\},
\\
&\Gamma_c=\left\{\lambda\in\C : \Re (\lambda)= \frac{\Im 
(\lambda)^2}{4c^2}-c^2\right\}.
\end{split}
\end{align}
Note that $U_c, V_c, \Gamma_c$ are the images of $\{\lambda\in\C: 
|\Im(\lambda)|\geq c \}$, $\{\lambda\in\C: |\Im (\lambda)|\leq c \}$,
$\{\lambda\in\C: |\Im (\lambda)|= c \}$ by the map $\lambda\mapsto 
\lambda^2$.
	
The following lemma is an analogue of 
\cite[Theorem 7.12]{BG00}.
	
\begin{lemma}\label{lem:5.12}
There exists $C>0$ such that given $c\in (0,1]$, for $\beta>0$ and 
$t\in(0,1]$ small enough, if $\lambda\in U_c$, $|zK|\leq \beta$, 
the resolvent
$(\lambda-t\mA_{zK,t}^{(0)})^{-1}$ exists, extends to a continuous 
operator from $\mathbb{E}^{-1}$ into $\mathbb{E}^1$, and moreover, 
for $s\in \mathbb{E}$,
\begin{align}\label{eq:5.37}
\begin{split}
&\|(\lambda-t\mA_{zK,t}^{(0)})^{-1}s\|_0\leq \frac{2}{c^2}\|s\|_0,
\\
&\|(\lambda-t\mA_{zK,t}^{(0)})^{-1}s\|_{1}\leq 
\frac{C}{c^2t^4}(1+|\lambda|)^2\|s\|_{-1}.
\end{split}
\end{align}
\end{lemma}
\begin{proof}
From the same arguments in \cite[(7.47)-(7.49)]{BG00},
by Lemma \ref{lem:5.11}, if $\lambda\in \R$,
$\lambda\leq -(C_2t^2+C_2'(\beta))$, the resolvent
$(\lambda-t\mA_{zK,t}^{(0)})^{-1}$ exists.
		
Now we take $\lambda=a+ib$, $a,b\in\R$. By (\ref{eq:5.35}),
\begin{multline}\label{eq:5.38}
|\la(t\mA_{zK,t}^{(0)}-\lambda)s,s\ra|\geq \sup
\Big\{C_1t^2\|s\|_1^2-(C_2t^2+C_2'(\beta)+a)\|s\|_0^2,
\\
-C_3(\beta)t\|s\|_1\|s\|_0+(|b|-C_3'(\beta))\|s\|_0^2 \Big\}.
\end{multline}
	
Set
\begin{align}\label{eq:5.39}
C(\lambda, t)=\inf_{u\geq 1}\sup\Big\{C_1(tu)^2-(C_2t^2
+C_2'(\beta)+a),-C_3(\beta)tu-C_3'(\beta)+|b| \Big\}.
\end{align}
Since $\|s\|_0\leq \|s\|_1$, using (\ref{eq:5.38}), 
(\ref{eq:5.39}), we get
\begin{align}\label{eq:5.40}
|\la(t\mA_{zK,t}^{(0)}-\lambda)s,s\ra|\geq C(\lambda,t)\|s\|_0^2.
\end{align}
	
Now we fix $c\in (0,1]$. Suppose that $\lambda\in U_c$, i.e.,
\begin{align}\label{eq:5.41}
a\leq \frac{b^2}{4c^2}-c^2.
\end{align}
Assume that $u\in\R$ is such that
\begin{align}\label{eq:5.42}
|b|-C_3(\beta)tu-C_3'(\beta)\leq c^2.
\end{align}
Then by (\ref{eq:5.41}) and (\ref{eq:5.42}),
\begin{multline}\label{eq:5.43}
C_1(tu)^2-(C_2t^2+C_2'(\beta)+a)\geq
C_1(tu)^2-\frac{b^2}{4c^2}+c^2-C_2t^2-C_2'(\beta)
\\
\geq \left(C_1-\frac{C_3(\beta)^2}{4c^2}
\right)(tu)^2-\frac{(c^2+C_3'(\beta))C_3(\beta)}{2c^2}tu+c^2
-\frac{(c^2+C_3'(\beta))^2}{4c^2}-C_2t^2-C_2'(\beta).
\end{multline}
%
The discriminant $\Delta$ of the polynomial in the variable 
$tu$ in the right-hand side of (\ref{eq:5.43}) is given by
\begin{multline}\label{eq:5.44}
\Delta=-3c^2C_1+2C_1(C_3'(\beta)+2C_2t^2+2C_2'(\beta))
+C_3(\beta)^2
\\
+\frac{1}{c^2}(C_1C_3'(\beta)^2-C_2C_3(\beta)^2t^2
-C_2'(\beta)C_3(\beta)^2).
\end{multline}
Clearly, for $\beta$, $t$ small enough,
\begin{align}\label{eq:5.45}
\Delta\leq -2c^2C_1,\quad C_1-\frac{C_3^2(\beta)}{4c^2}>0.
\end{align}
From (\ref{eq:5.43})-(\ref{eq:5.45}), we get
\begin{align}\label{eq:5.46}
C_1(t^2u)^2-(C_2t^4+C_2'(\beta)+a)\geq 
-\frac{\Delta}{4(C_1-C_3^2(\beta)/(4c^2))}\geq \frac{c^2}{2}.
\end{align}
	
Ultimately, by (\ref{eq:5.39})-(\ref{eq:5.46}), we find that 
for $\beta>0$, $t\in (0,1]$ small enough, if $\lambda\in U_c$,
\begin{align}\label{eq:5.47}
C(\lambda,t)\geq \frac{c^2}{2}.
\end{align}
	
From (\ref{eq:5.38}), (\ref{eq:5.39}) and (\ref{eq:5.47}),
we get the first equation of (\ref{eq:5.37}). Then combining 
with (\ref{eq:5.35}) and the argument in \cite[(7.64)-(7.68)]{BG00},
we get the other part of Lemma \ref{lem:5.12}.
	
The proof of Lemma \ref{lem:5.12} is completed.
\end{proof}
	
As in (\ref{eq:5.15}), from (\ref{eq:5.13}), 
there exists $C>0$, such that for any $0<t\leq 1$, 
$s\in\mathbb{E}^1$,
\begin{align}\label{eq:5.48}
\|(t\mA_{zK,t}-t\mA_{zK,t}^{(0)})s\|_{-1}\leq C\|s\|_{1}.
\end{align}
	
From Lemma \ref{lem:5.12}, following the same process as the 
proof of (\ref{eq:5.16}), we get the following lemma.
\begin{lemma}\label{lem:5.13} 
There exist $k,m\in \N$, $C>0$, such that given $c\in(0,1]$, 
for $\beta>0$ and $t\in(0,1]$ small enough, if 
$\lambda\in U_c$, $|zK|\leq \beta$, the resolvent 
$(\lambda-t\mA_{zK,t})^{-1}$ exists, extends to a continuous 
operator from $\Lambda(T^*(\R\times B))\otimes\mathbb{E}^{-1}$ 
into $\Lambda(T^*(\R\times B))\otimes\mathbb{E}^1$, and 
moreover, for $s\in \mathbb{E}$,
\begin{align}\label{eq:5.49}
\|(\lambda-t\mA_{zK,t})^{-1}s\|_1\leq\frac{C}{c^kt^k}
(1+|\lambda|)^{m}\|s\|_{-1}.
\end{align}
\end{lemma}

\begin{defn}\label{defn:5.14} 
If $H, H'$ are separable Hilbert spaces, if $1\leq p< 
+\infty$, set
\begin{align}\label{eq:5.50} 
\mathscr{L}_p(H, H') =\{A\in \mathscr{L}(H, H'): 
\tr[(A^*A)^{p/2}]<+\infty\}.
\end{align}
If $A\in \mathscr{L}_p(H, H')$, set
\begin{align}\label{eq:5.51} 
\|A\|_{(p)}:=\left(\tr[(A^*A)^{p/2}]\right)^{1/p}.
\end{align}
Then by \cite[Chapter IX Proposition 6]{RS75}, 
$\|\cdot\|_{(p)}$ is a norm on $\mathscr{L}_p(H,H')$. 
Similarly, if $A\in \mathscr{L}(H,H')$, let 
$\|A\|_{_{(\infty)}}$ be the usual operator norm of $A$.
\end{defn}

In the sequel, the norms $\|\cdot\|_{(p)}$, 
$\|\cdot\|_{(\infty)}$ will always be calculated with 
respect to the Sobolev spaces $\mathbb{E}^0$.

From Lemma \ref{lem:5.13}, we get the following lemma 
with the same proof as in \cite[Theorem 7.13]{BG00}.

\begin{lemma}\label{lem:5.15}
Given $q\geq 2\dim X+1$, there exist $C>0$, $k,m\in\N$, 
such that given $c\in(0,1]$, for $\beta>0$ and $t\in (0,1]$ 
small enough, if $\lambda\in U_c$, $|zK|\leq \beta$,
\begin{align}\label{eq:5.52}
\begin{split}
&\|(\lambda-t\mA_{zK,t})^{-1}\|_{(q)}\leq\frac{C}{c^kt^k}
(1+|\lambda|)^m,
\\
&\|(\lambda-t\mA_{zK,t})^{-q}\|_{(1)}\leq\frac{C^q}{
(c^kt^k)^q}(1+|\lambda|)^{mq}.
\end{split}
\end{align}
\end{lemma}

Let $a_X$ be the inf of the injectivity radius of the fibers 
$X$. Let $\alpha\in (0,a_X/8]$. The precise value of $\alpha$ 
will be fixed later. The constants $C>0$, $C'>0...$ may depend 
on the choice of $\alpha$.

Let $f:\R\rightarrow [0,1]$ be a smooth even function such that
\begin{align}\label{eq:5.53}
f(u)=\left\{
\begin{aligned}
1\quad &\hbox{for $|u|\leq \alpha/2$;} \\
0\quad &\hbox{for $|u|\geq \alpha$.}
\end{aligned}
\right.
\end{align}
Set
\begin{align}\label{eq:5.54}
g(u)=1-f(u).
\end{align}	
	
For $t>0$, $a\in\C$, put
\begin{align}\label{eq:5.55}
\left\{
\begin{aligned}
&F_t(a)=\int_{-\infty}^{+\infty}\exp(\sqrt{2} iua)\exp
\left(-\frac{u^2}{2}\right)f(\sqrt{t}u)\frac{du}{\sqrt{2\pi}},
\\
&G_t(a)=\int_{-\infty}^{+\infty}\exp(\sqrt{2} iua)\exp
\left(-\frac{u^2}{2}\right)g(\sqrt{t}u)\frac{du}{\sqrt{2\pi}}.
\end{aligned}
\right.
\end{align}
Then $F_t(a)$, $G_t(a)$ are even holomorphic functions of $a$ 
such that
\begin{align}\label{eq:5.56}
\exp(-a^2)=F_t(a)+G_t(a).
\end{align}
Moreover when restricted on $\R$, $F_t$ and $G_t$ both lie 
in the Schwartz space $\mS(\R)$. Put
\begin{align}\label{eq:5.57}
I_t(a)=G_t(a/\sqrt{t}).
\end{align}
Clearly, there exist uniquely defined holomorphic functions 
$\wi{F}_t(a)$, $\wi{G}_t(a)$, $\wi{I}_t(a)$ such that
\begin{align}\label{eq:5.58}
F_t(a)=\wi{F}_t(a^2),\quad G_t(a)=\wi{G}_t(a^2),\quad 
I_t(a)=\wi{I}_t(a^2).
\end{align}
By (\ref{eq:5.56}) and (\ref{eq:5.57}), we have
\begin{align}\label{eq:5.59}
\exp(-a)=\wi{F}_t(a)+\wi{G}_t(a),\quad
\wi{I}_t(a)=\wi{G}_t(a/t).
\end{align}
From (\ref{eq:5.59}),
\begin{align}\label{eq:5.60}
\exp(-\mA_{zK,t})=\wi{F}_t(\mA_{zK,t})
+\wi{I}_t(t\mA_{zK,t}).	
\end{align}
	
From Lemma \ref{lem:5.15}, the proof of the following lemma 
is the same as that of \cite[Theorem 7.15]{BG00}.
\begin{lemma}\label{lem:5.16}
There exist $\beta>0$, $C>0$, $C'>0$ such that if $t\in(0,1]$,
$K\in \mathfrak{g}$, $|zK|\leq \beta$,
\begin{align}\label{eq:5.61}
\|\wi{I}_t(t\mA_{zK,t})\|_{(1)}\leq C\exp(-C'/t).
\end{align}
\end{lemma}
	
By (\ref{eq:5.60}) and (\ref{eq:5.61}), we find that to 
establish (\ref{eq:5.32}), we may as well replace 
$\exp(-\mA_{zK,t})$ by $\wi{F}_t(\mA_{zK,t})$.
	
Let $\wi{F}_t(\mA_{zK,t})(x,x')$, $(x,x'\in 
X_b, b\in B)$ be the smooth kernel associated with the 
operator $\wi{F}_t(\mA_{zK,t})$ with respect to $dv_X(x')$. 
Clearly the kernel of $g\wi{F}_t(\mA_{zK,t})$ is given 
by $g\wi{F}_t(\mA_{zK,t})(g^{-1}x,x')$. Then,
\begin{align}\label{eq:5.62}
\tr_s[g\wi{F}_t(\mA_{zK,t})]=\int_X
\tr_s[g\wi{F}_t(\mA_{zK,t})(g^{-1}x,x)]dv_X(x).
\end{align}
	
For $\var>0$, $x\in X_b$, $b\in B$, let $B^X(x,\var)$ be 
the open ball in $X_b$ with centre $x$ and radius $\var$. 
Using finite propagation speed for solutions of hyperbolic 
equations (cf. \cite[Appendix D.2]{MM07}), we find that 
given $x\in X_b$, $\wi{F}_t(\mA_{zK,t})(x,\cdot)$ vanishes 
on the complement of $B^X(x,\alpha)$ in $X_b$, and depends only on 
the restriction of the operator $\mA_{zK,t}$ to the ball 
$B^X(x,\alpha)$. Therefore, we have shown that the proof 
of (\ref{eq:5.32}) can be made local on $X_b$. Therefore, 
we may and we will assume that $TX_b$ is spin and 
\begin{align}\label{eq:5.63}
\mE=\mS_X\otimes E
\end{align}
over $X_b$, where $\mS_X$ 
is the spinor of $TX$ and $E$ is a complex vector bundle, 
and the metric and the connection on $\mE$ are induced 
from those on $TX$ and $E$.
	
By the above, it follows that $g\wi{F}_t(\mA_{zK,t})(
g^{-1}x,x)$, $x\in X_b$ vanishes if $d^{X_b}(g^{-1}x,x)\geq 
\alpha$. Here $d^{X_b}$ is the distance in ($X_b$, $g^{TX}$).
	
Now we explain our choice of $\alpha$. Recall that $N_{X^g/X}$ 
is identified with the orthogonal bundle to $TX^g$ in 
$TX|_{X^g}$. Given $\var>0$, let $\mU_{\var}$ be the 
$\var$-neighborhood of $X_b^g$ in $N_{X^g/X}$. There exists 
$\var_0\in (0,a_X/32]$ such that if $\var\in (0,16\var_0]$, 
the fiberwise exponential map $(x,Z)\in N_{X^g/X}\rightarrow 
\exp_x^X(Z)$ is a diffeomorphism of $\mU_{\var}$ on the tubular 
neighborhood $\mV_{\var}$ of $X^g$ in $X$. In the sequel, we 
identify $\mU_{\var}$ and $\mV_{\var}$. This identification is 
$g$-equivariant. We will assume that $\alpha\in 
(0,\var_0]$ is small enough so that for any $b\in B$, if $x\in
X_b$,  $d^{X_b}(g^{-1}x,x)\leq \alpha$, then $x\in \mV_{\var_0}$.
	
By (\ref{eq:5.61}), (\ref{eq:5.62}) and
the above considerations, it follows 
that for $\beta>0$ small enough, the problem is localized on 
the $\var_0$-neighborhood $\mV_{\var_0}$ of $X^g$.
	
As in (\ref{eq:3.09}), let $k(x,Z)$ be the smooth function 
on $\mU_{\var_0}$ such that
\begin{align}\label{eq:5.64}
dv_X(x,Z)=k(x,Z)dv_{X^g}(x)dv_{N_{X^g/X}}(Z).
\end{align}
In particular, $k|_{X^g}=1$.
	
For $\omega\in \Lambda(T^*\R)\widehat{\otimes}\Lambda(T^*W^g)$, 
via (\ref{eq:1.04}) 
and (\ref{eq:1.05}), we will write
$$\omega=\sum_{1\leq i_1<\cdots<i_p\leq \ell} 
\omega_{i_1,\cdots, i_p}
\wedge e^{i_1}\wedge\cdots\wedge e^{i_p},\quad 
\text{for}\ \omega_{i_1,\cdots, i_p}\in \Lambda(T^*\R)
\widehat{\otimes}\pi^*\Lambda(T^*B).$$
We denote by
\begin{align}\label{eq:5.65}
\omega^{\max}:=\omega_{1,\cdots,\ell}\in \Lambda(T^*\R)
\widehat{\otimes}\pi^*\Lambda(T^*B).
\end{align}
Note that if the fiber is odd dimensional,
our sign convention here is compatible with that in 
(\ref{eq:0.11}).
	
\begin{thm}\label{thm:5.17}
There exist $\beta>0$, $\delta\in 
(0,1]$ such that if $K\in \mathfrak{z}(g)$,  $z\in \C$, 
$|zK|\leq \beta$, $t\in (0,1]$, 
$x\in X^g$,
\begin{multline}\label{eq:5.66}
\left|t^{\frac{1}{2}\dim N_{X^g/X}}\int_{Z\in N_{X^g/X},|Z|\leq
\var_0/\sqrt{t}}\psi_{\R\times B}\tr_s[g\wi{F}_t(\mA_{zK,t})(
g^{-1}(x,\sqrt{t}Z),(x,\sqrt{t}Z))]\right.
\\
\cdot 
k(x,\sqrt{t}Z)dv_{N_{X^g/X}}(Z)-
\left\{\widehat{\mathrm{A}}_{g,zK}
(TX,\nabla^{TX})\ch_{g,zK}(E, \nabla^E) \right\}^{\max}
\Big|\leq Ct^{\delta}.
\end{multline}
\end{thm}
\begin{proof}
Sections \ref{s0504}-\ref{s0510} are devoted to the proof of this 
theorem.
\end{proof}
	
\begin{proof}[Proof of Theorem \ref{thm:5.09}]
By (\ref{eq:5.62}) and the finite propagation speed argument 
above, we have
\begin{multline}\label{eq:5.67}
\int_X\tr_s[g\wi{F}_t(\mA_{zK,t})(g^{-1}x,x)]dv_X(x)
=\int_{\mU_{\var_0}}\tr_s[g\wi{F}_t(\mA_{zK,t})(g^{-1}x,x)]
dv_X(x)
\\
=\int_{(x,Z)\in\mU_{\var_0/\sqrt{t}}}t^{\frac{1}{2}\dim
N_{X^g/X}}\tr_s[g\wi{F}_t(\mA_{zK,t})
(g^{-1}(x,\sqrt{t}Z),(x,\sqrt{t}Z))]
\\
\times k(x,\sqrt{t}Z)dv_{X^g}(x)dv_{N_{X^g/X}}(Z).
\end{multline}
		
By Lemma \ref{lem:5.16}, Theorem \ref{thm:5.17} and (\ref{eq:5.67}),
there exist $\beta>0$, $\delta\in (0,1]$ such that for 
$K\in \mathfrak{z}(g)$, $|zK|\leq \beta$, $t\in (0,1]$,
\begin{align}\label{eq:5.68}
\left|\psi_{\R\times B}\tr_s\left[g\exp\left(-\mA_{zK,t}
\right) \right]-\int_{X^g}\widehat{\mathrm{A}}_{g,zK}(TX,
\nabla^{TX})\,\ch_{g,zK}(\mE/\mS,\nabla^{\mE})\right|
\leq Ct^{\delta}.
\end{align}
		
So we obtain Theorem \ref{thm:5.09}. 
\end{proof}

\subsection{A Lichnerowicz formula}\label{s0504}

Let $e_1,\cdots,e_n$ be a locally defined smooth orthonormal 
frame of $TX$. Let $(F,\nabla^F)$ be a vector
bundle with connection on $X$. We use the notation
\begin{align}\label{eq:5.69}
(\nabla_{e_i}^F)^2=\sum_{i=1}^n(\nabla_{e_i}^F)^2
-\nabla^F_{\sum_{i=1}^n\nabla^{TX}_{e_i}e_i}.
\end{align}

Let $H$ be the scalar curvature of $X$. The following result 
is a combination of \cite[Theorem 1.6]{Bi85}, 
\cite[Proposition 7.18]{BG00} (for the term involved
$K^X$ and base $B=\mathrm{pt}$), \cite[Theorem 3.5]{Bi86} 
(for Bismut's Lichnerowicz formula for Bismut superconnection)
and \cite[Theorem 2.10]{BGSII} (for the term involved $dt$).

\begin{prop}\label{prop:5.18}
The following identity holds,
\begin{multline}\label{eq:5.70}
\mA_{zK,t}=-t\left(\nabla_{e_i}^{\mE}+\frac{1}{2
\sqrt{t}}\la S(e_i)e_j, f_p^H\ra c(e_j)f^p\wedge\right.
\\
\left.+\frac{1}{4t}\la S(e_i)f_p^H, f_q^H\ra f^p\wedge
f^q\wedge-\frac{z\la K^X,e_i\ra}{4t}-dt\wedge\frac{
c(e_i)}{4\sqrt{t}}\right)^2
\\
+\frac{t}{4}H+\frac{t}{2}R^{\mE/\mS}(e_i,e_j)c(e_i)c(e_j)
+\sqrt{t}R^{\mE/\mS}(e_i,f_p^H)c(e_i)f^p\wedge
\\
+\frac{1}{2}R^{\mE/\mS}(f_p^H,f_q^H)f^p\wedge
f^q\wedge-zm^{\mE/\mS}(K).
\end{multline}
\end{prop}
\begin{proof}
From Bismut's Lichnerowicz formula (cf. \cite[Theorem 3.5]{Bi86}),
\begin{multline}\label{eq:5.70a}
\mathbb{B}_t^2=-t\left(\nabla_{e_i}^{\mE}+\frac{1}{2
	\sqrt{t}}\la S(e_i)e_j, f_p^H\ra c(e_j)f^p\wedge
+\frac{1}{4t}\la S(e_i)f_p^H, f_q^H\ra f^p\wedge
f^q\wedge\right)^2
\\
+\frac{t}{4}H+\frac{t}{2}R^{\mE/\mS}(e_i,e_j)c(e_i)c(e_j)
+\sqrt{t}R^{\mE/\mS}(e_i,f_p^H)c(e_i)f^p\wedge
+\frac{1}{2}R^{\mE/\mS}(f_p^H,f_q^H)f^p\wedge
f^q\wedge.
\end{multline}	
From (\ref{eq:1.17}), (\ref{eq:2.03}) and (\ref{eq:2.06}), 
\begin{align}
\frac{1}{4}[D,c(K^X)]
=\frac{1}{4}c(e_k)c\left(\nabla_{e_k}^{TX}K^X\right)
-\frac{1}{2}\la K^X,e_j\ra\nabla_{e_j}^{\mE}
=m^{\mS}(K)-\frac{1}{2}\nabla_{K^X}^{\mE}. 
\end{align}
Since the $G$-action preserves the splitting (\ref{eq:1.04}), 
$\la[K^X,f_p^H],e_j\ra=0$. Thus from (\ref{eq:1.09}), (\ref{eq:1.17})
 and (\ref{eq:1.21}),
\begin{multline}
[\nabla^{\mathbb{E},u},c(K^X)]
=f^p\wedge c\left(\nabla_{f_p^H}^{TX}K^{X}\right)
=-\la \nabla_{f_p^H}^{TW,L}K^X,e_j\ra c(e_j)f^p\wedge
\\
=\la \nabla_{K^X}^{TW,L}e_j,f_p^H\ra c(e_j)f^p\wedge
=\la S(K^X)e_j,f_p^H\ra c(e_j)f^p\wedge.
\end{multline}
From (\ref{eq:1.10})-(\ref{eq:1.12}), we get
\begin{align}\label{eq:5.72b} 
S(e_j)e_k=S(e_k)e_j,\quad \la S(e_j)f_p^H,f_q^H\ra
=\frac{1}{2}\la T(f_p^H, f_q^H),e_j\ra.
\end{align}
Thus from (\ref{eq:1.21}),
\begin{multline}\label{eq:5.70b}
[c(T^H),c(K^X)]
=\la S(e_j)f_p^H,f_q^H\ra f^p\wedge f^q\wedge [c(e_j),c(K^X)]
\\
=-2\la S(K^X)f_p^H,f_q^H\ra f^p\wedge f^q\wedge.
\end{multline}
Thus from (\ref{eq:5.01}), (\ref{eq:5.02}) and 
(\ref{eq:5.70a})-(\ref{eq:5.70b}), we get (\ref{eq:5.70}) 
without $dt$ term. By comparing directly the coefficient of 
$dt$ on both sides of (\ref{eq:5.70}) as in \cite[Theorem 2.10]{BGSII},
we get (\ref{eq:5.70}).
	
The proof of Proposition \ref{prop:5.18} is completed.
\end{proof}

\subsection{A local coordinate system near $X^g$}\label{s0505}	
	
Take $x\in W^g$. Then the fiberwise exponential map 
$Z\in T_{x}X\rightarrow\exp_{x}^X(Z)\in X$ identifies 
$B^{T_{x}X}(0,16\var_0)$ with $B^{X}(x,16\var_0)$. 
With this identification, there exists a smooth function 
$k_{x}'(Z)$, $Z\in B^{T_{x}X}(0,a_X/2)$ such that
\begin{align}\label{eq:5.72}
dv_X(Z)=k_{x}'(Z)dv_{TX}(Z),\quad \text{with}\ k_{x}'(0)=1.
\end{align}
We may and we will assume that $\var_0$ is small enough 
so that if $Z\in T_{x}X, |Z|\leq 4\var_0$,
\begin{align}\label{eq:5.73}
\frac{1}{2}g_{x}^{TX}\leq g_Z^{TX}\leq \frac{3}{2}g_{x}^{TX}.
\end{align}

Assume from now, $K\in \mathfrak{z}(g)$.	
Recall that $\vartheta_K$ is the one form dual to $K^X$
defined in (\ref{eq:3.01}). 

\begin{defn}\label{defn:5.19}
Let $\,^1\nabla^{\mE,t}$ be the connection on $\Lambda(T^*\R)
\widehat{\otimes} \pi^*\Lambda(T^*B)\widehat{\otimes}\mE$ 
along the fibers,
\begin{multline}\label{eq:5.74}
\,^1\nabla^{\mE,t}_{\cdot}:=\nabla^{\mE}_{\cdot}+\frac{1}{
2\sqrt{t}}\la	S(\cdot)e_j, f_p^H\ra c(e_j)f^p\wedge
+\frac{1}{4t}\la S(\cdot)f_p^H, f_q^H\ra f^p\wedge f^q\wedge
\\
-\frac{z\vartheta_K(\cdot)}{4t}
-dt\wedge\frac{c(\cdot)}{4\sqrt{t}}.
\end{multline}
\end{defn}

In the sequel, we will trivialize $\Lambda(T^*\R)
\widehat{\otimes} \pi^*\Lambda(T^*B)\widehat{\otimes}\mE$ 
by parallel transport along 
$u\in [0,1]\rightarrow uZ$ with respect to the connection
$\,^1\nabla^{\mE,t}$. Observe that the above
connection is $g$-invariant.

From (\ref{eq:1.10}) and (\ref{eq:1.13}), 
we have $S(e_i)e_j=S(e_j)e_i$.  
Let $L$ be a trivial line bundle over $W$. We equip a 
connection on $L$ by 
\begin{align}\label{eq:5.75}
\nabla^L
=d-\frac{z\vartheta_K}{4}.
\end{align}
Thus
\begin{align}\label{eq:5.76}
R^L=(\nabla^L)^2=-\frac{zd\vartheta_K}{4}.
\end{align}
From (\ref{eq:1.29}), (\ref{eq:3.03}), (\ref{eq:3.05}), (\ref{eq:5.72b}),
(\ref{eq:5.75}) and (\ref{eq:5.76}), 
we could calculate that
\begin{multline}\label{eq:5.77}
\left(\,^1\nabla^{\mE,1}\right)^2(e_i,e_j)=\frac{1}{4}\la 
R^{TX}(e_k,e_l)e_i,e_j\ra c(e_k)c(e_l)+\frac{1}{2}\la 
R^{TX}(e_k,f_p^H)e_i,e_j\ra c(e_k)f^p\wedge
\\
+\frac{1}{4}\la R^{TX}(f_p^H,f_q^H)e_i,e_j\ra f^p\wedge 
f^q\wedge+R^{E}(e_i,e_j)-\frac{z}{2}\la m^{TX}(K)e_i,
e_j\ra.
\end{multline}	

Note that when $K=0$, (\ref{eq:5.77}) is 
\cite[Theorem 4.14]{Bi86}, \cite[Theorem 10.11]{BGV}
or \cite[Theorem 11.8]{Bi97}.
Note that from (\ref{eq:5.74}), 
$\left(\,^1\nabla^{\mE,t}\right)^2$ could be obtained 
from $\left(\,^1\nabla^{\mE,1}\right)^2$ by replacing 
$f^p\wedge$, $f^q\wedge$ and $K$ by 
$\frac{f^p}{\sqrt{t}}\wedge$, 
$\frac{f^q}{\sqrt{t}}\wedge$ and $\frac{K}{t}$.

Let $A$, $A'$ be smooth sections of $TX$. By (\ref{eq:5.74}),
\begin{align}\label{eq:5.78a} 
\,^1\nabla_A^{\mE,1}c(A')=c(\nabla_A^{TX}A')+
\la S(A)A',f_p^H\ra f^p\wedge +\frac{1}{2}\la A, A'\ra dt.
\end{align}
Let $c^1(TX)\simeq TX$ be the set of elements of length $1$
in $c(TX)$. It follows from (\ref{eq:5.78a}) that parallel
transport along the fiber $X$ with respect to $\,^1\nabla^{\mE,1}$
maps $c^1(TX)$ into $c^1(TX)\oplus T^*B\oplus T^*\R$, while 
leaving $\Lambda(T^*B)\widehat{\otimes} \Lambda(T^*\R)$ invariant.

\subsection{Replacing $X$ by $T_{x}X$}\label{s0506}
	
Let $\gamma(u)$ be a smooth even function from $\R$ into 
$[0,1]$ such that
\begin{align}\label{eq:5.79}
\gamma(u)=\left\{
\begin{aligned}
1\quad &\hbox{if $|u|\leq 1/2$;} \\
0\quad &\hbox{if $|u|\geq 1$.}
\end{aligned}
\right.
\end{align}
If $Z\in T_{x}X$, put
\begin{align}\label{eq:5.80}
\rho(Z)=\gamma\left(\frac{|Z|}{4\var_0}\right).
\end{align}
Then
\begin{align}\label{eq:5.81}
\rho(Z)=\left\{
\begin{aligned}
1\quad&\hbox{if $|Z|\leq 2\var_0$;} \\
0\quad&\hbox{if $|Z|\geq 4\var_0$.}
\end{aligned}
\right.
\end{align}
	
For $x\in W^g$, let $\mathbf{H}_x$ be the vector space of smooth 
sections of $\Lambda(T^*\R)\widehat{\otimes}
\pi^*(\Lambda (T^*B))\widehat{\otimes}
\mE_x$ over $T_xX$. Let $\Delta^{TX}$ be the (negative) standard
Laplacian on the fiber of $TX$.
	
Let $L_{x,zK}^{1,t}$ be the differential operator acting on
$\mathbf{H}_x$,
\begin{align}\label{eq:5.82}
L_{x,zK}^{1,t}=(1-\rho^2(Z))(-t\Delta^{TX})+\rho^2(Z)\mA_{zK,t}.
\end{align}
	
Let $\wi{F}_t(L_{x,zK}^{1,t})(Z,Z')$ be the smooth kernel of
$\wi{F}_t(L_{x,zK}^{1,t})$ with respect to $dv_{TX}(Z')$. Using 
the finite propagation speed for solutions of hyperbolic 
equations \cite[Appendix D.2]{MM07} and (\ref{eq:5.72}), we find 
that if $Z\in N_{X^g/X,x}$, $|Z|\leq \var_0$, then
\begin{align}\label{eq:5.83}
\wi{F}_t(\mA_{zK,t})(g^{-1}Z,Z)k_x'(Z)
=\wi{F}_t(L_{x,zK}^{1,t})(g^{-1}Z,Z).
\end{align}
Thus in our proof of Theorem \ref{thm:5.17}, we can then replace 
$\mA_{zK,t}$ by $L_{x,zK}^{1,t}$.

\subsection{The Getzler rescaling}\label{s0507}
	
Let $\mathrm{Op}_x$ be the set of scalar differential 
operators on $T_xX$ acting on $\mathbf{H}_x$. Then
by (\ref{eq:5.63}),
\begin{align}\label{eq:5.84}
L_{x,zK}^{1,t}\in (\Lambda(T^*\R)\widehat{\otimes}
\pi^*(\Lambda (T^*B))\widehat{\otimes} c(TX)
\otimes\End( E))_x\otimes \mathrm{Op}_x.
\end{align}
	
For $t>0$, let $H_t:\mathbf{H}_x\rightarrow \mathbf{H}_x$ be 
the linear map
\begin{align}\label{eq:5.85}
H_th(Z)=h(Z/\sqrt{t}).
\end{align}	
Let $L_{x,zK}^{2,t}$ be the differential operator acting on
$\mathbf{H}_x$ defined by
\begin{align}\label{eq:5.86}
L_{x,zK}^{2,t}=H_t^{-1}L_{x,zK}^{1,t}H_t.
\end{align}
By (\ref{eq:5.84}),
\begin{align}\label{eq:5.87}
L_{x,zK}^{2,t}\in (\Lambda(T^*\R)\widehat{\otimes}
\pi^*(\Lambda (T^*B))\widehat{\otimes} c(TX)
\otimes\End( E))_x\otimes \mathrm{Op}_x.
\end{align}

Recall that $\dim X^g=\ell$ and $\dim N_{X^g/X}=n-\ell$.
Let $(e_1,\cdots,e_{\ell})$ be an orthonormal oriented basis 
of $T_xX^g$, let $(e_{\ell+1},\cdots,e_{n})$ be an orthonormal 
oriented basis of $N_{X^g/X}$, so that $(e_1,\cdots,e_n)$ is 
an orthonormal oriented basis of $T_xX$. We denote 
with an superscript the corresponding dual basis. 
	
For $1\leq j\leq \ell$, the operators $e^j\wedge$ and 
$i_{e_j}$ act as odd operators on $\Lambda(T^*X^g)$.
	
\begin{defn}\label{defn:5.20}
For $t>0$, put
\begin{align}\label{eq:5.88}
c_t(e_j)=\frac{1}{\sqrt{t}}e^j\wedge-\sqrt{t}i_{e_j},
\quad 1\leq j\leq \ell.
\end{align}

\end{defn}

Let $L_{x,zK}^{3,t}$ be the differential operator acting on
$\mathbf{H}_x$ obtained from $L_{x,zK}^{2,t}$ by replacing 
$c(e_j)$ by $c_t(e_j)$ for $1\leq j\leq \ell$.

For $A\in (\Lambda(T^*\R)\widehat{\otimes}
\pi^*(\Lambda (T^*B))\widehat{\otimes} c(TX)
\otimes\End( E))_x\otimes \mathrm{Op}_x$, we denote by
$[A]_t^{(3)}$ the differential operator obtained from $A$ 
by using the Getzler rescaling of the Clifford variables
which is given in Definition \ref{defn:5.20}.

Let $\tau e_j(Z)$ be the parallel transport of $e_j$ along the curve
$t\in [0,1]\rightarrow tZ$ with respect to the connection $\nabla^{TX}$.
Let $\mO_1(|Z|^2)$ be any object in $\Lambda(T^*\R)\widehat{\otimes}
\pi^*(\Lambda (T^*B))\widehat{\otimes} c(TX)$ 
which is of length at most 
$1$ and is also $\mO(|Z|^2)$. By (\ref{eq:5.78a}), in the trivialization
associated with $\,^1\nabla^{\mE,t}$,
\begin{align}\label{eq:5.89a} 
c(\tau e_j(Z))=c(e_j)+\frac{1}{\sqrt{t}}\la S(Z)e_j,f_p^H\ra f^p\wedge
+\frac{1}{2\sqrt{t}}\la Z, e_j\ra dt\wedge +\mO_1(t^{-1/2}|Z|^2).
\end{align}
From (\ref{eq:5.89a}), for $1\leq j\leq \ell$,
\begin{align}\label{eq:5.89b} 
\left[\sqrt{t}c(\tau e_j(\sqrt{t}Z)) \right]_t^{(3)}
=e^j\wedge +\mO(\sqrt{t}|Z|);
\end{align}
for $\ell+1\leq j\leq n$,
\begin{align}\label{eq:5.89c} 
\left[c(\tau e_j(\sqrt{t}Z)) \right]_t^{(3)}
=c(e_j)+\la S(Z)e_j,f_p^H\ra f^p\wedge
+\frac{1}{2}\la Z, e_j\ra dt\wedge +\mO(\sqrt{t}|Z|^2).
\end{align}

From \cite[Proposition 1.18]{BGV}, (\ref{eq:5.70}), (\ref{eq:5.77}), 
(\ref{eq:5.82}),  (\ref{eq:5.86})
and (\ref{eq:5.88}), 
we calculate that
\begin{multline}\label{eq:5.89} 
L_{x,zK}^{3,t}=(1-\rho^2(\sqrt{t}Z))(-\Delta^{TX})
+\rho^2(\sqrt{t}Z)\cdot\left\{-g^{ij}(\sqrt{t}Z)
\left(\nabla_{e_i}'\nabla_{e_j}'
-\Gamma_{ij}^k(\sqrt{t}Z)\sqrt{t}\nabla_{e_k}'  \right)\right.
\\
\left.+\frac{t}{4}H_{\sqrt{t}Z}+\frac{t}{2}
R^{E}_{\sqrt{t}Z}(\tau e_i,\tau e_j)\left[
c(\tau e_i(\sqrt{t}Z))c(\tau e_j(\sqrt{t}Z))\right]_t^{(3)}\right.
\\
+\sqrt{t}R^{E}_{\sqrt{t}Z}(\tau e_j,f_p^H)\left[
c(\tau e_j(\sqrt{t}Z))\right]_t^{(3)}
f^p\wedge
\left.+\frac{1}{2}R^{E}_{\sqrt{t}Z}
(f_p^H, f_q^H)f^p\wedge 
f^q\wedge -m_{\sqrt{t}Z}^{E}(zK)\right\},
\end{multline}
where $\left(g^{ij}(Z)\right)$ is the inverse matrix of 
$\left(g_{ij}(Z)=\langle e_i, e_j\rangle_Z   \right)$, 
$\left(\nabla^{TX}_{e_i}e_j\right)_Z=\Gamma_{ij}^k(Z)e_k$ and
\begin{multline}\label{eq:5.90}
\nabla_{e_i}'=\nabla_{\tau e_i(\sqrt{t}Z)}+\frac{t}{8}
\Big(\la 
R^{TX}_x(e_k,e_l)Z,e_i\ra+\mathcal{O}(\sqrt{t}|Z|^2)\Big)
\left[
c(\tau e_i(\sqrt{t}Z))c(\tau e_j(\sqrt{t}Z))\right]_t^{(3)}
\\
+\frac{\sqrt{t}}{4}\Big(\la R^{TX}_x(e_j,f_p^H)Z,
e_i\ra+\mathcal{O}(\sqrt{t}|Z|^2)\Big)\left[
c(\tau e_j(\sqrt{t}Z))\right]_t^{(3)}f^p\wedge 
\\
+\frac{1}{8}\Big(\la R^{TX}_x(f_p^H,f_q^H)Z,e_i\ra
+\mathcal{O}(\sqrt{t}|Z|^2)\Big) f^p\wedge f^q\wedge 
+ \frac{t}{2}\left( R_x^E(Z, e_i)+\mathcal{O}(\sqrt{t}|Z|^2)\right) 
\\
-\frac{1}{4}\la m^{TX}_x(zK)Z,e_i\ra
+\frac{1}{\sqrt{t}}h_i(zK,\sqrt{t}Z).
\end{multline}
Here $\nabla_U$ is the ordinary differentiation operator on $TX$ in the 
direction $U$,  $h_i(zK, Z)$ is a function depending linearly on $zK$ and 
$h_i(zK,Z)=\mathcal{O}(|Z|^2)$ for $|zK|<\beta$.
	
Let $\wi{F}_t(L_{x,zK}^{3,t})(Z,Z')$ be the smooth kernel 
associated with $\wi{F}_t(L_{x,zK}^{3,t})$ with respect to 
$dv_{TX}(Z')$.

From the finite propagation speed argument explained before
(\ref{eq:5.63}), 
we could also assume that $TX^g$ and $N_{X^g/X}$ are spin.
Let $\mS_{X^g}$ and $\mS_N$ be the spinors of $TX^g$ and 
$N_{X^g/X}$ respectively. Then $\mS_X=\mS_{X^g}\widehat{\otimes}
\mS_N$. Recall that $g$ acts on $(\mS_N\otimes E)_x$.

We may write $\wi{F}_t(L_{x,zK}^{3,t})(Z,Z')$ in the form
\begin{multline}\label{eq:5.91}
\wi{F}_t(L_{x,zK}^{3,t})(Z,Z')=\sum \wi{F}^{j_1
	\cdots j_q}_{t,i_1\cdots i_p}(Z,Z')
e^{i_1}\wedge\cdots
\wedge e^{i_p}i_{e_{j_1}}\cdots i_{e_{j_q}},
\\
1\leq i_1<\cdots<i_p\leq \ell,\ 1\leq j_1<
\cdots<j_q\leq \ell,
\end{multline}
and $\wi{F}^{j_1\cdots j_q}_{t,i_1\cdots i_p}(Z,Z')\in 
\Lambda(T^*\R)\widehat{\otimes}\pi^*\Lambda(T^*B)
\otimes \big(c(N_{X^g/X})\otimes \End(E)\big)_x$.
As explained in Section \ref{s0103},
$\ell=\dim X^g$ has the same parity as $n=\dim X$.
As in (\ref{eq:5.65}),
put
\begin{align}\label{eq:5.92}
[\wi{F}_t(L_{x,zK}^{3,t})(Z,Z') ]^{\max}=
\wi{F}_{t,1\cdots \ell} (Z,Z').
\end{align}
In other words, $\wi{F}_{t,1\cdots \ell}(Z,Z')$ is the 
coefficient of $e^1\wedge\cdots\wedge e^{\ell}$ in 
(\ref{eq:5.91}).

	
\begin{prop}\label{prop:5.21}
If $Z\in T_xX$, $|Z|\leq \var_0/\sqrt{t}$,
\begin{multline}\label{eq:5.93}
t^{(n-\ell)/2}\tr_s[g\wi{F}_t(\mA_{zK,t})
(g^{-1}(\sqrt{t}Z), \sqrt{t}Z)]k_x'(\sqrt{t}Z)
\\
=(-i)^{\ell/2}2^{\ell/2}\tr_s^{\mS_N\otimes E}[g
\wi{F}_t(L_{x,zK}^{3,t})(g^{-1}Z,Z) ]^{\max}.
\end{multline}
\end{prop}
\begin{proof}
As $K\in \mathfrak{z}(g)$, $\,^1\nabla^{\mE,t}$ is $g$-equivariant.
Thus the trivialization $\Lambda(T^*\R)\widehat{\otimes}
\pi^*\Lambda(T^*B)\widehat{\otimes}\mE$ is $g$-equivariant
and the action of $g$ on $\big(\Lambda(T^*\R)\widehat{\otimes}
\pi^*\Lambda(T^*B)\widehat{\otimes}\mE \big)_{g^{-1}Z}$
is the action of $g$ on $\big(\Lambda(T^*\R)\widehat{\otimes}
\pi^*\Lambda(T^*B)\widehat{\otimes}\mE \big)_x$,
which is an element in $\big(c(N_{X^g/X})\otimes \End(E)\big)_x$.
Now we get Proposition \ref{prop:5.21} by the same proof of 
\cite[Proposition 7.25]{BG00}.
\end{proof}

\begin{rem}\label{rem:5.22}
As in \cite[(1.6) and (1.7)]{BF86II}, 
if $n=\dim X$ is even,
\begin{align}\label{eq:5.94}
\begin{split}
&\tr_s^{\mS_X}[c(e_{i_1})\cdots c(e_{i_p})]=0, \ \text{for}
\ p<n, 1\leq i_1<\cdots<i_p\leq n,
\\
&\tr_s^{\mS_X}[c(e_1)\cdots c(e_n)]=(-2i)^{n/2};
\end{split}
\end{align}
if $n=\dim X$ is odd,
\begin{align}\label{eq:5.95}
\tr^{\mS_X}[1]=2^{(n-1)/2},\quad \tr^{\mS_X}[c(e_1)\cdots c(e_n)]
=(-i)^{(n+1)/2}2^{(n-1)/2},
\end{align}
and the trace of the other monomials 
is zero.

If $n=\dim X$ is odd, since (\ref{eq:5.95}) holds
and the total degree of $\wi{F}_t(\mA_{zK,t})$ is even, we only take the 
trace for the odd degree Clifford part.
In this case, (\ref{eq:5.66}) is replaced 
by
\begin{multline}\label{eq:5.96}
\left|t^{(n-\ell)/2}\int_{Z\in N,|Z|\leq\frac{
\var_0}{\sqrt{t}}}\psi_{\R\times B}\tr^{\mathrm{odd}}
[g\wi{F}_t(\mA_{zK,t})(g^{-1}(x,\sqrt{t}Z),(x,\sqrt{t}Z))]
\right.
\\
\left.\times k(x,\sqrt{t}Z)dv_{N_{X^g/X}}(Z)-\left\{\widehat{
\mathrm{A}}_{g,zK}(TX,\nabla^{TX})\ch_{g,zK}(E, \nabla^E) 
\right\}^{\max}\right|\leq Ct^{\delta}.
\end{multline}
In particular, since $n-\ell$ is even,
\begin{multline}\label{eq:5.97}
\tr^{\mS_X}[c(e_1)\cdots c(e_n)]
\\
=(-i)^{(\ell+1)/2}2^{(\ell-1)/2}
t^{\ell/2}\left\{\tr_s^{\mS_N}[c_t(e_1)\cdots c_t(e_{\ell})
c(e_{\ell+1})\cdots c(e_n)] \right\}^{\max},
\end{multline}
the analogue of (\ref{eq:5.93}) is
\begin{multline}\label{eq:5.98}
t^{(n-\ell)/2}\tr^{\mathrm{odd}}[g\wi{F}_t(\mA_{zK,t})
(g^{-1}(\sqrt{t}Z), \sqrt{t}Z)]k_x'(\sqrt{t}Z)
\\
=(-i)^{(\ell+1)/2}2^{(\ell-1)/2}
\{\tr_s^{\mS_N\otimes E}[g
\wi{F}_t(L_{x,zK}^{3,t})(g^{-1}Z,Z) ]\}^{\max}.
\end{multline}
\end{rem}

Let $\jmath:W^g\rightarrow W$ be the obvious embedding.
	
\begin{defn}\label{defn:5.23}
Let $L_{x,zK}^{3,0}$ be the operator in
\begin{align}\label{eq:5.99}
(\pi^*(\Lambda (T^*B))\widehat{\otimes}
\Lambda(T^*X^g)\widehat{\otimes}  c(N_{X^g/X})
\otimes\End( E))_x\otimes\mathrm{Op}_x,
\end{align}
under the notation (\ref{eq:5.69}),
given by
\begin{align}\label{eq:5.100}
L_{x,zK}^{3,0}=-\left(\nabla_{e_i}+\frac{1}{4}\left\la
(\jmath^*R^{TX}_x-m^{TX}(zK)_x)Z,e_i\right\ra\right)^2
+\jmath^*R^E_x-m^{E}(zK)_x.
\end{align}
\end{defn}
	
In the sequel, we will write that a sequence of differential 
operators on $T_xX$ converges if its coefficients converge 
together with their derivatives uniformly on the compact 
subsets in $T_xX$.
	
Comparing with \cite[Proposition 7.27]{BG00}, from 
(\ref{eq:5.89b})-(\ref{eq:5.90}), we have
\begin{prop}\label{prop:5.24}
As $t\rightarrow 0$,
\begin{align}\label{eq:5.101}
L_{x,zK}^{3,t}\rightarrow L_{x,zK}^{3,0}.
\end{align}
\end{prop}

\subsection{A family of norms}\label{s0508}

For $x\in W^g$, let $\mathbf{I}_x$ be the vector space of 
smooth sections of $(\Lambda(T^*\R)\widehat{\otimes}
\pi^*\Lambda(T^*B)\widehat{\otimes}
\Lambda(T^*X^g)\widehat{\otimes}
\mS_N\otimes E)_x$ on $T_xX$, let $\mathbf{I}_{(r,q),x}$ be the 
vector space of smooth sections of $$
\Big(\big(T^*\R\widehat{\otimes}\pi^*
\Lambda^{r-1}(T^*B)\oplus 
\pi^*\Lambda^{r}(T^*B)\big)\widehat{\otimes}
\Lambda^q(T^*X^g)\widehat{\otimes}
\mS_N\otimes E\Big)_x$$ on $T_xX$. We denote by $\mathbf{I}^0_x
=\bigoplus_{r,q} \mathbf{I}^0_{(r,q),x}$ the corresponding vector space
of square-integrable sections. Put $k=\dim B$.

\begin{defn}\label{defn:5.25}
If $s\in \mathbf{I}_{(r,q),x}$ has compact support, put
\begin{align}\label{eq:5.102}
|s|_{t,x,0}^2=\int_{T_xX}|s(Z)|^2\left(1+|Z|\,\rho\left(
\frac{\sqrt{t}Z}{2}\right) \right)^{2(k+\ell+1-q-r)}dv_{TX}(Z).
\end{align}
\end{defn}

Recall that by (\ref{eq:5.81}), if $\rho(\sqrt{t}Z)>0$, then 
$|\sqrt{t}Z|\leq 4\var_0$. 
If $\sqrt{t}|Z|\leq 4\var_0$, then $\rho(\sqrt{t}Z/2)=1$.
By the same arguments as in 
\cite[Proposition 11.24]{BL91},
for $t\in (0,1]$, the following family of operators acting on
$(\mathbf{I}_x^0, |\cdot|_{t,x,0})$ are uniformly bounded:
\begin{align}\label{eq:5.103}
\begin{split}
&1_{|\sqrt{t}Z|\leq 4\var_0}\sqrt{t}c_t(e_j),\quad 1_{
|\sqrt{t}Z|\leq 4\var_0}|Z|\sqrt{t}c_t(e_j),\quad \text{for}\  1\leq 
j\leq \ell,
\\
&1_{|\sqrt{t}Z|\leq 4\var_0}|Z|f^p\wedge,\quad 
1_{|\sqrt{t}Z|\leq 4\var_0}|Z| dt\wedge.
\end{split}
\end{align}

\begin{defn}\label{defn:5.26}
If $s\in \mathbf{I}_x$ has compact support, put
\begin{align}\label{eq:5.104}
|s|_{t,x,1}^2=|s|_{t,x,0}^2+\sum_{i=1}^{n}|\nabla_{e_i}
s|_{t,x,0}^2,
\end{align}
and
\begin{align}\label{eq:5.105}
|s|_{t,x,-1}=\sup_{0\neq s'\in \mathbf{I}_x}\frac{|\la s,s' 
\ra_{t,x,0}|}{|s'|_{t,x,1}}.
\end{align}
\end{defn}

Let $(\mathbf{I}_x^1, |\cdot|_{t,x,1})$ be the Hilbert closure 
of the above
vector space with respect to $|\cdot|_{t,x,1}$.
Let $(\mathbf{I}_x^{-1}, |\cdot|_{t,x,-1})$
be the antidual of $(\mathbf{I}_x^1, |\cdot|_{t,x,1})$.
 Then $(\mathbf{I}_x^1,
|\cdot|_{t,x,1})$ and $(\mathbf{I}_x^0, |\cdot|_{t,x,0})$ 
are densely embedded in $(\mathbf{I}_x^0, |\cdot|_{t,x,0})$ 
and $(\mathbf{I}_x^{-1}, |\cdot|_{t,x,-1})$
with norms smaller than 1 respectively. 


Comparing with \cite[Proposition 7.31]{BG00}, by (\ref{eq:5.89}) 
and (\ref{eq:5.103}), we get the following estimates.

\begin{lemma}\label{lem:5.27}
There exist constants $C_i>0$, $i=1,2,3,4$, such that if 
$t\in (0,1]$, $z\in \C$, $|zK|\leq 1$, if $n\in \N$, $x\in X^g$, 
if the support of $s,s'\in \mathbf{I}_x$ is included in 
$\{Z\in T_xX: |Z|\leq n \}$, then
\begin{align}\label{eq:5.106}
\begin{split}
&\Re\la L_{x,zK}^{3,t}s,s\ra_{t,x,0}\geq
C_1|s|_{t,x,1}^2-C_2(1+|nzK|^2)|s|_{t,x,0}^2,
\\
&|\Im\la L_{x,zK}^{3,t}s,s\ra_{t,x,0}|\leq
C_3\Big((1+|nzK|)|s|_{t,x,1}|s|_{t,x,0}+|nzK|^2|s|_{t,x,0}^2\Big),
\\
&|\la L_{x,zK}^{3,t}s,s'\ra_{t,x,0}|\leq
C_4(1+|nzK|^2)|s|_{t,x,1}|s'|_{t,x,1}.
\end{split}
\end{align}
\end{lemma}
\begin{proof}
We only need to observe that the
terms containing $|nzK|^2$ come from terms
\begin{align}\label{eq:5.107}
\left|\left\la\left(\rho(\sqrt{t}Z)\left(-\frac{1}{4}
\la m^{TX}(zK)Z,e_i
\ra+\frac{1}{\sqrt{t}}h_i(zK,\sqrt{t}Z)\right)\right)^2s,s\right\ra_{t,x,0}
\right|,
\end{align}
which can be dominated by $C(1+|nzK|^2)|s|_{t,x,0}^2$.
	
The proof of Lemma \ref{lem:5.27} is completed.
\end{proof}

\subsection{The kernel $\wi{F}_t(L_{x,K}^{3,t})$ as an 
infinite sum}\label{s0509}

Let $h$ be a smooth even function from $\R$ into $[0,1]$ 
such that
\begin{align}\label{eq:5.108}
h(u)=\left\{
\begin{aligned}
1\quad&\hbox{if $|u|\leq \frac{1}{2}$;} \\
0\quad&\hbox{if $|u|\geq 1$.}
\end{aligned}
\right.
\end{align}
For $n\in \N$, put
\begin{align}\label{eq:5.109}
h_n(u)=h\left(u+\frac{n}{2}\right)+h\left(u-\frac{n}{2}\right).
\end{align}
Then $h_n$ is a smooth even function whose support is included 
in $\left[-\frac{n}{2}-1,-\frac{n}{2}+1\right]\cup
\left[\frac{n}{2}-1,\frac{n}{2}+1\right]$.

Set
\begin{align}\label{eq:5.110}
\mH(u)=\sum_{n\in \N}h_n(u).
\end{align}
The above sum is locally finite, and $\mH(u)$ is a bounded 
smooth even function which takes positive values and has a 
positive lower bound on $\R$.

Put
\begin{align}\label{eq:5.111}
k_n(u)=\frac{h_n}{\mH}(u).
\end{align}
Then the $k_n$ are bounded even smooth functions with bounded 
derivatives, and moreover
\begin{align}\label{eq:5.112}
\sum_{n\in \N}k_n=1.
\end{align}

Note that here we use $n$ as an index for the natural numbers, not
 the $\dim X$ in the previous sections.

\begin{defn}\label{defn:5.28}
For $t\in [0,1]$, $n\in \N$, $a\in\C$, put
\begin{align}\label{eq:5.113}
F_{t,n}(a)=\int_{-\infty}^{+\infty}\exp(\sqrt{2}iua)
\exp\left(-\frac{u^2}{2}\right)f(\sqrt{t}u)k_n(u)\frac{du
}{\sqrt{2\pi}}.
\end{align}
\end{defn}

By (\ref{eq:5.55}), (\ref{eq:5.112}) and (\ref{eq:5.113}),
\begin{align}\label{eq:5.114}
F_t(a)=\sum_{n\in \N}F_{t,n}(a).
\end{align}
Also, given $m,m'\in\N$, there exist $C>0$, $C'>0$, $C''>0$ 
such that for any $t\in [0,1]$, $n\in \N$, $c>0$,
\begin{align}\label{eq:5.115}
\sup_{a\in \C, |\Im(a)|\leq c}|a|^m\left|F_{t,n}^{(m')}(a) 
\right|\leq C\exp(-C'n^2+C''c^2).
\end{align}
Let $\wi{F}_{t,n}(a)$ be the unique holomorphic function 
such that
\begin{align}\label{eq:5.116}
F_{t,n}(a)=\wi{F}_{t,n}(a^2).
\end{align}
Recall that $V_c$ was defined in (\ref{eq:5.36}). By 
(\ref{eq:5.115}), given $m,m'\in\N$, there exist $C>0$, 
$C'>0$, $C''>0$ such that for any $t\in [0,1]$, $n\in \N$, $c>0$, 
$\lambda\in V_c$,
\begin{align}\label{eq:5.117}
|\lambda|^m\left|\wi{F}_{t,n}^{(m')}(\lambda) \right|\leq
C\exp(-C'n^2+C''c^2).
\end{align}
By (\ref{eq:5.114}),
\begin{align}\label{eq:5.118}
\wi{F}_t(a)=\sum_{n\in \N}\wi{F}_{t,n}(a).
\end{align}

Using (\ref{eq:5.118}), we get
\begin{align}\label{eq:5.119}
\wi{F}_t(L_{x,zK}^{3,t})=\sum_{n\in \N}\wi{F}_{t,n}
(L_{x,zK}^{3,t}).
\end{align}
More precisely, by (\ref{eq:5.117}) and using standard 
elliptic estimates, given $t\in (0,1]$, we have the 
identity
\begin{align}\label{eq:5.120}
\wi{F}_t(L_{x,zK}^{3,t})(Z,Z')=\sum_{n\in
	\N}\wi{F}_{t,n}(L_{x,zK}^{3,t})(Z,Z')
\end{align}
and the series in the right-hand side of (\ref{eq:5.120}) 
converges uniformly together with its derivatives on the 
compact sets in $T_xX\times T_xX$.

\begin{defn}\label{defn:5.29}
For $\gamma$ in (\ref{eq:5.79}), put
\begin{align}\label{eq:5.121}
L_{x,zK,n}^{3,t}=-\left(1-\gamma^2\left(\frac{|Z|}{2(n+2)}
\right)\right)\Delta^{TX}+\gamma^2\left(\frac{|Z|}{2(n+2)}
\right)L_{x,zK}^{3,t}.
\end{align}
\end{defn}

Observe that if $k_n(u)\neq 0$, then $|u|\leq \frac{n}{2}+1$. 
Using finite propagation speed and (\ref{eq:5.73}), we find 
that if $Z\in T_xX$, the support of 
$\wi{F}_{t,n}(L_{x,zK}^{3,t})(Z,\cdot)$ is included in 
$\{Z'\in T_xX: |Z'-Z|\leq n+2 \}$. Therefore, 
given $p\in\N$, if $Z\in T_xX$, $|Z|\leq p$, the support of
$\wi{F}_{t,n}(L_{x,zK}^{3,t})(Z,\cdot)$ is included in 
$\{Z'\in T_xX: |Z'|\leq n+p+2 \}$.

If $|Z|\leq n+p+2$, then $\gamma(|Z|/2(n+p+2))=1$. Using 
finite propagation speed again, we see that by 
(\ref{eq:5.121}), for $Z\in T_xX$, $|Z|\leq p$,
\begin{align}\label{eq:5.122}
\wi{F}_{t,n}(L_{x,zK}^{3,t})(Z,Z')=\wi{F}_{t,n}(L_{x,zK,
	n+p}^{3,t})(Z,Z').
\end{align}

From Lemma \ref{lem:5.27}, we have
\begin{align}\label{eq:5.123}
\begin{split}
&\Re\la L_{x,zK,n}^{3,t}s,s\ra_{t,x,0}\geq
C_1|s|_{t,x,1}^2-C_2(1+|nzK|^2)|s|_{t,x,0}^2,
\\
&|\Im\la L_{x,zK,n}^{3,t}s,s\ra_{t,x,0}|\leq
C_3\Big((1+|nzK|)|s|_{t,x,1}|s|_{t,x,0}+|nzK|^2
|s|_{t,x,0}^2\Big),
\\
&|\la L_{x,zK,n}^{3,t}s,s'\ra_{t,x,0}|\leq
C_4(1+|nzK|^2)|s|_{t,x,1}|s'|_{t,x,1}.
\end{split}
\end{align}

Put
\begin{align}\label{eq:5.124}
L_{x,zK,n}^{3,0}=-\left(1-\gamma^2\left(\frac{|Z|}{2(n+2)}
\right)\right)\Delta^{TX}+\gamma^2\left(\frac{|Z|}{2(n+2)}
\right)L_{x,zK}^{3,0}.
\end{align}
By Proposition \ref{prop:5.24}, as $t\rightarrow 0$,
\begin{align}\label{eq:5.125} 
L_{x,zK,n}^{3,t}\rightarrow L_{x,zK,n}^{3,0}.
\end{align}

By (\ref{eq:5.123}), the functional analysis arguments in 
\cite[\S 7.10-7.12]{BG00} work perfectly here. We have the 
following uniform estimates, which is formally the same as 
\cite[Theorem 7.38]{BG00}.
In particular, since the estimates in (\ref{eq:5.106}) and 
(\ref{eq:5.123}) are the analogue of 
\cite[(7.131) and (7.148)]{BG00}, the proof of 
the following theorem is exactly the same as that of 
\cite[Theorem 7.38]{BG00}. 

\begin{thm}\label{thm:5.30} 
There exist $C'>0$, $C''>0$, $C'''>0$ such that for $\eta>0$ 
small enough, there is $c_{\eta}\in (0,1]$ such that 
for any $m\in \N$, there are $C>0$, $r\in \N$ 
such that for $t\in (0,1]$, $|zK|\leq c_{\eta}$, $n\in \N$, 
$x\in X^g$, $Z, Z'\in T_xX$,
\begin{multline}\label{eq:5.126} 
\sup_{|\alpha|,|\alpha'|\leq m}\left|\frac{\partial^{|\alpha|
+|\alpha'|}}{\partial Z^{\alpha}\partial 
Z'^{\alpha'}}\widetilde{F}_{t,n}(L_{x,zK}^{3,t})(Z,Z') \right|
\\
\leq C(1+|Z|+|Z'|)^r\exp\Big(-C'n^2/4+2C''\eta^2\sup 
(|Z|^2, |Z'|^2)-C'''|Z-Z'|^2 \Big).
\end{multline}
\end{thm}

\subsection{A proof of Theorem \ref{thm:5.17}}\label{s0510} 

Remark that as explained in the introduction of \cite{BG00}, 
$L_{x,zK}^{3,t}$ does not have a fixed lower bound. So 
it is not possible to define a priori a honest heat kernel for 
$\exp(-L_{x,zK}^{3,t})$. So we cannot prove Theorem 
\ref{thm:5.17} following the arguments in \cite[\S 11]{BL91}. 

Since $L_{x,zK,n+p}^{3,0}$ coincides with $-\Delta^{TX}$ near 
infinity, the operator $\wi{F}_{0,n}(L_{x,zK,n+p}^{3,0})$ is 
well-defined. Also, by proceeding as in (\ref{eq:5.122}), if 
$|Z|, |Z'|\leq p$, using finite propagation speed, we 
find that the kernel $\wi{F}_{0,n}(L_{x,zK,n+p}^{3,0})(Z,Z')$ 
does not depend on $p$. Finally this kernel verifies estimates 
similar to (\ref{eq:5.126}) for $\eta>0$ small enough and 
$|zK|\leq c_{\eta}$. Therefore we 
may define the kernel $\exp(-L_{x,zK}^{3,0})(Z,Z')$ by 
\begin{align}\label{eq:5.127} 
\exp(-L_{x,zK}^{3,0})(Z,Z')=\sum_{n\in \N}\wi{F}_{0,n}(L_{x,zK,
n+p}^{3,0})(Z,Z'),\quad \text{for}\ |Z|, |Z'|\leq p.
\end{align}
Note that the estimate in (\ref{eq:5.126}) also works for $t=0$.
Thus the series in (\ref{eq:5.127}) converges uniformly on 
compact subsets of $T_xX\times T_xX$ together with its 
derivatives.


From (\ref{eq:5.89}), (\ref{eq:5.100}), (\ref{eq:5.121})
and (\ref{eq:5.124}), there exists $C>0$ such that for $t\in (0,1]$,
$z\in \C$, $|zK|\leq 1$, $n\in \N$, $x\in X^g$, if $s\in \mathbf{I}_x$
has compact support, then
\begin{align}\label{eq:5.127a} 
\left|(L_{x,zK,n}^{3,t}-L_{x,zK,n}^{3,0})s \right|_{t,x,-1}
\leq C\sqrt{t}(1+n^4)|s|_{0,x,1}.
\end{align}
From Theorem \ref{thm:5.30}, (\ref{eq:5.127}) and (\ref{eq:5.127a}), 
the proof of the 
following theorem is exactly the same as that of 
\cite[Theorem 7.43]{BG00}.

\begin{thm}\label{thm:5.31} 
There exist $C''>0$, $C'''>0$ such that 
for $\eta>0$ small enough, there exist $c_{\eta}\in (0,1]$,
$r\in \N$,  $C>0$,   such that for $t\in (0,1]$, 
$z\in \C$, $|zK|\leq c_{\eta}$, $x\in X^g$, $Z, Z'\in T_xX$, 
\begin{multline}\label{eq:5.128} 
\left|\left(\widetilde{F}_{t}(L_{x,zK}^{3,t})
-\exp(-L_{x,zK}^{3,0})\right)(Z,Z') \right|
\leq Ct^{\frac{1}{4(\dim X+1)}}(1+|Z|+|Z'|)^r
\\
\cdot\exp\big(2C''\eta^2\sup (|Z|^2, |Z'|^2)
-C'''|Z-Z'|^2/2 \big).
\end{multline}
\end{thm}

Now there is $C>0$ such that if $Z\in N_{X^g/X}$, then 
\begin{align}\label{eq:5.129} 
|g^{-1}Z-Z|\geq C|Z|.
\end{align}
By (\ref{eq:5.128}) and (\ref{eq:5.129}), we find that there 
exists $C''''>0$ such that if $Z\in N_{X^g/X}$,
\begin{multline}\label{eq:5.130} 
\left|(\widetilde{F}_{t}(L_{x,zK}^{3,t})-\exp(-L_{x,zK}^{3,0}))
(g^{-1}Z,Z) \right|
\\
\leq Ct^{\frac{1}{4(\dim X+1)}}(1+|Z|)^r
\cdot\exp\left(2C''\eta^2|Z|^2-C''''|Z|^2 \right).
\end{multline}
For $\eta>0$ small enough,
\begin{align}\label{eq:5.131} 
2C''\eta^2-C''''\leq -C''''/2.
\end{align}
So by (\ref{eq:5.130}), if $Z\in N_{X^g/X}$,
\begin{align}\label{eq:5.132} 
\left|\left(\widetilde{F}_{t}(L_{x,zK}^{3,t})
-\exp(-L_{x,zK}^{3,0})\right)(g^{-1}Z,Z) \right|
\leq Ct^{\frac{1}{4(\dim X+1)}}\exp\left(-C''''|Z|^2/4 \right).
\end{align}

For $K\in \mathfrak{z}(g)$, put
\begin{align}\label{eq:5.133} 
H^{TX}=\jmath^*R^{TX}-m^{TX}(zK).
\end{align}
Clearly $H^{TX}$ splits under $TX=TX^g\oplus N_{X^g/X}$ as 
\begin{align}\label{eq:5.134}
H^{TX}=H^{TX^g}+H^N.
\end{align}
Using the Mehler's formula (cf. e.g., \cite[(1.34)]{LM00}),
by (\ref{eq:5.100}), for $|zK|$ small enough,
\begin{multline}\label{eq:5.135} 
\exp(-L_{x,zK}^{3,0})(g^{-1}Z,Z)=
(4\pi)^{-\dim X/2}\mathrm{det}^{\frac{1}{2}}\left(\frac{
H^{TX}/2}{\sinh(H^{TX}/2)}\right)
\\
\cdot\exp\left(-\frac{1}{2}\left\la \frac{H^{N}/2}{\sinh(H^{N}/2)}
\big(\cosh(H^{N}/2)-\exp(H^{N}/2)g^{-1}\big)Z,Z \right\ra 
\right)
\\
\cdot\exp(-\jmath^*R^{E}+m^{E}(zK)).
\end{multline}

Observe that for $z\in \C$, $|zK|$ small enough, 
the right-hand side of (\ref{eq:5.135}) is well-defined.
Using (\ref{eq:5.135}), comparing with \cite[(1.37)]{LM00}, 
if $|zK|$ is small enough, 
\begin{multline}\label{eq:5.138} 
\int_{N_{X^g/X}}\exp(-L_{x,zK}^{3,0})(g^{-1}Z,Z)dv_N(Z)
=(4\pi)^{-\ell/2}\mathrm{det}^{\frac{1}{2}}
\left(\frac{H^{TX^g}/2}{\sinh(
H^{TX^g}/2)}\right)
\\
\cdot\Big(\mathrm{det}^{1/2}(1-g^{-1}|_N)
\mathrm{det}^{1/2}(1-g\exp (-H^N))\Big)^{-1}
\cdot\exp(-\jmath^*R^{E}+m^{E}(zK)).
\end{multline}
Also compare with \cite[(1.38)]{LM00},
\begin{multline}\label{eq:5.139} 
\tr_s^{\mS_N\otimes E}[g\exp(-\jmath^*R^{E}
+m^{E}(zK))]
\\
=(-i)^{(\dim X-\ell)/2}\mathrm{det}^{1/2}(1-g^{-1}|_N)
\tr^E[g\exp(-\jmath^*R^{E}+m^{E}(zK))].
\end{multline}
Using (\ref{eq:2.14}), (\ref{eq:2.15}),
(\ref{eq:5.138}) and (\ref{eq:5.139}),  we get
\begin{multline}\label{eq:5.140}
\psi_{\R\times B}\int_{N_{X^g/X}}(-i)^{\ell/2}2^{\ell/2}
\left\{\tr_s^{\mS_N\otimes 
E}[g\exp(-L_{x,zK}^{3,0})(g^{-1}Z,Z)]\right\}^{\max}dv_N(Z)
\\
=\left\{\widehat{\mathrm{A}}_{g,zK}(TX,\nabla^{TX})
\ch_{g,zK}(E,\nabla^{E})\right\}^{\max}.
\end{multline}

From (\ref{eq:5.93}), (\ref{eq:5.132}) and 
(\ref{eq:5.140}), we obtain 
Theorem \ref{thm:5.17} for $\dim X$ even.

If $\dim X$ is odd, following the explanation in Remark \ref{rem:5.22},
the proof is the same.

The proof of Theorem \ref{thm:5.17} is completed.

\subsection{A proof of Theorem \ref{thm:4.02}}\label{s0511}

Since $v\geq t>0$, we have 
\begin{align}\label{eq:5.141} 
0\leq t^{-1}-v^{-1}< t^{-1}.
\end{align}
Set
\begin{align}\label{eq:5.142} 
\mA_{K,t,v}'=\left(\mathbb{B}_t+\frac{\sqrt{t}c(K^X)}{4}\left(
\frac{1}{t}-\frac{1}{v} \right)+t\cdot dt\wedge 
\frac{\partial}{\partial t}\right)^2+\mL_K.
\end{align}
Let $\mA_{K,t,v}'^{(0)}$ be the piece of $\mA_{K,t,v}'$ which has 
degree $0$ in $\Lambda(T^*(\R\times B))$. Then from 
(\ref{eq:5.141}), $\mA_{K,t,v}'^{(0)}$ satisfies the same 
estimate in Lemma \ref{lem:5.02} and the estimate 
(\ref{eq:5.15}) of $\mA_{K,t}-\mA_{K,t}^{(0)}$ also
holds for $\mA_{K,t,v}'
-\mA_{K,t,v}'^{(0)}$ uniformly on $v\geq t\geq 1$. 
Since $v\geq t$, as $t\rightarrow +\infty$, 
we have
\begin{align}\label{eq:5.143} 
\left|\frac{\partial}{\partial t}\left(\frac{\sqrt{t}c(K^X)}{4}
\left(\frac{1}{t}-\frac{1}{v} \right)-\frac{c(T^H)}{4\sqrt{t}}
\right)\right|=\mathcal{O}(t^{-3/2}).
\end{align}
Then the analogue of Propositions \ref{prop:5.05} and \ref{prop:5.08}
 holds for $\mA_{K,t,v}'$ uniformly for $v\geq t\geq 1$. 
Thus replacing $\mA_{zK,t}$ by $\mA_{K,t,v}'$ in the proof of 
Theorem \ref{thm:5.01}, we obtain Theorem \ref{thm:4.02}.

\section{A proof of Theorem \ref{thm:4.03}}\label{s06}

In this section, we prove Theorem \ref{thm:4.03}. This section is 
organized as follows. In Section \ref{s0601}, we establish a 
Lichnerowicz formula for $\mB_{K,t,v}$ in (\ref{eq:4.11}).
In Section \ref{s0602}, 
we prove Theorem \ref{thm:4.03} a). In Sections \ref{s0603} -
\ref{s0608}, we prove Theorem \ref{thm:4.03} b). In Section 
\ref{s0609}, we prove Theorem \ref{thm:4.03} c). In Section 
\ref{s0610}, we prove Theorem \ref{thm:4.03} d). In this 
section, we use the assumptions and the 
notations in Section \ref{s04}.

\subsection{A Lichnerowicz formula}\label{s0601}

Let $L$ be a trivial line bundle over $W$. We equip a 
connection on $L$ by 
\begin{align}\label{eq:6.01} 
\nabla_v^L
=d-\frac{\vartheta_K}{4v}.
\end{align}
Thus
\begin{align}\label{eq:6.02} 
R_v^L=(\nabla_v^L)^2=-\frac{d\vartheta_K}{4v}.
\end{align}
Let $\nabla_v^{\mE\otimes L}$ be the connection on $\mE\otimes L$
induced by $\nabla^{\mE}$ and $\nabla_v^L$.
The corresponding Dirac operator is 
\begin{align}\label{eq:6.03} 
D_v=\sum_{i=1}^nc(e_i)\nabla^{\mE\otimes L}_{v,e_i}=D
-\frac{c(K^X)}{4v}.
\end{align}
Since
\begin{align}\label{eq:6.04} 
\nabla^{\mE\otimes L}_{v,f_p^H}=\nabla^{\mE}_{f_p^H},
\end{align}
from (\ref{eq:6.03}),
the new Bismut superconnection associated with 
$\mE\otimes L$ is
\begin{align}\label{eq:6.05} 
\mathbb{B}_t^v=\mathbb{B}_t-\frac{\sqrt{t}c(K^X)}{4v}.
\end{align}

\begin{thm}\label{thm:6.01} 
The following identity holds,
\begin{multline}\label{eq:6.06}
\mB_{K,t,v}=-t\left(\nabla_{e_i}^{\mE}
+\frac{1}{2\sqrt{t}}\la S(e_i)e_j, f_p^H\ra c(e_j)f^p
\wedge\right.
\\
\left.+\frac{1}{4t}\la S(e_i)f_p^H, f_q^H\ra f^p\wedge
f^q\wedge-\frac{\la K^X,e_i\ra}{4}\left(\frac{1}{t}
+\frac{1}{v}\right)\right)^2
\\
+\frac{t}{4}H+\frac{t}{2}\left(R^{\mE/\mS}(e_i,e_j)-\frac{1}{2v}
\la\nabla_{e_i}^{TX}K^X,e_j\ra\right)c(e_i)c(e_j)
\\
+\sqrt{t}\left(R^{\mE/\mS}(e_i,f_p^H)-\frac{1}{2v}\la 
T(e_i,f_p^H),K^X\ra \right)c(e_i)f^p\wedge
\\
+\frac{1}{2}\left(R^{\mE/\mS}(f_p^H,f_q^H)-\frac{1}{8v}\la 
T(f_p^H,f_q^H), K^X\ra \right)f^p\wedge
f^q\wedge-m^{\mE/\mS}(K^X)+\frac{1}{4v}|K^X|^2.
\end{multline}	
\end{thm}
\begin{proof}
From (\ref{eq:4.11}), (\ref{eq:5.01}), (\ref{eq:5.02}), 
(\ref{eq:5.70}) and (\ref{eq:6.05}), we have
\begin{multline}\label{eq:6.07}
\mB_{K,t,v}=\left(\mathbb{B}_t^v+\frac{c(K^X)}{4\sqrt{t}} 
\right)^2+\mL_{K}=-t\left(\nabla_{v,e_i}^{\mE\otimes L}
+\frac{1}{2\sqrt{t}}\la S(e_i)e_j, f_p^H\ra c(e_j)f^p
\wedge\right.
\\
\left.+\frac{1}{4t}\la S(e_i)f_p^H, f_q^H\ra f^p\wedge
f^q\wedge-\frac{\la K^X,e_i\ra}{4t}\right)^2
+\frac{t}{4}H+\frac{t}{2}R^{\mE\otimes L/\mS}(e_i,e_j)
c(e_i)c(e_j)
\\
+\sqrt{t}R^{\mE\otimes L/\mS}(e_i,f_p^H)c(e_i)f^p\wedge
+\frac{1}{2}R^{\mE\otimes L/\mS}(f_p^H,f_q^H)f^p\wedge
f^q\wedge-m^{\mE\otimes L/\mS}(K^X).
\end{multline}	
Since $G$ acts trivially on $L$, the corresponding $m^L(K)$ 
in the sense of (\ref{eq:2.02}) is given by
\begin{align}\label{eq:6.08}  
m^L(K^X)=-K^X+\nabla^L_{v,K^X}=-\frac{|K^X|^2}{4v}.
\end{align}
Then (\ref{eq:6.06}) follows from (\ref{eq:3.03})-(\ref{eq:3.05}), 
 (\ref{eq:6.02}), (\ref{eq:6.07}) 
and (\ref{eq:6.08}). 
	
The proof of Theorem \ref{thm:6.01} is completed.
\end{proof}

\subsection{A proof of Theorem \ref{thm:4.03} a)}\label{s0602}

Comparing with (\ref{eq:5.74}), we set
\begin{align}\label{eq:6.18}
\,^2\nabla^{\mE,t}_{\cdot}:=\nabla^{\mE}_{\cdot}+\frac{1}{
	2\sqrt{t}}\la S(\cdot)e_j, f_p^H\ra c(e_j)f^p\wedge
+\frac{1}{4t}\la S(\cdot)f_p^H, f_q^H\ra f^p\wedge
f^q\wedge
-\frac{\vartheta_K(\cdot)}{4t}\left(1+\frac{t}{v}\right).
\end{align}	
We trivialize $\pi^*\Lambda(T^*B)\widehat{\otimes}\mE$ 
by parallel transport along 
$u\in [0,1]\rightarrow uZ$ with respect to the connection
$\,^2\nabla^{\mE,t}$. Observe that the above
connection is $g$-equivariant as $K\in \mathfrak{z}(g)$.
Let $A$, $A'$ be smooth sections of $TX$. As in (\ref{eq:5.78a}),
from (\ref{eq:6.18}),
\begin{align}\label{eq:6.18a} 
\,^2\nabla_A^{\mE,1}c(A')=c(\nabla_A^{TX}A')+
\la S(A)A',f_p^H\ra f^p\wedge.
\end{align}	
For $x\in W^g$, in this section, we denote by $\mathbf{H}_x$ 
the vector space of smooth sections of $\pi^*(\Lambda(T^*B))
\widehat{\otimes}\mE_x$.
Let $L_{x,K}^{1,(t,v)}$ be the differential operator acting on
$\mathbf{H}_x$,
\begin{align}\label{eq:6.18b}
L_{x,K}^{1,(t,v)}=(1-\rho^2(Z))(-t\Delta^{TX})+\rho^2(Z)\mB_{K,t,v}.
\end{align}	
We define $L_{x,K}^{2,(t,v)}:=H_t^{-1}L_{x,K}^{1,(t,v)}H_t$ 
and $L_{x,K}^{3,(t,v)}:=\left[L_{x,K}^{2,(t,v)} \right]_t^{(3)}$
as in Section \ref{s0507}.
By Proposition \ref{prop:5.24} for $(\mE\otimes L, 
\nabla_v^{\mE\otimes L})$,
we have
\begin{align}\label{eq:6.18c}
L_{x,K}^{3,(0,v)}=-\left(\nabla_{e_i}+\frac{1}{4}\left\la
(\jmath^*R^{TX}_x-m^{TX}(K)_x)Z,e_i\right\ra\right)^2
+\jmath^*R^{E\otimes L}_x-m^{E\otimes L}(K)_x,
\end{align}	
and as $t\rightarrow 0$,
\begin{align}\label{eq:6.18d}
L_{x,K}^{3,(t,v)}\rightarrow L_{x,K}^{3,(0,v)}.
\end{align}
By (\ref{eq:6.18a}), as in (\ref{eq:5.89a}) and (\ref{eq:5.89c}),
\begin{align}\label{eq:6.16b}
\left[\sqrt{t}c(K^X)(\sqrt{t}Z) \right]_t^{(3)}=\jmath^*\vartheta_{K}
+\mO(\sqrt{t}Z+\sqrt{t}).
\end{align}
By (\ref{eq:2.08}), (\ref{eq:2.16}), (\ref{eq:3.02}), 
(\ref{eq:6.02}) and (\ref{eq:6.08}), we get
\begin{align} 
\jmath^*R_{v,K}^L=\jmath^*R_v^L-2i\pi m^L(K)=-\frac{1}{4v}
(d^{W^g}\vartheta_{K}-2i\pi|K^X|^2)
=-\frac{d_K^{W^g}\vartheta_{K}}{4v}.
\end{align}
Then by (\ref{eq:2.17}),
\begin{align} 
\ch_{g,K}(L,\nabla_v^L)=\exp\left(\frac{d_K^{W^g}\vartheta_{K}
}{8\pi iv} \right).
\end{align}
From (\ref{eq:2.14}), (\ref{eq:2.15}),
(\ref{eq:3.01}), (\ref{eq:3.02}) and (\ref{eq:3.07}),
set
\begin{align}\label{eq:6.17b}
\gamma_{K,v}=-\frac{\vartheta_K}{8vi\pi}
\exp\left(\frac{d_K\vartheta_K}{8vi\pi} 
\right)\widehat{\mathrm{A}}_{g,K}(TX,\nabla^{TX})
\ch_{g,K}(\mE/\mS,\nabla^{\mE/\mS})\in \Omega(W^g,
\det(N_{X^g/X})).
\end{align}
By (\ref{eq:4.10}), (\ref{eq:6.18c}) and (\ref{eq:6.17b}),
if $\dim X$ is even, as in (\ref{eq:5.140}), 
we get 
\begin{multline}\label{eq:6.19b}
\phi\int_{N_{X^g/X}}(-i)^{\ell/2}2^{\ell/2}
\left\{\tr_s^{\mS_N\otimes 
	E\otimes L}\left[g\frac{\jmath^*\vartheta_{K}}{4v}
	\exp(-L_{x,K}^{3,(0,v)})
(g^{-1}Z,Z)\right]\right\}^{\max}dv_N(Z)
\\
=-\left\{\gamma_{K,v}\right\}_x^{\max}.
\end{multline}
By (\ref{eq:3.07}), (\ref{eq:6.18d})-(\ref{eq:6.19b}), 
from the same argument of Section \ref{s0507} and (\ref{eq:5.132})
for $(\mE\otimes L, \nabla_v^{\mE\otimes L})$, 
we obtain Theorem \ref{thm:4.03} a) for $\dim X$ even.

If $n$ is odd, following the explanation in Remark \ref{rem:5.22},
the proof is the same.

The proof of Theorem \ref{thm:4.03} a) is completed.

\subsection{Localization of the problem}\label{s0603}
The proof of Theorem \ref{thm:4.03} b) is devoted to Sections 
\ref{s0603}-\ref{s0608}.

Let $\mB^{0}$ be the piece of $\mB_{K,t,v}$ which has 
degree $0$ in $\Lambda(T^*B)$. Then for $t\in(0,1]$, 
$v\in [t,1]$, by (\ref{eq:5.141}), $t\mB^{0}$ satisfies 
the same estimates as Lemma \ref{lem:5.11} uniformly for $v\in 
[t,1]$.
 
Thus following the same arguments in the proof of Lemma 
\ref{lem:5.16}, we have
\begin{thm}\label{thm:6.03}
There exist $\beta>0$, $C>0$, $C'>0$ such that if 
$K\in\mathfrak{g}$, $|K|\leq \beta$, $t\in(0,1]$, $v\in [t,1]$,
\begin{align}\label{eq:6.14}
\|\wi{I}_t(t\mB_{K,t,v})\|_{(1)}\leq C\exp(-C'/t).
\end{align}
\end{thm}

So our proof of inequality (\ref{eq:4.13}) in Theorem \ref{thm:4.03} 
can be localized near $X^g$. As in Section 5.3, we may and we 
will assume that $W=B\times X$, $TX$ is spin and 
$\mE=\mS_X\otimes E$.

\subsection{A rescaling of the normal coordinate to $X^{g,K}$ 
in $X^g$}\label{s0604}

In the sequel, we fix $g\in G$, $0\neq K_0\in \mathfrak{z}(g)$
and
\begin{align}
K=zK_0,\quad z\in \R^*.
\end{align}

Recall that $X^g$ and $X^{g,K}$ are totally geodesic in $X$. 
Given $\var>0$, let $\mU_{\var}''$ be the $\var$-neighbourhood 
of $X^{g,K}$ in $N_{X^{g,K}/X^g}$ (cf. the notation in the 
proof of Lemma \ref{lem:3.01}). 
By zooming out $\var_0\in(0,a_X/32]$ in Section \ref{s0503},
we can assume that
the map $(y_0,Z_0)\in N_{X^{g,K}/X^g}\rightarrow 
\exp_{y_0}^{X^g}(Z_0)\in X^g$ is a diffeomorphism from 
$\mU_{\var}''$ into the tubular neighbourhood $\mV_{\var}''$ 
of $X^{g,K}$ in $X^g$ for any $0<\var\leq 16\var_0$. 

Since $X^g$ is totally geodesic in $X$, the connection 
$\nabla^{TX}$ induces the connection $\nabla^{N_{X^g/X}}$ 
on $N_{X^g/X}$ (cf. (\ref{eq:1.30}) and (\ref{eq:3.18b})).

For $(y_0,Z_0)\in\mU_{\var}''$, we identify $N_{X^g/X,(y_0,Z_0)}$ 
with $N_{X^g/X,y_0}$ by parallel transport along the geodesic 
$u\in [0,1]\rightarrow u Z_0$ with respect to $\nabla^{TX}$. 
If $y_0\in X^{g,K}$, $Z_0\in N_{X^{g,K}/X^g,y_0}$, 
$Z\in N_{X^g/X,y_0}$, $|Z_0|,|Z|\leq 4\var_0$, we identify 
$(y_0,Z_0,Z)$ with $\exp_{\exp_{y_0}^{X^g}(Z_0)}^X(Z)\in X$. 
Therefore, $(y_0, Z_0,Z)$ defines a coordinate system on $X$ 
near $X^{g,K}$.

From (\ref{eq:2.14}), (\ref{eq:2.15}) and (\ref{eq:6.17b}), 
for $|z|$ small enough, 
$\gamma_{K,v}$ is a smooth form 
on $W^g$. Recall that
the function $k$ is defined in (\ref{eq:5.64})
 and $\ell'=\dim X^{g,K}$.

\begin{thm}\label{thm:6.04} 
There exist $\beta\in (0,1]$, $\delta\in (0,1]$ such that for 
$p\in \N$, there is $C>0$ such that if $z\in \R^*$, $|z|\leq 
\beta$, $t\in (0,1]$, $v\in [t,1]$, $y_0\in X^{g,K}$, $Z_0\in 
N_{X^{g,K}/X^g,y_0}$, $|Z_0|\leq \var_0/\sqrt{v}$, then
for $K=zK_0$,
\begin{multline}\label{eq:6.16} 
\left|v^{\frac{1}{2}\dim N_{X^{g,K}/X^g}}\left(\phi\int_{Z\in 
N_{X^g/X,y_0},|Z|\leq \var_0}\wi{\tr}'\left[g\frac{\sqrt{t}
c(K^X)}{4v}\wi{F}_t\left(-\mB_{K,t,v}\right)\right.\right.\right.
\\
\left.\Big.\left.(g^{-1}(y_0,\sqrt{v}Z_0,Z), (y_0,\sqrt{v}Z_0,Z)) 
\right]\cdot k(y_0,\sqrt{v}Z_0,Z)dv_{N_{X^g/X}}(Z)\right.
\\
\left.+\{\gamma_{K,v}\}^{\max}
(y_0,\sqrt{v}Z_0)\Big)\right|
\leq C\frac{(1+|Z_0|)^{\ell'+1}}{(1+|zZ_0|)^p}\left(\frac{t}{v}
\right)^{\delta}.
\end{multline}
\end{thm}
\begin{proof}
Sections 6.5-6.7 will be devoted to the proof of Theorem 
\ref{thm:6.04}.
\end{proof}

\subsection{A new trivialization and Getzler rescaling near 
$X^{g,K}$}\label{s0605}

Since $g$ preserves geodesics and the parallel transport, in 
the coordinate system in above subsection,
\begin{align}\label{eq:6.17}
g(Z_0,Z)=(Z_0,gZ).
\end{align}
 
By an abuse of notation, we will often write $Z_0+Z$ instead of 
$\exp_{\exp_{y_0}^{X^g}(Z_0)}^X(Z)$.

Firstly, we fix $Z_0\in N_{X^{g,K}/X^g,y_0}$, $|Z_0|\leq \var_0$, 
and we take $Z\in T_{y_0}X, |Z|\leq 4\var_0$. The curve 
$u\in [0,1]\rightarrow Z_0+uZ$ lies in $B_{y_0}^X(0,5\var_0)$. 
Moreover we identify $TX_{Z_0+Z}$, 
$\pi^*\Lambda(T^*B)\otimes\mE_{Z_0+Z}$ with 
$TX_{Z_0}$, $\pi^*\Lambda(T^*B)\otimes\mE_{Z_0}$ 
by parallel transport with respect 
to the connections $\nabla^{TX}$, $\,^2\nabla^{\mE,t}$ 
along the curve.

When $Z_0\in N_{X^{g,K}/X^g,y_0}$ is allowed to vary, we 
identify $TX_{Z_0}$, $\pi^*\Lambda(T^*B)\otimes\mE_{Z_0}$ 
with $TX_{y_0}$, 
$\pi^*\Lambda(T^*B)\otimes\mE_{y_0}$ by parallel transport 
with respect to the 
connections $\nabla^{TX}$, $\nabla^{\mE}$ along 
the curve $u\in [0,1]\rightarrow uZ_0$.
Then $\mathbf{H}_{Z_0}$ is identified 
with $\mathbf{H}_{y_0}$ associated with this trivialization.
Furthermore the fiber of $\pi^*\Lambda(T^*B)\otimes\mE$
at $Z_0+Z$ and $y_0$ are identified by parallel transport along the 
broken curve $u\in [0,1]\rightarrow 2uZ_0$,
for $0\leq u\leq \frac{1}{2}$;
$Z_0+(2u-1)Z$ for $\frac{1}{2}\leq u\leq 1$.

Note that here we use the trick in \cite[Section 11.4]{Bi97} (cf. also 
\cite[Section 9.5]{BG00})
and the trivialization here is different from that in the proof of
Theorem \ref{thm:4.03} a) in Section \ref{s0602}. 
Under this new trivialization,
the identification between $\mathbf{H}_{y_0}$ and $\mathbf{H}_{Z_0}$ 
is an isometry with respect to (\ref{eq:1.19}).

For $Z_0\in N_{X^{g,K}/X^g,y_0}$, $|Z_0|\leq \var_0$, 
the considered trivializations depend explicitly on $Z_0$. 
We denote by 
$(\mB_{K,t,v})_{Z_0}$ the action of $\mB_{K,t,v}$ centered 
at $Z_0$.
Thus the operator $(\mB_{K,t,v})_{Z_0}$ acts 
on $\mathbf{H}_{Z_0}$. As $\mathbf{H}_{Z_0}$ is identified 
with $\mathbf{H}_{y_0}$, so that ultimately, 
$(\mB_{K,t,v})_{Z_0}$ acts on $\mathbf{H}_{y_0}$.

We may and we will assume that $\var_0$ is small enough so 
that if $|Z_0|\leq \var_0$, $|Z|\leq 4\var_0$, then 
\begin{align}\label{eq:6.19} 
\frac{1}{2}g^{TX}_{y_0}\leq g^{TX}_{Z_0+Z}\leq \frac{3}{2}
g^{TX}_{y_0}.
\end{align}

We define $k_{(y_0,Z_0)}'(Z)$ as in (\ref{eq:5.72}).
Recall that $\rho(Z)$ is defined 
in (\ref{eq:5.81}).

\begin{defn}\label{defn:6.05}
Let $L_{Z_0,K}'^{1,(t,v)}$ be the differential operator 
acting on $\mathbf{H}_{y_0}$,
\begin{align}\label{eq:6.21}
L_{Z_0,K}'^{1,(t,v)}=-(1-\rho^2(Z))(-t\Delta^{TX})+\rho^2(Z)
(\mB_{K,t,v})_{Z_0}.
\end{align}
\end{defn}

By proceeding as in (\ref{eq:5.83}), and using
Theorem \ref{thm:6.03} and (\ref{eq:6.17}), 
we find that if $Z_0\in N_{X^{g,K}/X^g,y_0}, Z\in N_{X^g/X,y_0}$, 
$|Z|, |Z_0|\leq \var_0$,
\begin{align}\label{eq:6.22}
\wi{F}_t(\mB_{K,t,v})(g^{-1}(Z_0,Z),(Z_0,Z))k_{(y_0,Z_0)}'(Z)
=\wi{F}_t(L_{Z_0,K}^{1,(t,v)})(g^{-1}Z,Z).
\end{align}

We still define $H_t$ as in (\ref{eq:5.85}). Let 
\begin{align}\label{eq:6.23} 
L_{Z_0,K}'^{2,(t,v)}=H_{t}^{-1}L_{Z_0,K}'^{1,(t,v)}H_{t}.
\end{align}

Let $(e_1,\cdots,e_{\ell'})$, $(e_{\ell'+1},\cdots,e_{\ell})$, 
$(e_{\ell+1},\cdots,e_{n})$ be orthonormal basis of 
$T_{y_0}X^{g,K}$, $N_{X^{g,K}/X^g,y_0}$, $N_{X^g/X,y_0}$ 
respectively. 

\begin{defn}\label{defn:6.07}
Let $L_{Z_0,K}'^{3,(t,v)}$ be the differential operator acting on
$\mathbf{H}_{y_0}$ obtained from $L_{Z_0,K}'^{2,(t,v)}$ by 
replacing $c(e_j)$ by $c_{t}(e_j)$ (cf. (\ref{eq:5.88})) 
for $1\leq j\leq \ell'$, 
by $c_{t/v}(e_j)$ for $\ell'+1\leq j\leq \ell$, while leaving
unchanged the $c(e_j)$'s for $\ell+1\leq j\leq n$.
\end{defn}

For $A\in (\pi^*(\Lambda (T^*B))\widehat{\otimes} c(TX)
\otimes\End( E))_x\otimes \mathrm{Op}_x$, we denote by
$[A]_{(t,v)}^{(3)}$ the differential operator obtained from $A$ 
by using the Getzler rescaling of the Clifford variables
which is given in Definition \ref{defn:6.07}.

If $Z_0\in N_{X^{g,K}/X^g,y_0}$, $|Z_0|\leq \var_0$, $Z\in 
T_{y_0}X$, $|Z|\leq 4\var_0$, if $U\in T_{y_0}X$, let 
$\tau^{Z_0}U(Z)\in TX_{Z_0+Z}$ be the parallel transport of 
$U$ along the curve $u\rightarrow 2uZ_0$, $0\leq u\leq \frac{1}{2}$, 
$u\rightarrow \exp_{Z_0}^X((2u-1)Z)$, $\frac{1}{2}\leq u\leq 1$, with 
respect to $\nabla^{TX}$.

By (\ref{eq:6.18a}), 
under the identification of $\pi^*\Lambda(T^*B)\otimes \mE_{Z_0+Z}$
and $\pi^*\Lambda(T^*B)\otimes \mE_{y_0}$ at the beginning of 
this subsection,
in the trivialization
\begin{align}\label{eq:6.25a} 
c\left(\tau^{Z_0} e_j(Z)\right)=c(e_j)+\frac{1}{\sqrt{t}}
\big(\la S(Z)e_j,f_p^H\ra_{Z_0}
+\mO(|Z|^2)\big) f^p\wedge.
\end{align}

Then comparing with (\ref{eq:5.89}) and (\ref{eq:5.90}),
from (\ref{eq:6.06}), we have
\begin{multline}\label{eq:6.25} 
L_{Z_0,K}'^{3,(t,v)}=-(1-\rho^2(\sqrt{t}Z))\Delta^{TX}
+\rho^2(\sqrt{t}Z)\cdot\left\{ 
-g^{ij}(\sqrt{t}Z)\left(\nabla_{e_i}''\nabla_{e_j}''
-\Gamma_{ij}^k(\sqrt{t}Z)\sqrt{t}\nabla_{e_k}'' 
\right)\right.
\\
+\frac{t}{2}\left(
R^{\mE/\mS}_{(Z_0,\sqrt{t}Z)}(e_i,e_j)-\frac{1}{2v}
\la\nabla_{e_i}^{TX}K^X,e_j\ra_{(Z_0,\sqrt{t}Z)}
\right)\left[c\left(\tau^{Z_0} e_i(\sqrt{t}Z)\right)
c\left(\tau^{Z_0} e_j(\sqrt{t}Z)\right)\right]_{(t,v)}^3
\\
+\sqrt{t}\left(R^{\mE/\mS}_{(Z_0,\sqrt{t}Z)}(e_i,f_p^H)-\frac{1}{2v}
\la T(e_i,f_p^H),K^X\ra_{(Z_0,\sqrt{t}Z)}\right)
\left[c\left(\tau^{Z_0} e_i(\sqrt{t}Z)\right)
\right]_{(t,v)}^3f^p\wedge
\\
+\frac{1}{2}\left(R^{\mE/\mS}_{(Z_0,\sqrt{t}Z)}(f_p^H, f_q^H)-
\frac{1}{8v}\la T(f_p^H,f_q^H), K^X
\ra_{(Z_0,\sqrt{t}Z)}\right)f^p\wedge f^q\wedge 
\\
\left.+\frac{t}{4}H_{(Z_0,\sqrt{t}Z)}
-m_{(Z_0,\sqrt{t}Z)}^{\mE/\mS}(K^X)+\frac{1}{4v}
|K^X(Z_0,\sqrt{t}Z)|^2\right\},
\end{multline}
where 
\begin{multline}\label{eq:6.26}
\nabla_{e_i}''=\nabla_{\tau^{Z_0}e_i(\sqrt{t}Z)}+\frac{t}{8}\Big(
\la R^{TX}_{Z_0}(e_k,e_l)Z,
e_i\ra+\mathcal{O}(\sqrt{t}|Z|^2)\Big)
\left[c\left(\tau^{Z_0} e_k(\sqrt{t}Z)\right)
c\left(\tau^{Z_0} e_l(\sqrt{t}Z)\right)\right]_{(t,v)}^3
\\
+\frac{\sqrt{t}}{4}\Big(\la R^{TX}_{Z_0}
(e_k,f_p^H)Z,e_i\ra
+\mathcal{O}(\sqrt{t}|Z|^2)\Big) 
\left[c\left(\tau^{Z_0} e_k(\sqrt{t}Z)\right)
\right]_{(t,v)}^3f^p\wedge 
\\
+\frac{1}{8} \Big(\la R^{TX}_{Z_0}
(f_p^H,f_q^H)Z,e_i\ra
+\mathcal{O}(\sqrt{t}|Z|^2)\Big) f^p\wedge f^q\wedge
+ \frac{t}{2}\left(R^E_{Z_0}(Z, e_i)+\mathcal{O}(\sqrt{t}|Z|^2) \right) 
\\
-\frac{1}{4}\left(1+\frac{t}{v}\right)\la m^{TX}_{Z_0}
(K)Z,e_i\ra
+\sqrt{t}h_i(K,\sqrt{t}Z)\left(\frac{1}{t}+\frac{1}{v}\right).
\end{multline}
Here $h_i(K, Z)$ is a function depending linearly on $K$ and 
$h_i(K,Z)=\mathcal{O}(|Z|^2)$ for $|K|$ bounded.

Let $\psi_v\in \End(\Lambda(T^*X^g))$ be the morphism of 
exterior algebras such that 
\begin{align}\label{eq:6.31}
\begin{split}
&\psi_v(e^j)=e^j,\quad 1\leq j\leq \ell',
\\
&\psi_v(e^j)=\sqrt{v}e^j,\quad \ell'+1\leq j\leq \ell.
\end{split}
\end{align}
Recall that for $x=(y_0,Z_0)\in X^g$, $\Lambda(T^*X^g)_{(y_0,
Z_0)}$ has been identified with $\Lambda(T^*X^g)_{y_0}$.

\begin{defn}\label{defn:6.10} 
	Let $L_{x,K}'^{3,(0,v)}$ be the operator
	\begin{align}\label{eq:6.32} 
	L_{x,K}'^{3,(0,v)}=\psi_v L_{x,K}^{3,(0,v)}\psi_v^{-1}.
	\end{align}
\end{defn}

By Definitions \ref{defn:6.07} and \ref{defn:6.10}, (\ref{eq:6.18d}) 
and (\ref{eq:6.31}),
as $t\rightarrow 0$, 
\begin{align}\label{eq:6.33} 
L_{Z_0,K}'^{3,(t,v)}\rightarrow L_{(y_0,Z_0),K}'^{3,(0,v)}.
\end{align}

\subsection{A family of norms}\label{s0606}

For $0\leq p\leq \ell'$, $0\leq q\leq \ell-\ell'$, put
\begin{align}\label{eq:6.34} 
\Lambda^{(p,q)}(T^*X^g)_{y_0}=\Lambda^{p}(T^*X^{g,K})_{y_0}
\widehat{\otimes}\Lambda^{q}(N^*_{X^{g,K}/X^g})_{y_0}.
\end{align}
The various $\Lambda^{(p,q)}(T^*X^g)_{y_0}$ are mutually 
orthogonal in $\Lambda(T^*X^g)_{y_0}$. Let $\mathbf{I}_{y_0}$ 
be the vector space of smooth sections of
$(\pi^*\Lambda(T^*B)\otimes\Lambda(T^*X^g)\widehat{\otimes}
\mS_N\otimes E)_{y_0}$ on $T_{y_0}X$, let
$\mathbf{I}_{(r,p,q),y_0}$ be the vector space of smooth 
sections of $(\pi^*\Lambda^r(T^*B)\otimes\Lambda^{(p,q)}(T^*X^g)
\widehat{\otimes}\mS_N\otimes E)_x$ on $T_{y_0}X$. 
Let $\mathbf{I}^0_{y_0}, \mathbf{I}^0_{(r,p,q),y_0}$ be the 
corresponding vector spaces of square-integrable sections.

Now we imitate constructions in \cite[\S 11]{BL91}.
Recall that $\dim B=k$.

\begin{defn}\label{defn:6.12}
For $t\in [0,1]$, $v\in \R_+^*$, $y_0\in X^{g,K}$, $Z_0\in 
N_{X^{g,K}/X^g,y_0}$, $|Z_0|\leq \var_0/\sqrt{v}$,
$s\in \mathbf{I}_{(r,p,q),y_0}$, set
\begin{multline}\label{eq:6.35}
|s|_{t,v,Z_0,0}^2=\int_{T_{y_0}X}|s(Z)|^2\left(1+\big(|Z_0|+|Z|\big)
\rho\left(\frac{\sqrt{t}Z}{2}\right) \right)^{2(k+\ell'-p-r)}
\\
\cdot\left(1+\sqrt{v}|Z|\rho\left(\frac{\sqrt{t}Z}{2}
\right) \right)^{2(\ell-\ell'-q)}dv_{TX}(Z).
\end{multline}
\end{defn}
Then (\ref{eq:6.35}) induces a Hermitian product $\la\cdot,
\cdot\ra_{t,v,Z_0,0}$ on $\mathbf{I}^0_{(r,p,q),y_0}$. We equip 
$\mathbf{I}^0_{y_0}=\bigoplus \mathbf{I}^0_{(r,p,q),y_0}$ with the 
direct sum of these Hermitian metrics.

Recall that by (\ref{eq:5.81}), if $\rho(\sqrt{t}Z)>0$, then 
$|\sqrt{t}Z|\leq 4\var_0$. The proof of the following proposition 
is almost the same as that of
\cite[Proposition 8.16]{BG00}
(cf. also \cite[Proposition 11.24]{BL91}). 

\begin{prop}\label{prop:6.13}
For $t\in (0,1]$, $v\in [t,1]$, $y_0\in X^{g,K}$, $Z_0\in 
N_{X^{g,K}/X^g,y_0}$, $|Z_0|\leq \var_0/\sqrt{v}$, the 
following family of operators acting on $(\mathbf{I}_{y_0}^0, 
|\cdot|_{t,v,Z_0,0})$ are uniformly bounded:
\begin{multline}\label{eq:6.36}
1_{|\sqrt{t}Z|\leq 4\var_0}\sqrt{t}c_{t}(e_j),\  1_{
|\sqrt{t}Z|\leq 4\var_0}|Z|\sqrt{t}c_{t}(e_j),\ 
1_{|\sqrt{t}Z|\leq 4\var_0}|Z_0|\sqrt{t}c_{t}(e_j),\, 
\text{for}\ \ 1\leq j\leq \ell',
\\
1_{|\sqrt{t}Z|\leq 4\var_0}|Z_0|f^p\wedge,\quad 
1_{|\sqrt{t}Z|\leq 4\var_0}|Z|f^p\wedge,\ 
\\
1_{|\sqrt{t}Z|\leq 4\var_0}\sqrt{\frac{t}{v}}
c_{\frac{t}{v}}(e_j),\quad  1_{
|\sqrt{t}Z|\leq 4\var_0}|Z|\sqrt{t}c_{\frac{t}{v}}(e_j),\  \, 
\text{for}\ \ell'+1\leq j\leq \ell.
\end{multline}
\end{prop}

\begin{defn}\label{defn:6.14} 
For $t\in [0,1]$, $v\in \R_+^*$, $y_0\in X^{g,K}$, $Z_0\in 
N_{X^{g,K}/X^g,y_0}$, $|Z_0|\leq \var_0/\sqrt{v}$,
if $s\in \mathbf{I}_{y_0}$ has compact support, set
\begin{align}\label{eq:6.37} 
|s |_{t,v,Z_0,1}^2=|s|_{t,v,Z_0,0}^2+\frac{1}{v}
\left|\rho(\sqrt{t}Z)|K^X|(\sqrt{v}Z_0+\sqrt{t}Z)s
\right|_{t,v,Z_0,0}^2+\sum_{i=1}^{n}|\nabla_{e_i}s|_{t,v,
	Z_0,0}^2.
\end{align}
\end{defn}

Note that $|s |_{t,v,Z_0,1}$ depends explicitly on $K=zK_0$. 
In fact, $|s |_{t,v,Z_0,1}$ depends on $z\in \R^*$.

\begin{thm}\label{thm:6.15}
There exist constants $C_i>0$, $i=1,2,3,4$, such that if 
$t\in (0,1]$, $v\in [t,1]$, $n\in\N$, $y_0\in X^{g,K}$, 
$Z_0\in N_{X^{g,K}/X^g,y_0}$, $|Z_0|\leq \var_0/\sqrt{v}$, 
$z\in \R$, $|z|\leq 1$, and if the support of $s,s'\in 
\mathbf{I}_{y_0}$ is included in $\{Z\in T_{y_0}X: |Z|\leq n \}$, 
then
\begin{align}\label{eq:6.38}
\begin{split}
&\Re\la L_{\sqrt{v}Z_0,zK_0}'^{3,(t,v)}s,s\ra_{t,v,Z_0,0}\geq
C_1|s|_{t,v,Z_0,1}^2-C_2(1+|nz|^2)|s|_{t,v,Z_0,0}^2,
\\
&|\Im\la L_{\sqrt{v}Z_0,zK_0}'^{3,(t,v)}s,s\ra_{t,v,Z_0,0}|\leq
C_3((1+|nz|)|s|_{t,v,Z_0,1}|s|_{t,v,Z_0,0}+|nz|^2|s|_{t,v,Z_0,0}^2),
\\
&|\la L_{\sqrt{v}Z_0,zK_0}'^{3,(t,v)}s,s'\ra_{t,v,Z_0,0}|\leq
C_4(1+|nz|^2)|s|_{t,v,Z_0,1}|s'|_{t,v,Z_0,1}.
\end{split}
\end{align}
\end{thm}
\begin{proof}
Comparing with $L_{x,K}^{3,t}$ in (\ref{eq:5.89}) and 
(\ref{eq:5.106}), there are four additional terms in (\ref{eq:6.25}) 
which should be estimated:
\begin{align}\label{eq:6.39} 
\frac{1}{4v}\left|\rho(\sqrt{t}Z)|zK_0^X|(\sqrt{v}Z_0+\sqrt{t}Z)s
\right|_{t,v,Z_0,0}^2,
\end{align}
\begin{multline}\label{eq:6.40} 
-\rho^2(\sqrt{t}Z)\frac{t}{4v}\left\la \la\nabla^{TX}_{
\tau^{\sqrt{v}Z_0}e_i(\sqrt{t}Z)}zK_0^X(\sqrt{v}Z_0+\sqrt{t}Z),
\tau^{\sqrt{v}Z_0}e_j(\sqrt{t}Z))\ra\right.
\\
\left. \cdot \left[c\left(\tau^{\sqrt{v}Z_0} e_i(\sqrt{t}Z)\right)
c\left(\tau^{\sqrt{v}Z_0} e_j(\sqrt{t}Z)\right)\right]_{(t,v)}^3s,s
\right\ra_{t,v,Z_0,0},
\end{multline}
\begin{align}\label{eq:6.41}
-\rho^2(\sqrt{t}Z)\frac{\sqrt{t}}{2v}\left\la\la T(e_i,f_p^H),
zK_0^X\ra (\sqrt{v}Z_0+\sqrt{t}Z) \left[c\left(\tau^{\sqrt{v}Z_0} 
e_i(\sqrt{t}Z)\right)
\right]_{(t,v)}^3f^p\wedge s,
s\right\ra_{t,v,Z_0,0},
\end{align}
and
\begin{align}\label{eq:6.42}
-\rho^2(\sqrt{t}Z)\frac{1}{16v}\left\la\la T(f_p^H, f_q^H),
zK_0^X\ra (\sqrt{v}Z_0+\sqrt{t}Z)f^p\wedge f^q\wedge s,s
\right\ra_{t,v,Z_0,0}.
\end{align}
	
The first term is controlled by (\ref{eq:6.37}) and the second 
term was estimated in the proof of \cite[Theorem 8.18]{BG00}. 
We only need to estimate (\ref{eq:6.41}) and (\ref{eq:6.42}), which 
are new terms in the family case.

By (\ref{eq:3.04}), $\wi{T}$ is $G$-invariant, 
thus $[K^X,\wi{T}]=0$. 
Since $m^{TX}(K)$ is skew-adjoint, by (\ref{eq:2.03}),
\begin{align}\label{eq:6.60b} 
Z\la \wi{T}, K^X\ra
=\la \nabla_Z^{TX}\wi{T},K^X\ra
+\la\wi{T},\nabla^{TX}_{Z}K^X\ra
=\la \nabla_Z^{TX}\wi{T},K^X\ra
-\la \nabla_{K^X}^{TX}\wi{T},
Z\ra.
\end{align}
As $y_0\in X^{g,K}\subset X^K$, we know $K_{y_0}^X=0$.
Thus from (\ref{eq:6.60b}), we have
\begin{align}\label{eq:6.61} 
\frac{\partial}{\partial s}\la \wi{T},
K^X\ra_{(y_0,sZ)}|_{s=0}=0.
\end{align}
From (\ref{eq:6.61}), we have
\begin{align}\label{eq:6.62}
\la \wi{T},K^X\ra_{(y_0,Z)}=\mO(|Z|^2).
\end{align}
 Thus we have
\begin{multline}\label{eq:6.43}
\rho^2(\sqrt{t}Z)\frac{\sqrt{t}}{2v}\la T(e_i,f_p^H),zK_0^X\ra 
(\sqrt{v}Z_0+\sqrt{t}Z)\left[c\left(\tau^{\sqrt{v}Z_0} 
e_i(\sqrt{t}Z)\right)
\right]_{(t,v)}^3f^p\wedge
\\
=\rho^2(\sqrt{t}Z)v^{-1}\sqrt{t}|z|\left[c\left(\tau^{\sqrt{v}Z_0} 
e_i(\sqrt{t}Z)\right)
\right]_{(t,v)}^3f^p\wedge\cdot 
\mathcal{O}((\sqrt{v}|Z_0|+\sqrt{t}|Z|)^2),
\end{multline}
and
\begin{multline}\label{eq:6.44}
\rho^2(\sqrt{t}Z)\frac{1}{16v}\la T(f_p^H, f_q^H),zK_0^X\ra 
(\sqrt{v}Z_0+\sqrt{t}Z)f^p\wedge f^q\wedge
\\
=\rho^2(\sqrt{t}Z)|z|f^p\wedge f^q\wedge
\cdot \mathcal{O}((|Z_0|
+\sqrt{t/v}|Z|)^2).
\end{multline}
Using the fact that $v\leq 1$ and $t/v\leq 1$ and also 
Proposition \ref{prop:6.13}, from (\ref{eq:6.25a}), we find 
that the operators in (\ref{eq:6.43}) and (\ref{eq:6.44}) remain 
uniformly bounded with respect to 
$|\cdot|_{t,v,Z_0,0}$. 
	
The proof of Theorem \ref{thm:6.15} is completed.
\end{proof}

\begin{defn}\label{defn:6.16}
Put
\begin{align}\label{eq:6.45}
L_{Z_0,K,n}^{3,(t,v)}=-\left(1-\gamma^2\left(\frac{|Z|}{2(n+2)}
\right)\right)\Delta^{TX}+\gamma^2\left(\frac{|Z|}{2(n+2)}
\right)L_{Z_0,K}^{3,(t,v)}.
\end{align}
\end{defn}

Let $\wi{F}_t(L_{Z_0,K}^{3,(t,v)})(Z,Z')$ and 
$\wi{F}_t(L_{Z_0,K,n}^{3,(t,v)})(Z,Z')$ be the smooth kernels 
associated with $\wi{F}_t(L_{Z_0,K}^{3,(t,v)})$ and
$\wi{F}_t(L_{Z_0,K,n}^{3,(t,v)})$ with respect 
to $dv_{TX}(Z')$.
Using (\ref{eq:6.19}) and proceeding as in (\ref{eq:5.122}), 
i.e., using finite propagation
speed, we see that if $Z\in T_{y_0}X$, $|Z|\leq p$,
\begin{align}\label{eq:6.46}
\wi{F}_{t,n}(L_{Z_0,K}^{3,(t,v)})(Z,Z')=\wi{F}_{t,n}
(L_{Z_0,K,n+p}^{3,(t,v)})(Z,Z').
\end{align}

Clearly, when replacing $L_{\sqrt{v}Z_0,zK_0}^{3,(t,v)}$ in 
(\ref{eq:6.38}) by $L_{\sqrt{v}Z_0,zK_0,n}^{3,(t,v)}$, the 
estimates (\ref{eq:6.38}) still hold. 

\subsection{A Proof of Theorem \ref{thm:6.04}}\label{s0607}

Since $W$ is a compact manifold, there exists a finite 
family of smooth functions $f_1,\cdots,f_r:W\rightarrow [-1,1]$ 
which have the following properties:
\begin{itemize}
\item $W^K=\cap_{j=1}^r\{x\in W: f_j(x)=0 \}$;
\item On $W^K$, $df_1\cdots, df_r$ span $N_{X^{g,K}/X}$.
\end{itemize}

\begin{defn}\label{defn:6.17} 
Let $\mQ_{t,v,Z_0}$ be the family of operators
\begin{align}\label{eq:6.47} 
\mQ_{t,v,Z_0}=\left\{\nabla_{e_i}, 1\leq i\leq \dim X; 
\frac{z}{\sqrt{v}}\rho(\sqrt{t}Z)f_j(\sqrt{v}Z_0+\sqrt{t}Z), 
1\leq j\leq r \right\}.
\end{align}
For $j\in\N$, let $\mQ_{t,v,Z_0}^j$ be the set of operators 
$Q_1\cdots Q_j$, with $Q_i\in \mQ_{t,v,Z_0}$, $1\leq i\leq j$.
\end{defn}

Following the arguments in \cite[\S 8.8-8.10]{BG00}, we have the 
following uniform estimate, which is formally the same as 
\cite[(8.76)]{BG00}.
We only need to take care that in the proof of the analogue of 
\cite[Proposition 8.22 and Theorems 8.23, 8.24]{BG00}, there are two 
new terms like (\ref{eq:6.41}) and (\ref{eq:6.42}) appear in our 
family case. However, they are easy to be controlled as in 
(\ref{eq:6.43}) and (\ref{eq:6.44}).

\begin{thm}
There exist $C>0$, $C'>0$ such that given $m>0$, there exists 
$\beta_1>0$ such that if $t\in (0,1]$, $v\in [t,1]$, $z\in \R$, 
$|z|\leq \beta_1$, $y_0\in X^{g,K}$, $Z_0\in N_{X^{g,K}/X^g, y_0},
|Z_0|\leq \varepsilon_0/\sqrt{v}$,  $Z\in N_{X^g/X,y_0}$, 
$|Z|\leq \var_0/\sqrt{t}$,
 \begin{multline}\label{eq:6.48} 
 \left|\left(\widetilde{F}_{t}\left(L_{\sqrt{v}Z_0,zK_0}^{3,(t,v)}
 \right)
 -\exp\left(-L_{\sqrt{v}Z_0,zK_0}^{3,(0,v)}\right)
 \right)(g^{-1}Z,Z) \right|
 \\
 \leq 
 C\left(\frac{t}{v}\right)^{\frac{1}{4(\dim X+1)}}
 \cdot\frac{(1+|Z_0|)^{\ell'+1}}{(1+|zZ_0|)^m}\exp\left(
 -C'|Z|^2/4 \right).
 \end{multline}
\end{thm}
The kernel $\exp(-L_{\sqrt{v}Z_0,zK_0}^{3,(0,v)}))(g^{-1}Z,Z)$ 
here is defined in the same way as in (\ref{eq:5.127}).

From (\ref{eq:6.25a}), we get
\begin{align}\label{eq:6.50b} 
\sqrt{\frac{t}{v}}\left[c\left(\tau^{\sqrt{v}Z_0} 
e_j(\sqrt{t}Z)\right)
\right]_{(t,v)}^3=\left\{
\begin{aligned}
&\frac{1}{\sqrt{v}}e^j\wedge +\sqrt{\frac{t}{v}}\Big(\mO(\sqrt{t})+
\mO(|Z|)\Big),\ &\hbox{if $1\leq j\leq \ell'$;} \\
&e^j\wedge +\sqrt{\frac{t}{v}}
\mO(1+|Z|),\quad&\hbox{if $\ell'+1\leq j\leq \ell$;} \\
&\sqrt{\frac{t}{v}}\big(c(e_j)+\mO(|Z|) \big),
\quad&\hbox{if $\ell+1\leq j\leq n$.}
\end{aligned}
\right.
\end{align}
Moreover as $K^X$ vanishes on $W^{g,K}$, we have
\begin{align}\label{eq:6.51b} 
\begin{split}
\la K_0^X(\sqrt{v}Z_0+\sqrt{t}Z),\tau^{\sqrt{v}Z_0} 
e_j(\sqrt{t}Z)\ra &=\la K_0^X(\sqrt{v}Z_0),\tau^{\sqrt{v}Z_0} 
e_j\ra_{(y_0,\sqrt{v}Z_0)}+\mO(\sqrt{t}|Z|),
\\
\la K_0^X(\sqrt{v}Z_0),\tau^{\sqrt{v}Z_0} 
e_j\ra_{(y_0,\sqrt{v}Z_0)}&=\mO(\sqrt{v}|Z_0|).
\end{split}
\end{align}
By (\ref{eq:3.01}), (\ref{eq:6.31}), (\ref{eq:6.50b}) and (\ref{eq:6.51b}), 
we get
\begin{align}\label{eq:6.52b} 
\frac{\sqrt{t}}{4v}\left[c\left(K^X_{\sqrt{v}|Z_0|+
	\sqrt{t}|Z|}\right)
\right]_{(t,v)}^3=\left(\psi_v \frac{\jmath^*\vartheta_K}{4v} 
\psi_v^{-1}\right)_{\sqrt{v}Z_0}+\frac{\sqrt{t}}{v}z
\mO(\sqrt{v}|Z_0|+\sqrt{t}|Z|)\mO(1+|Z|).
\end{align}
Note that we have 
$(\delta_v^*\alpha)_{Z_0}=(\psi_v\alpha\psi_v^{-1})_{\sqrt{v}Z_0}$ 
for any $\alpha\in \Omega(W^g)$ with $\delta_v$ 
defined above (\ref{eq:3.10}). 
Therefore, from (\ref{eq:3.17}), (\ref{eq:6.48}) and (\ref{eq:6.52b}),
we get Theorem \ref{thm:6.04}.

\subsection{A Proof of Theorem \ref{thm:4.03} b)}\label{s0608}

Theorem \ref{thm:4.03} b) follows directly from the following 
theorem.

\begin{thm}\label{thm:6.19} 
There exist $\beta_1>0$, $r\in \N$, $C>0$, $\delta\in (0,1]$, 
such that if $t\in (0,1]$, $v\in [t,1]$, if $z\in \R\backslash \{0\}$, 
$|z|\leq \beta_1$, then
\begin{align}\label{eq:6.51} 
|z|^r\left|\phi\wi{\tr}'\left[g\frac{\sqrt{t}c(K^X)}{4v}
\exp\left(-\mB_{K,t,v}\right) \right]+\wi{e}_v\right|
\leq C\left(\frac{t}{v} \right)^{\delta}.
\end{align}
\end{thm}
\begin{proof} 
Recall that $\mU_{\var}$, $\mU_{\var}'$, $\mU_{\var}''$ are 
$\var$-neighborhoods of $X^g$, $X^{g,K}$, $X^{g,K}$
in $N_{X^g/X}$, $N_{X^{g,K}/X}$, $N_{X^{g,K}/X^g}$
respectively.
Let $\bar{k}(y_0,Z_0)$ be the function defined on $X^g\cap 
\mU_{\var}''$ by the relation 
\begin{align}\label{eq:6.52} 
dv_{X^g}(y_0,Z_0)=\bar{k}(y_0,Z_0)dv_{X^{g,K}}(y_0)
dv_{N_{X^{g,K}/X^g}}(Z_0).
\end{align}
Then 
\begin{align}\label{eq:6.53} 
\bar{k}|_{X^{g,K}}=1.
\end{align}
	
Recall that $\wi{F}_t(\mB_{K,t,v})(g^{-1}x,x)$ vanishes 
on $X\backslash \mU_{\var_0}$. Using (\ref{eq:5.72}), 
(\ref{eq:6.52}), we get
\begin{multline}\label{eq:6.54} 
\phi\int_{\mU_{\var_0}'}\wi{\tr}'\left[g\frac{\sqrt{t}c(K^X)}{
4v}\exp\left(-\mB_{K,t,v}\right)(g^{-1}x,x) \right]dv_X(x)
+\int_{X^g\cap\mU_{\var_0}'} \{\gamma_{K,v}\}^{\max} dv_{X^g}
\\
=\int_{y_0\in X^{g,K}}v^{(\ell-\ell')/2}\int_{|Z_0|\leq 
\var_0/\sqrt{v}} \left[\phi\int_{|Z|\leq \var_0}\wi{\tr}'
\left[g\frac{\sqrt{t}c(K^X)}{4v}\exp\left(-\mB_{K,t,v}
\right)\right.\right. 
\\
(g^{-1}(y_0,\sqrt{v}Z_0,Z),(y_0,\sqrt{v}Z_0,Z))
\Big]\cdot k(y_0,\sqrt{v}Z_0,Z)dv_{N_{X^g/X}}(Z) 
\\
+\{\gamma_{K,v}\}^{\max}(y_0,\sqrt{v}Z_0) \Big]
\bar{k}(y_0,\sqrt{v}Z_0) dv_{N_{X^{g,K}/X^g}}(Z_0)dv_{X^{g,K}}
(y_0).
\end{multline}
	
Using Theorem \ref{thm:6.04} and (\ref{eq:6.54}), we find that 
there exist $C>0$ and $\beta_1>0$ such that for $z\in \R^*$, 
$|z|\leq \beta_1$,
\begin{multline}\label{eq:6.55} 
|z|^{\ell+1}\left|\phi\int_{\mU_{\var_0}'}
\wi{\tr}'\left[g\frac{\sqrt{t}c(K^X)}{4v}\exp\left(-\mB_{K,t,v}
\right)(g^{-1}x,x) \right]dv_X(x)+\int_{X^g\cap\mU_{\var_0}'} 
\gamma_{K,v} \ 
\right|
\\
\leq C|z|^{\ell-\ell'}\int_{y_0\in X^{g,K}}\int_{Z_0\in N_{X^{g,K}/X^g}, 
	|Z_0|<\var_0}(1+|zZ_0|)^{-(\ell-\ell')-1}dZ_0\cdot 
\left(\frac{t}{v}\right)^{\delta}
\leq C\left(\frac{t}{v}\right)^{\delta}.
\end{multline}
	
Similar estimates can be obtained for 
\begin{align}\label{eq:6.56} 
\left|\phi\int_{X\backslash\mU_{\var_0}'}\wi{\tr}'\left[g\frac{
\sqrt{t}c(K^X)}{4v}\exp\left(-\mB_{K,t,v}\right)(g^{-1}x,x) 
\right]dv_X(x)+\int_{X^g\backslash\mU_{\var_0}'} 
\gamma_{K,v} \right|.
\end{align}
In fact, on $X\backslash\mU_{\var_0}'$, we observe that 
$|K^X|^2/2v$ has a positive lower bound. Then we 
adopt the 
above techniques to the case where $X^{g,K}=\emptyset$. 
The potentially annoying term $\frac{\sqrt{t}c(K^X)}{4v}$  
can be controlled by the term $|K^X|^2/2v$.
	
The proof of Theorem \ref{thm:6.19} is completed.
\end{proof}

\subsection{A Proof of Theorem \ref{thm:4.03} c)}\label{s0609}

When $v\in [1,+\infty)$, $\frac{1}{v}$ remains bounded. By using the 
methods of the last section and of the present section, one 
sees easily that for $K_0\in \mathfrak{z}(g)$, $K=zK_0$, 
there exist $C>0$, $\beta>0$ such that for $t\in (0,1]$, 
$v\in [1,+\infty)$, $|zK_0|<\beta$, we have
\begin{align}\label{eq:6.57} 
\left|\wi{\tr}'\left[g\sqrt{t}c(K^X)\exp\left(-\mB_{K,t,v}
\right) \right]\right|\leq C,
\end{align} 
which is equivalent to Theorem \ref{thm:4.03} c).

The proof of Theorem \ref{thm:4.03} c) is completed.

\subsection{A Proof of Theorem \ref{thm:4.03} d)}\label{s0610}

In this subsection, we will prove Theorem \ref{thm:4.03} d)
by using the method in \cite[\S 9]{BG00}. Since the singular 
term there does not appear here, our proof is in fact 
much easier. 

We fix $g\in G$, $0\neq K_0\in \mathfrak{z}(g)$,
and take $K=zK_0$ with $z\in \R^*$. 

From Theorem \ref{thm:6.01}, we have
\begin{multline}\label{eq:6.58}
\mB_{K,t,tv}=-t\left(\nabla_{e_i}^{\mE}
+\frac{1}{2\sqrt{t}}\la S(e_i)e_j, f_p^H\ra c(e_j)f^p
\wedge\right.
\\
\left.+\frac{1}{4t}\la S(e_i)f_p^H, f_q^H\ra f^p\wedge
f^q\wedge-\frac{\la K^X,e_i\ra}{4t}\left(1
+\frac{1}{v}\right)\right)^2
\\
+\frac{t}{4}H+\frac{t}{2}\left(R^{\mE/\mS}(e_i,e_j)-\frac{1}{2vt}
\la\nabla_{e_i}^{TX}K^X,e_j\ra\right)c(e_i)c(e_j)
\\
+\sqrt{t}\left(R^{\mE/\mS}(e_i,f_p^H)-\frac{1}{2vt}\la 
T(e_i,f_p^H),K^X\ra\right)c(e_i)f^p\wedge
\\
+\frac{1}{2}\left(R^{\mE/\mS}(f_p^H,f_q^H)-\frac{1}{8vt}\la 
T(f_p^H,f_q^H), K^X\ra\right)f^p\wedge
f^q\wedge-m^{\mE/\mS}(K^X)+\frac{1}{4vt}|K^X|^2.
\end{multline}

As in sections \ref{s0503} and \ref{s0603}, the proof of
Theorem \ref{thm:4.03} d) can be localized near $X^g$. 
In the following,
we will concentrate on the estimates near $X^{g,K}$. 
As in (\ref{eq:6.56}),  the proof of the estimates 
near $X^g$ and far from $X^{g,K}$ is much easier.

We may assume that for $\var_0$ taken in Section \ref{s0604},
if $\var\in (0,8\var_0]$, the map $(y_0,Z)\in N_{X^{g,K}/X}
\rightarrow \exp_{y_0}^X(Z)\in X$ induces a diffeomorphism
from the $\var$-neighborhood $\mU_{\var}'$ of $X^{g,K}$
in $N_{X^{g,K}/X}$ on the tubular neighborhood $\mV_{\var}'$
of $X^{g,K}$ in $X$ as in the proof of Theorem \ref{thm:6.19}.

As in (\ref{eq:5.74}) and (\ref{eq:6.18}), we put
\begin{multline}\label{eq:6.59}
\,^3\nabla^{\mE,t}_{\cdot}:=\nabla^{\mE}_{\cdot}+\frac{1}{
	2\sqrt{t}}\la S(\cdot)e_j, f_p^H\ra c(e_j)f^p\wedge
\\
+\frac{1}{4t}\la S(\cdot)f_p^H, f_q^H\ra f^p\wedge
f^q\wedge
-\frac{\vartheta_K(\cdot)}{4t}\left(1+\frac{1}{v}\right).
\end{multline}

Take $y_0\in W^{g,K}$ in (\ref{eq:2.11}). 
If $Z\in N_{X^{g,K}/X,y_0}$, 
$|Z|\leq 4\var_0$, we identify 
$\pi^*\Lambda(T^*B)\widehat{\otimes}\mE_Z$ with 
$\pi^*\Lambda(T^*B)\widehat{\otimes}\mE_{y_0}$
by parallel transport with respect to the connection 
$\,^3\nabla^{\mE,t}$ along the curve $u\in [0,1]
\rightarrow uZ$.

Recall that $\rho$ is the cut-off function in (\ref{eq:5.81}).
Let  
\begin{align}\label{eq:6.60}
L_{y_0,K}^{1,(t,v)}=(1-\rho^2(Z))(-t\Delta^{TX})+\rho^2(Z)
(\mB_{K,t,tv}).
\end{align}
We still define $H_t$ as in (\ref{eq:5.85}) and define 
$L_{y_0,K}^{2,(t,v)}$ as in (\ref{eq:5.86}) from 
$L_{y_0,K}^{1,(t,v)}$. Let $L_{y_0,K}^{3,(t,v)}$
be the operator obtained from $L_{y_0,K}^{2,(t,v)}$
by replacing $c(e_j)$ by $c_t(e_j)$ as in (\ref{eq:5.88}) for
$1\leq j\leq \ell'$ (cf. (\ref{eq:2.13})), 
while leaving the $c(e_j)$'s unchanged for 
$\ell'+1\leq j\leq n$.

As in (\ref{eq:3.19}), we have
\begin{align}\label{eq:6.63} 
|K^X(y_0,Z)|^2=|m_{y_0}^{TX}(K)Z|^2+\mO(|Z|^3).
\end{align}

Let $\jmath':W^{g,K}\rightarrow W$ be the obvious embedding.
Put
\begin{multline}\label{eq:6.64} 
L_{y_0,K}^{3,(0,v)}=-\left(\nabla_{e_i}+\frac{1}{4}\left\la
\left(\jmath'^*R^{TX}-\left(1+\frac{1}{v}\right)m^{TX}(K)\right)Z,
e_i\right\ra\right)^2
\\
+\jmath'^*R^E_{y_0}-m^{E}(K)_{y_0}
-\frac{1}{4v}\sum_{j,k\geq \ell'+1}
\la m^{TX}(K)e_j,e_k\ra_{y_0} c(e_j)c(e_k)
\\
+\frac{1}{4v}\langle \jmath'^*R^{TX}(m^{TX}(K)Z), Z\rangle_{y_0}
+\frac{1}{4v}|m_{y_0}^{TX}(K)Z|^2.
\end{multline}	
From (\ref{eq:3.05}), (\ref{eq:3.18}), (\ref{eq:3.21}), (\ref{eq:6.62}), 
(\ref{eq:6.58}) and (\ref{eq:6.64}), as Proposition
\ref{prop:5.24}, we have
\begin{align}\label{eq:6.65} 
L_{y_0,K}^{3,(t,v)}\rightarrow L_{y_0,K}^{3,(0,v)}.
\end{align}

Now we take a new trivialization as in Section \ref{s0605}.
Take $Z_0\in N_{X^{g,K}/X^g,y_0}, |Z_0|\leq \var_0$. If $Z\in 
T_{y_0}X$, $|Z|\leq 4\var_0$, we identify 
$\pi^*\Lambda(T^*B)\widehat{\otimes}\mE_{Z+Z_0}$ with 
$\pi^*\Lambda(T^*B)\widehat{\otimes}\mE_{Z_0}$ 
by parallel transport along the curve 
$u\in[0,1]\rightarrow \exp_{Z_0}^X(uZ)$ with respect to 
the connection $\,^3\nabla^{\mE,t}$. Also we identify 
$\pi^*\Lambda(T^*B)\widehat{\otimes}\mE_{Z_0}$
with $\pi^*\Lambda(T^*B)\widehat{\otimes}\mE_{y_0}$ 
by parallel transport along the curve 
$u\in[0,1]\rightarrow uZ_0$ with respect to 
the connection $\nabla^{\mE}$.
Using this trivialization,
the analogues of \cite[Theorems 9.19 and 9.22]{BG00} hold here 
following the same arguments except for replacing the norm 
in \cite[(9.43)]{BG00} by
\begin{align}
|s|_{t,Z_0,0}^2=\int_{T_{y_0}X}|s(Z)|^2\left(1+(|Z|+|Z_0|)\,\rho\left(
\frac{\sqrt{t}Z}{2}\right) \right)^{2(k+\ell'-p-r)}dv_{TX}(Z).
\end{align}
Here $s$ is a square integrable section of $\left(\pi^*\Lambda^r
(T^*B)\widehat{\otimes}\Lambda^p(T^*X^{g,K})\widehat{\otimes}
\mS_{N_{X^{g,K}/X}}\otimes E\right)_{y_0}$ over $T_{y_0}X$,
and $\dim B=k$.

As in \cite[(9.52)-(9.57)]{BG00}, combining with (\ref{eq:3.17}), 
if $n$ is even, 
there exists $\beta>0$,
if $z\in \R^*$, $|z|\leq \beta$, for $t\rightarrow 0$,
\begin{multline}\label{eq:6.66}
\int_{X^{g,K}}\int_{\substack{(Z_0,Z)\in N_{X^{g,K}/X^g}
		\times N_{X^g/X},
\\ |Z_0|,|Z|\leq \var_0}}\tr_s\left[g
\frac{c(K^X)}{4\sqrt{t}v}\wi{F}_t(\mB_{zK_0,t,tv})
(g^{-1}(y_0,Z_0,Z),\right.
\\
\left. (y_0, Z_0, Z))\right]dv_X(y_0, Z_0, Z)
\\
\rightarrow \int_{X^{g,K}}\int_{N_{X^{g,K}/X}}(-i)^{\ell'/2}
2^{\ell'/2}
\tr_s\left[g\frac{c\left(m_{y_0}^{TX}(K)Z\right)}{4v} 
\exp\left(-L_{y_0,zK_0}^{3,(0,v)} \right)(g^{-1}Z,Z)\right]
dv_{N_{X^{g,K}/X}}(Z).
\end{multline}
The heat kernel $\exp\left(-L_{y_0,zK_0}^{3,(0,v)} 
\right)(g^{-1}Z,Z)$ could be calculated 
as in (\ref{eq:5.135}) by \cite[Theorem 4.13]{BG00}, 
which is an even function on $Z$ and can be controlled by
$C\exp(-C'|Z|^2)$. So 
the right-hand side of (\ref{eq:6.66})
is an integral of an odd function on $Z$ over $N_{X^{g,K}/X}$,
which is zero.

If $n$ is odd, by Remark \ref{rem:5.22}, from the same argument
above, as $t\rightarrow 0$,
\begin{align}\label{eq:6.67} 
\int_{\mU_{\var_0}'}\tr^{\mathrm{even}}\left[g\frac{c(K^X)}{
	4\sqrt{t}v}\exp\left(-\mB_{K,t,tv}
\right) \right]\rightarrow 0.
\end{align}

After 
adopting the above technique to the case where $X^{g,K}
=\emptyset$, for $z\in \R^*$, $|z|$ small enough, as 
$t\rightarrow 0$, we have

\begin{align}\label{eq:6.68} 
\int_{X\backslash \mU_{\var_0}'}\wi{\tr}'\left[g\frac{c(K^X)}{
		4\sqrt{t}v}\exp\left(-\mB_{K,t,tv}
	\right) \right]\rightarrow 0.
\end{align}
	
The proof of Theorem \ref{thm:4.03} d) is completed.


\end{document}